\documentclass{article}
\usepackage{amssymb, amsmath, amscd, latexsym, enumerate}
\usepackage{eepic}
\usepackage[dvipdfmx]{hyperref}
\usepackage[dvipdfmx]{graphicx}
\usepackage[usenames]{color}
\title{Minimal surfaces with two ends \\ 
       which have the least total absolute curvature}
\author
 {Shoichi Fujimori 
    \thanks{Partially supported by JSPS Grant-in-Aid for 
        Young Scientists (B) 25800047.}
 \and 
  Toshihiro Shoda 
    \thanks{Partially supported by JSPS Grant-in-Aid for 
        Young Scientists (B) 24740047.
  }
 }
\date{September 21, 2015}
\usepackage{amsthm}
\theoremstyle{plain}
 \newtheorem{theorem}{Theorem}[section]
 \newtheorem{main-theorem}{Main Theorem}

 \newtheorem{lemma}[theorem]{Lemma}
 \newtheorem{corollary}[theorem]{Corollary}
\theoremstyle{definition}

 \newtheorem{remark}[theorem]{Remark}
 \newtheorem*{acknowledgements}{Acknowledgements}
 \newtheorem{example}[theorem]{Example}
 \newtheorem{problem}[theorem]{Problem}
\renewcommand{\Re}{\operatorname{Re}}

\numberwithin{equation}{section}
\numberwithin{figure}{section}
\numberwithin{table}{section}
\begin{document}
\maketitle

\begin{abstract}
In this paper, we consider complete non-catenoidal minimal surfaces 
of finite total curvature with two ends.  
A family of such minimal surfaces with 
least total absolute curvature is given.  
Moreover, we obtain a uniqueness theorem for this family from its symmetries. 
\end{abstract}

\section{Introduction} %%%%%%%%%%%%%%%%%%%%%%%%%%%%%%%%%%%%%%%%%%%%%%%%%
\label{sec:intro} %%%%%%%%%%%%%%%%%%%%%%%%%%%%%%%%%%%%%%%%%%%%%%%%%%%%%%

For a complete minimal surface in Euclidean space, 
an inequality stronger than the classical inequality of Cohn-Vossen holds, 
giving a lower bound for the total absolute curvature.  
It is then natural to ask whether there is a minimal surface 
which attains this minimum value for the total absolute curvature. 
We consider this problem and contribute to the 
theory of existence of minimal surfaces in Euclidean space. 
Our work connects with the Bj\"orling problem for minimal surfaces 
in Euclidean space. 

Let $f:M\to\mathbb{R}^3$ be a minimal immersion of a $2$-manifold $M$ 
into Euclidean 3-space $\mathbb{R}^3$, and we usually call 
$f$ a \textit{minimal surface} in $\mathbb{R}^3$. 
Choosing isothermal coordinates makes $M$ a Riemann surface, 
and then $f$ is called a \textit{conformal minimal immersion}. 
The following representation formula is one of the basic tools 
in the theory of minimal surfaces: 

\begin{theorem}[Weierstrass representation \cite{O}]\label{th:w-rep}
Let $(g,\,\eta )$ be a pair of a meromorphic function $g$ and a 
holomorphic differential $\eta$ on a Riemann surface $M$ so that 
\begin{equation}\label{eq:1stff}
(1+|g|^2)^2\eta \bar{\eta}
\end{equation}
gives a Riemannian metric on $M$. 
We set 
\begin{equation}\label{eq:Phi}
\Phi := \begin{pmatrix}
          (1-g^2)\eta \\ i(1+g^2)\eta\\ 2g\eta\
        \end{pmatrix},
\end{equation}
where $i=\sqrt{-1}$.  
Assume that 
\begin{equation}\label{eq:period}\tag{P}
\Re\oint_\ell \Phi = {\bf 0} \quad 
\text{holds for any}\;\; \ell\in\pi_1(M). 
\end{equation}
Then
\begin{equation}\label{eq:surf}
f=\Re\int_{z_0}^z \Phi:M\to\mathbb{R}^3\qquad (z_0\in M)
\end{equation}
defines a conformal minimal immersion. 
\end{theorem}
The pair $(g,\,\eta)$ in Theorem~\ref{th:w-rep} is called the {\it Weierstrass data} of $f$. 
\begin{remark}
Condition \eqref{eq:period} is called the \textit{period condition} 
of the minimal surface. 
\eqref{eq:period} is equivalent to 
\begin{equation}\label{eq:period1}
\oint_\ell \eta = \overline{\oint_\ell g^2\eta}
\end{equation}
and
\begin{equation}\label{eq:period2}
\Re\oint_\ell g\eta = 0
\end{equation}
for any $\ell\in\pi_1(M)$. 
\end{remark}

\begin{remark}
The first fundamental form $ds^2$ and the second fundamental form 
${\rm I}\!{\rm I}$ of the surface \eqref{eq:surf} are given by
\[
ds^2=\left( 1+|g|^2\right)^2\eta\bar\eta , \qquad
{\rm I}\!{\rm I}=-\eta dg-\overline{\eta dg}.
\]
Moreover, $g:M\to\mathbb{C}\cup\{\infty\}$ coincides with the composition of 
the Gauss map $G:M\to S^2$ of the minimal surface and 
stereographic projection 
$\sigma:S^2\to\mathbb{C}\cup\{\infty\}$, that is, 
$g=\sigma\circ G$.
So we call $g$ the Gauss map of the minimal surface. 
\end{remark}

Next, we assume that a minimal surface is complete and 
of finite total curvature. 
These two conditions give rise to restrictions on the topological and 
conformal types of minimal surfaces.  

\begin{theorem}[\cite{H, O}]\label{th:huber-osserman}
Let $f:M\to\mathbb{R}^3$ be a conformal minimal immersion.  
Suppose that $f$ is complete and of finite total curvature. 
Then, the following hold{\rm :}
\begin{enumerate}
\item $M$ is conformally equivalent to a compact Riemann surface 
      $\overline{M}_{\gamma}$ of genus $\gamma$ punctured at 
      a finite number of points $p_1,\dots ,p_n$. 
\item The Gauss map $g$ extends to a holomorphic mapping 
      $\hat{g}:\overline{M}_{\gamma}\to\mathbb{C}\cup\{\infty\}$. 
\end{enumerate}
Removed points $p_1,\dots ,p_n$ correspond to ends of the minimal surface. 
\end{theorem}
The asymptotic behavior around each end $p_i$ can be described by the order of the poles of 
$\Phi=(\Phi_1,\,\Phi_2,\,\Phi_3)$ in Theorem~\ref{th:w-rep} at $p_i$. 
Let 
\begin{equation}\label{eq:di}
d_i=\max_{1\leq j\leq 3} \{ {\rm ord} (\Phi_j,\,p_i) \}-1,
\end{equation}
where ${\rm ord} (\Phi_j,\,p_i)$ is the order of the pole of $\Phi_j$ at $p_i$ ($1\leq i\leq n$, $1\leq j\leq 3$). 
Condition \eqref{eq:period} yields ${\rm residue} (\Phi,\,p_i)\in \mathbb{R}^3$, and thus $d_i\geq 1$. 
The following theorem shows the geometric properties of $d_i$, 
which includes a stronger inequality than the Cohn-Vossen inequality: 
\begin{theorem}[\cite{O, JM, Sc}]\label{th:oss-ineq}
Let $f:M\to\mathbb{R}^3$ be a minimal surface as in Theorem~{\rm \ref{th:huber-osserman}}. \\
{\rm (a)} The immersion $f$ is proper. \\
{\rm (b)} If $S^2 (r)$ is the sphere of radius $r$, then 
$\frac{1}{r} (f(M) \cap S^2 (r))$ consists of $n$ closed curves $\Gamma_1,\,\cdots,\,\Gamma_n$ in  $S^2 (1)$ which 
converge $C^1$ to closed geodesics $\gamma_1,\,\cdots,\,\gamma_n$ of $S^2 (1)$, with multiplicities 
$d_1,\,\cdots, \,d_n$, as $r\to\infty$. Moreover, 
\begin{equation}\label{eq:oss-ineq1}
\frac{1}{2\pi}\int_{M}KdA = \chi (\overline{M}_{\gamma})-\sum_{i=1}^n (d_i+1)
\le \chi (M)-n
=   \chi (\overline{M}_\gamma)-2n
=   2 (1 - \gamma - n), 
\end{equation}
and equality holds if and only if each end is embedded. 
\end{theorem}

The equation on the left of \eqref{eq:oss-ineq1} 
is called the {\em Jorge-Meeks formula}.

Moreover, a relation between the total (absolute) curvature 
and the degree of $g$ is as follows. 
Note that since $g$ extends to a holomorphic map $\hat{g}$ 
from a compact Riemann surface $\overline{M}_{\gamma}$ 
to a compact Riemann surface $\mathbb{C}\cup\{\infty\}$, 
we can define the degree of $g$ by $\deg(g):=\deg(\hat{g})$. 
Since the Gaussian curvature of a minimal surface $M\to\mathbb{R}^3$ 
is always non-positive, 
its total absolute curvature $\tau (M):=\int_M|K|\,dA$ is given by 
\[
\tau (M)=\int_M(-K)\,dA.
\]
Recall that the total absolute curvature of a minimal surface in 
$\mathbb{R}^3$ is just the area under the Gauss map 
$g:M\to\mathbb{C}\cup\{\infty\}\cong S^2$, that is, 
\[
\tau (M)=(\text{\rm the area of $S^2$})\,\deg(g)
        =4\pi\,\deg(g)\in 4\pi\,\mathbb{Z}.
\]
(See, for example, (3.11) in \cite{HO} for details.)  
Hence \eqref{eq:oss-ineq1} is rewritten as 
\begin{equation}\label{eq:oss-ineq2}
\deg(g) \geq \gamma + n - 1 
\end{equation}
and we consider sharpness of the inequality \eqref{eq:oss-ineq2}. 

For $n\geq 3$, there exist many examples of minimal surfaces 
which satisfy $\deg(g)=\gamma +n-1$.  (See Figure~\ref{fg:ngeq3}.) 

\begin{figure}[htbp] %%%%%%%%%%%%%%%%%%%%%%%%%%%%%%%%%%%%%%%%%%%%%%%%%%%
\begin{center}
\begin{tabular}{cccc}
 \includegraphics[width=.22\linewidth]{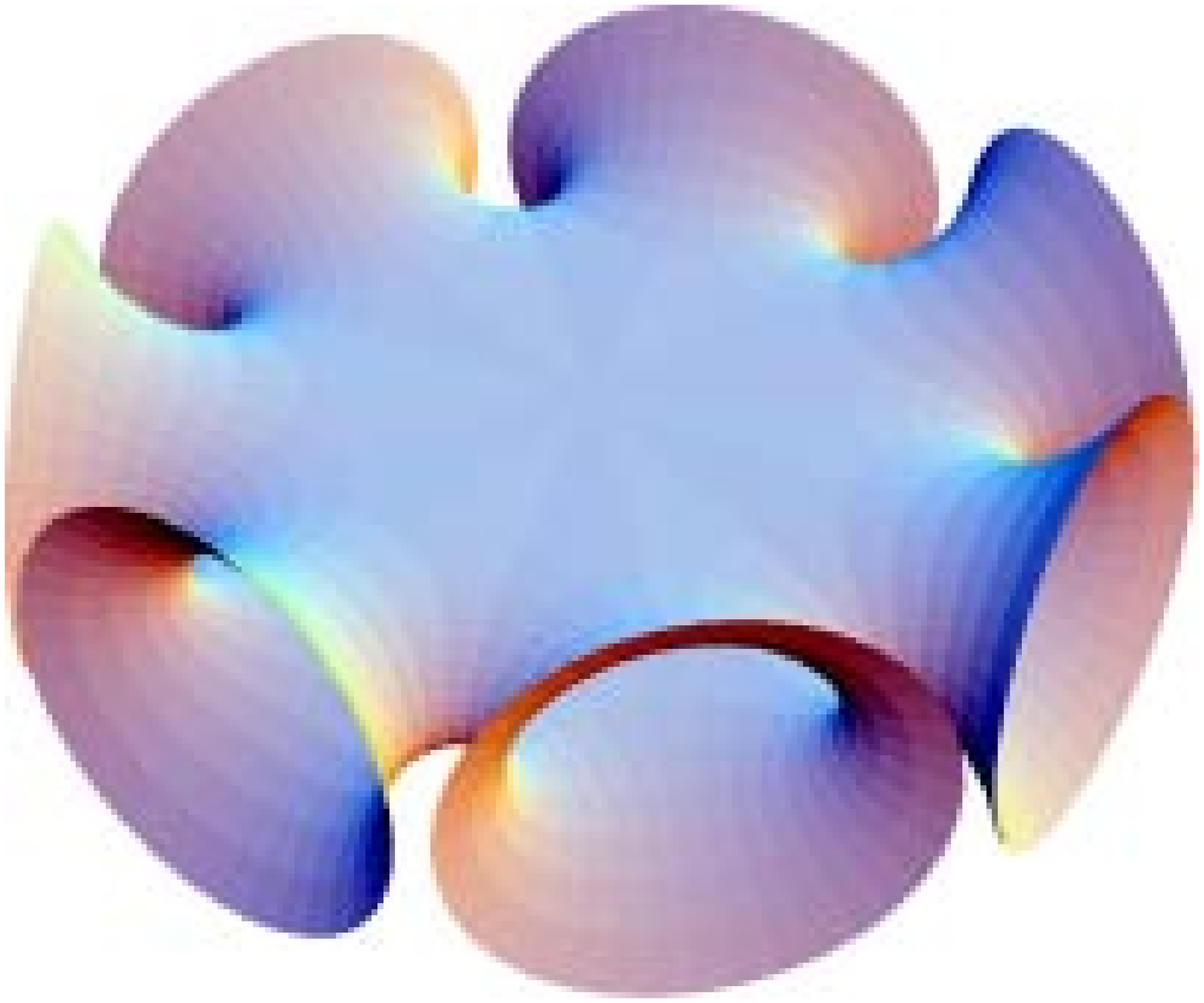} &
 \includegraphics[width=.16\linewidth]{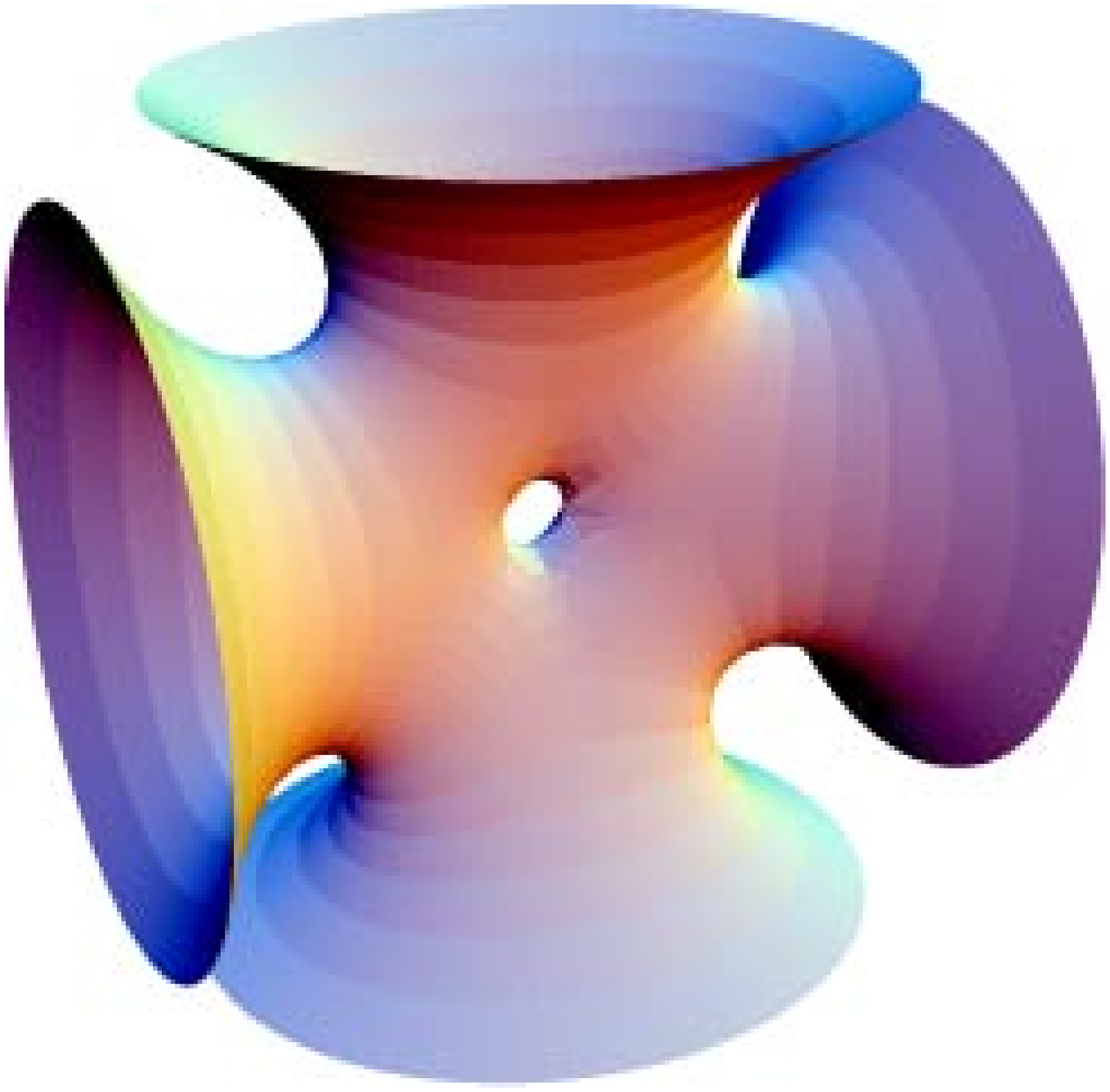} &
 \includegraphics[width=.23\linewidth]{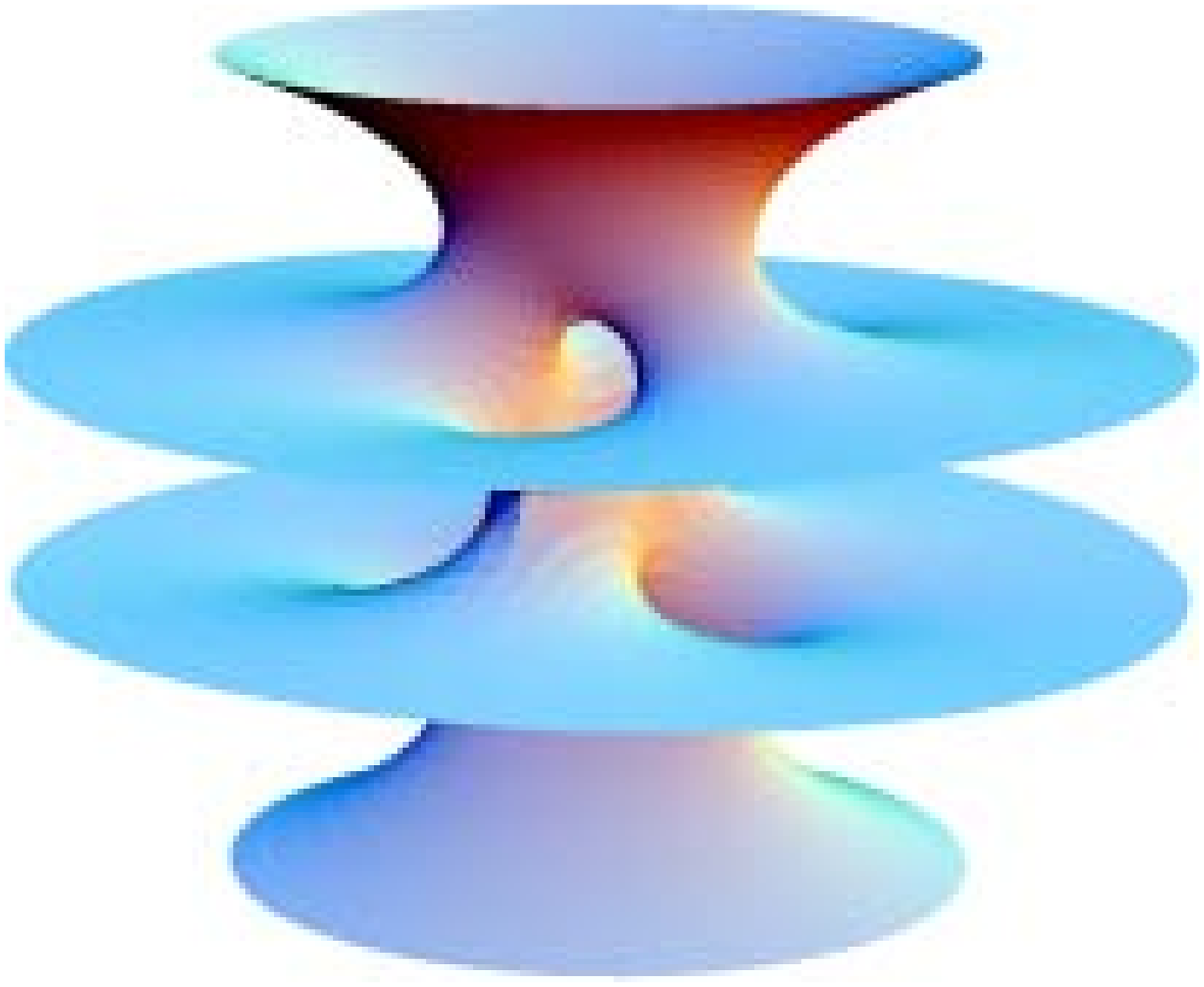} &
 \raisebox{10pt}{\includegraphics[width=.23\linewidth]{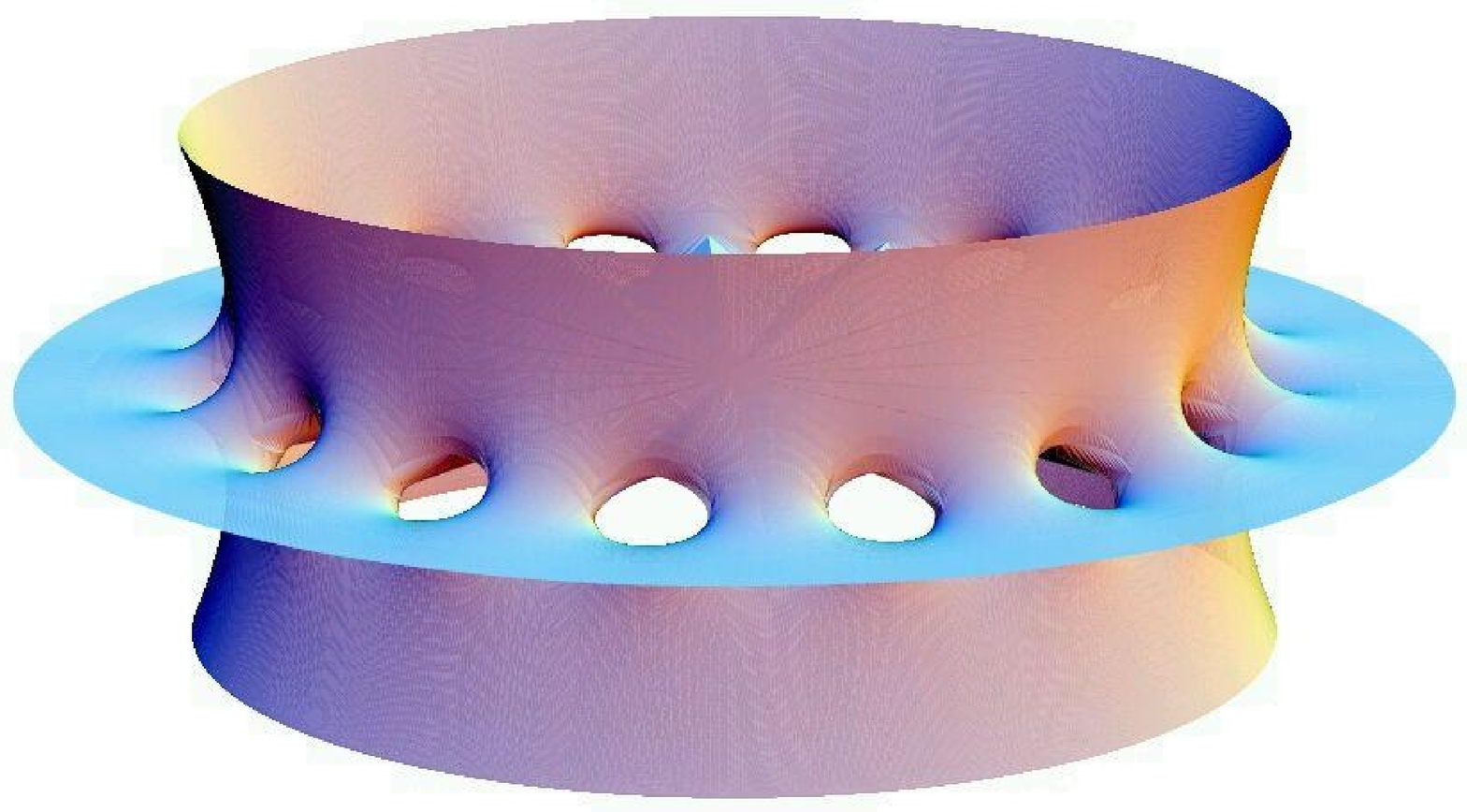}} \\
 $(\gamma ,n)=(0,7)$ & $(\gamma ,n)=(1,4)$ & 
 $(\gamma ,n)=(2,4)$ & $(\gamma ,n)=(14,3)$ \\
\end{tabular}
\caption{Minimal surfaces with $n\ge 3$ satisfying equality in 
         \eqref{eq:oss-ineq2}. 
         For details on these surfaces, see, for instance, 
         \cite{JM}, \cite{BR}, \cite{Wo}, \cite{HM}.}
\label{fg:ngeq3}
\end{center}
\end{figure} %%%%%%%%%%%%%%%%%%%%%%%%%%%%%%%%%%%%%%%%%%%%%%%%%%%%%%%%%%%

If $n = 1$, then a minimal surface satisfying 
$\deg(g) =\gamma$ must be a plane.  (See, for instance, \cite[Remark 2.2]{HK}.)  
Thus on a non-planar minimal surface, 
$\deg(g) \geq \gamma +1$. 
The existence of minimal surfaces with $\deg(g) = \gamma +1$ 
was shown by C.~C.~Chen and F.~Gackstatter \cite{CG} (for $\gamma=1,2$), 
N.~Do~Espirito-Santo \cite{ES} (for $\gamma=3$), 
K.~Sato \cite{Sa}, and M.~Weber and M.~Wolf \cite{WW}. 
(See Figure~\ref{fg:n=1}.)  

\begin{figure}[htbp] %%%%%%%%%%%%%%%%%%%%%%%%%%%%%%%%%%%%%%%%%%%%%%%%%%%
\begin{center}
\begin{tabular}{cccc}
 \includegraphics[width=.20\linewidth]{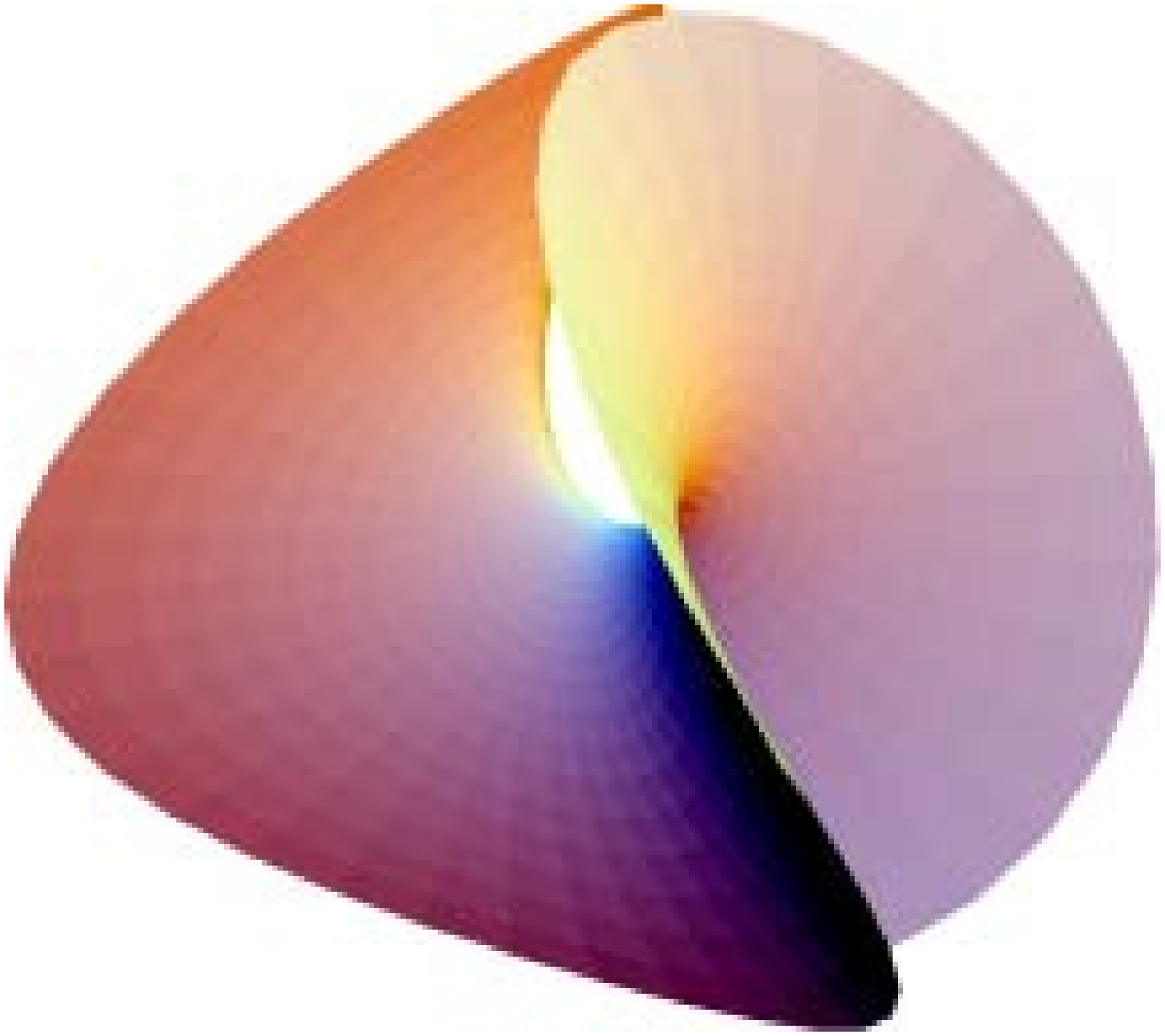} &
 \includegraphics[width=.20\linewidth]{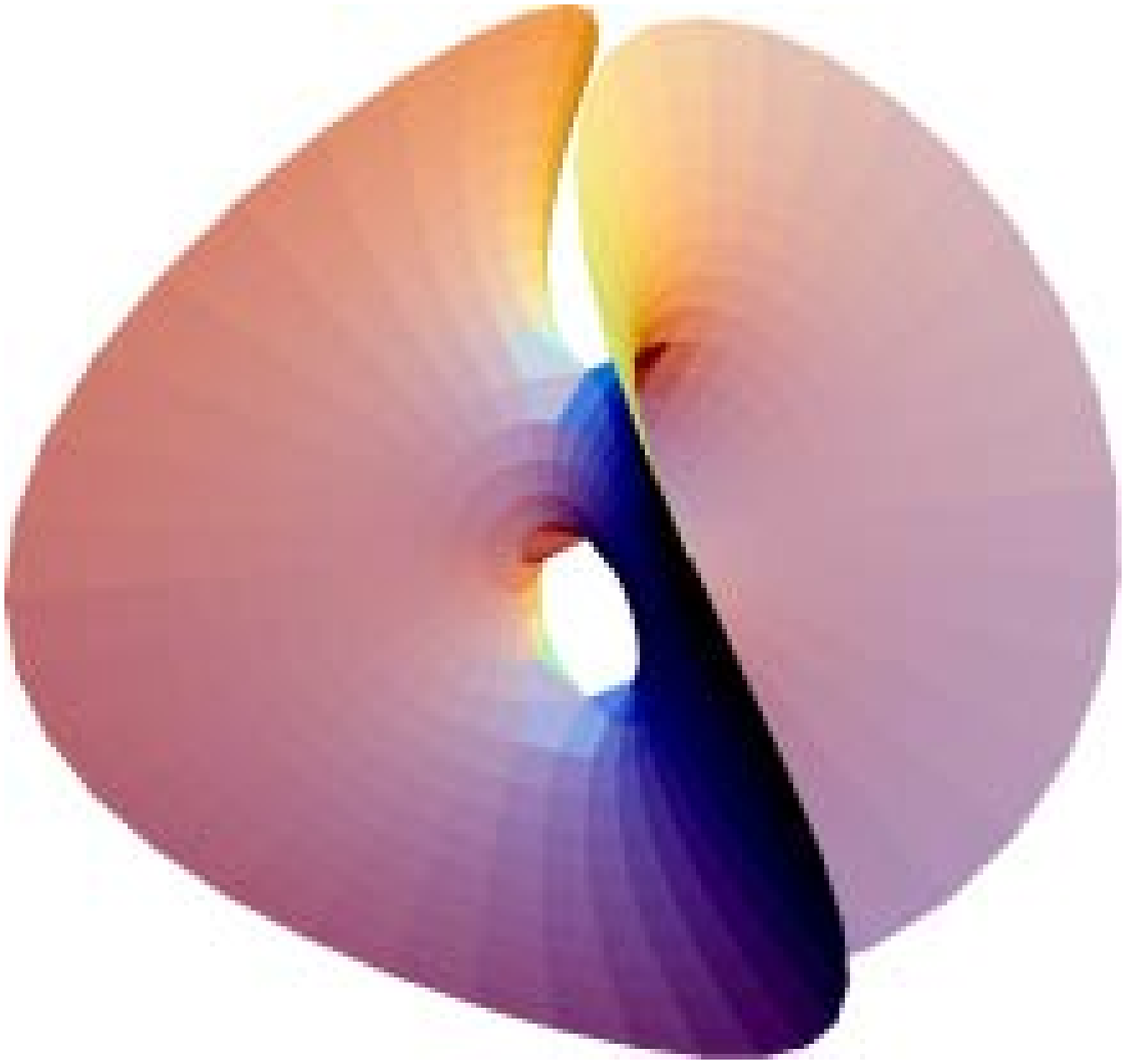} &
 \includegraphics[width=.20\linewidth]{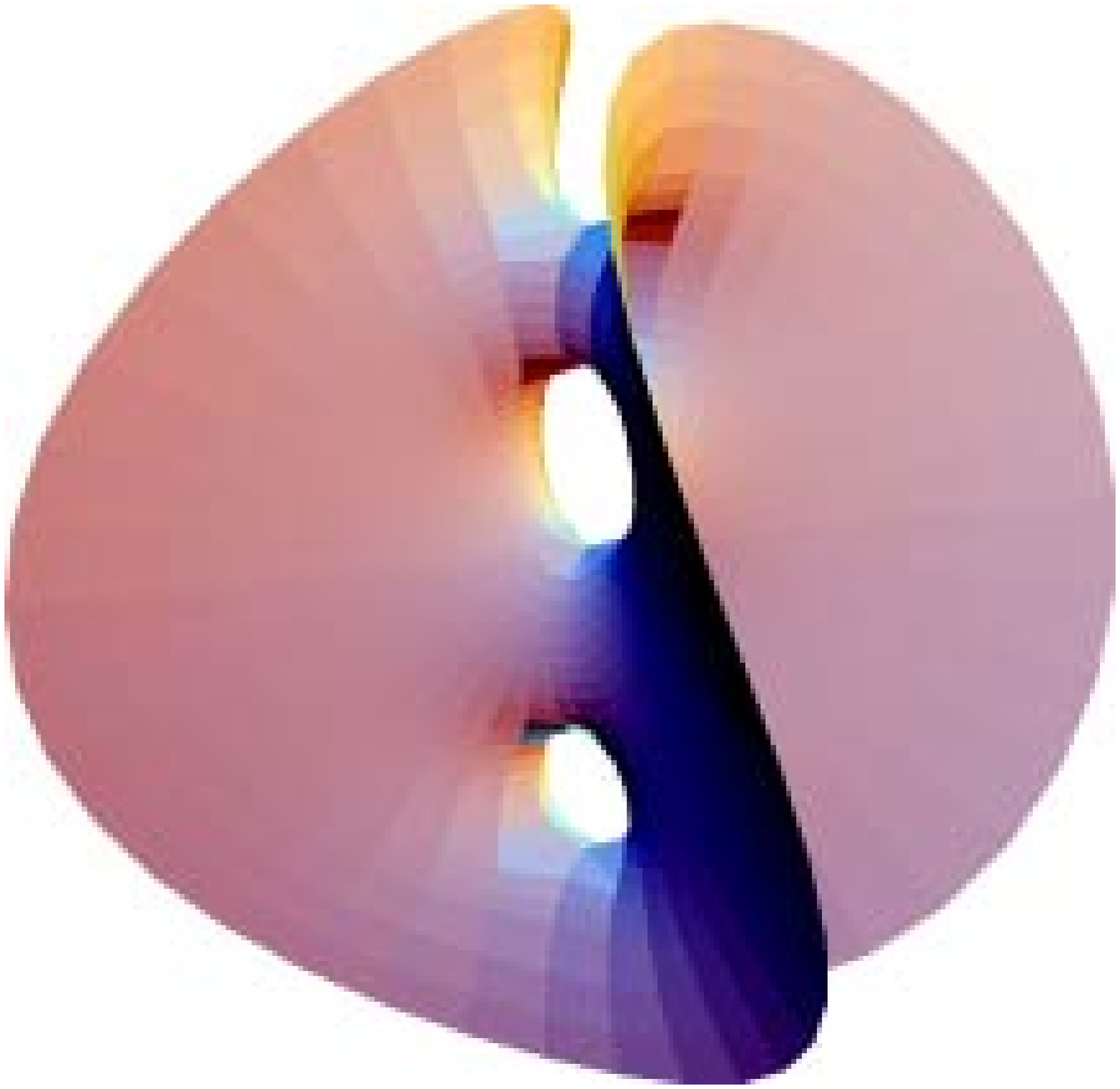} &
 \includegraphics[width=.20\linewidth]{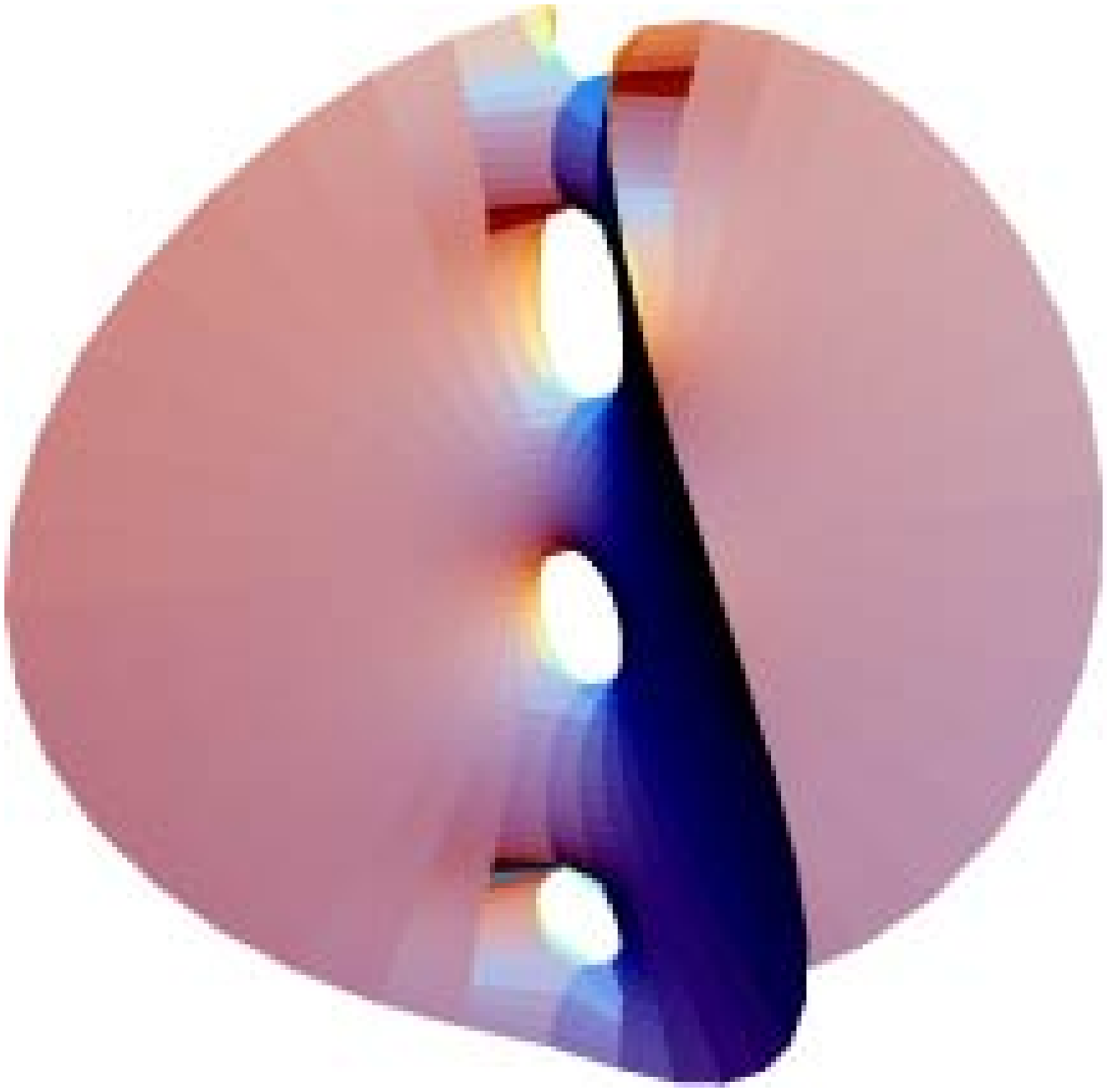} \\
 $\gamma =0$ & $\gamma =1$ & $\gamma =2$ & $\gamma =3$ \\
 Enneper & \multicolumn{2}{c}{Chen-Gackstatter} & Do Espirito-Santo
 %, Thayer-Sato
\end{tabular}
\caption{Minimal surfaces with $n=1$ satisfying $\deg(g) = \gamma +1$.}
\label{fg:n=1}
\end{center}
\end{figure} %%%%%%%%%%%%%%%%%%%%%%%%%%%%%%%%%%%%%%%%%%%%%%%%%%%%%%%%%%%

Finally, we consider the case $n = 2$. 
In this case, the following uniqueness result is known: 

\begin{theorem}[\cite{Sc}]
Let $f:M\to\mathbb{R}^3$ be a complete conformal minimal surface 
of finite total curvature.  
If $f$ has two ends and equality holds in \eqref{eq:oss-ineq2}, 
then $f$ must be a catenoid. 
\end{theorem}

It follows that on a non-catenoidal minimal surface with two ends, 
\begin{equation}\label{eq:ineq}
  \deg (g) \ge \gamma + 2 .
\end{equation}
As a consequence, it is reasonable to consider the following problem: 

\begin{problem}\label{prob:1}
For an arbitrary genus $\gamma$, does there exist 
a complete conformal minimal surface 
of finite total curvature with two ends which satisfies equality in 
\eqref{eq:ineq}?
\end{problem}

In the case $\gamma =0$ such minimal surfaces exist, and moreover, 
these minimal surfaces have been classified by 
F.~J.~L\'opez \cite{L}.  (See Figure~\ref{fg:genus0}.) 

\begin{figure}[htbp] %%%%%%%%%%%%%%%%%%%%%%%%%%%%%%%%%%%%%%%%%%%%%%%%%%%
\begin{center}
\begin{tabular}{ccc}
 \includegraphics[width=.30\linewidth]{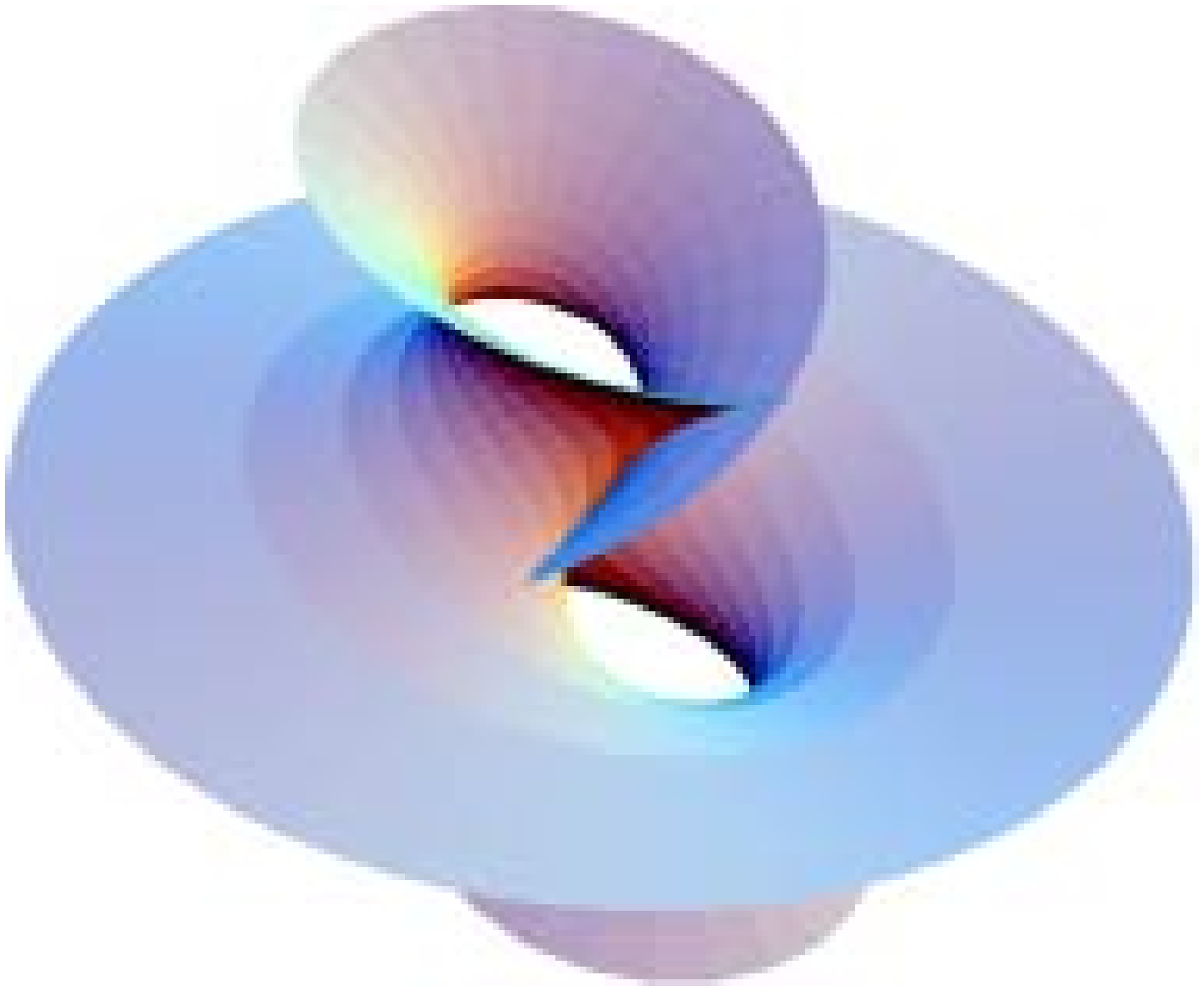} &
 \includegraphics[width=.29\linewidth]{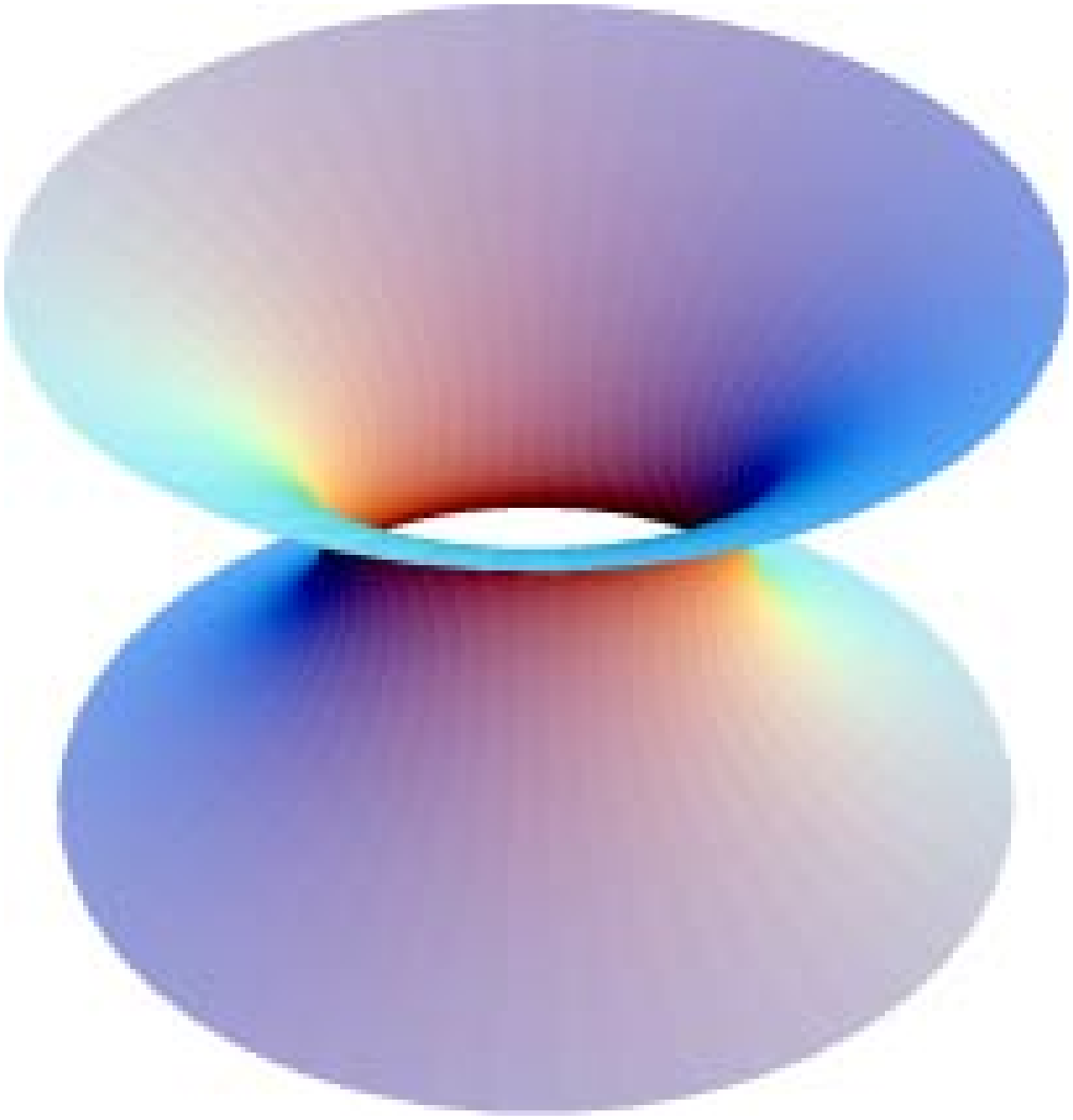} &
 \includegraphics[width=.30\linewidth]{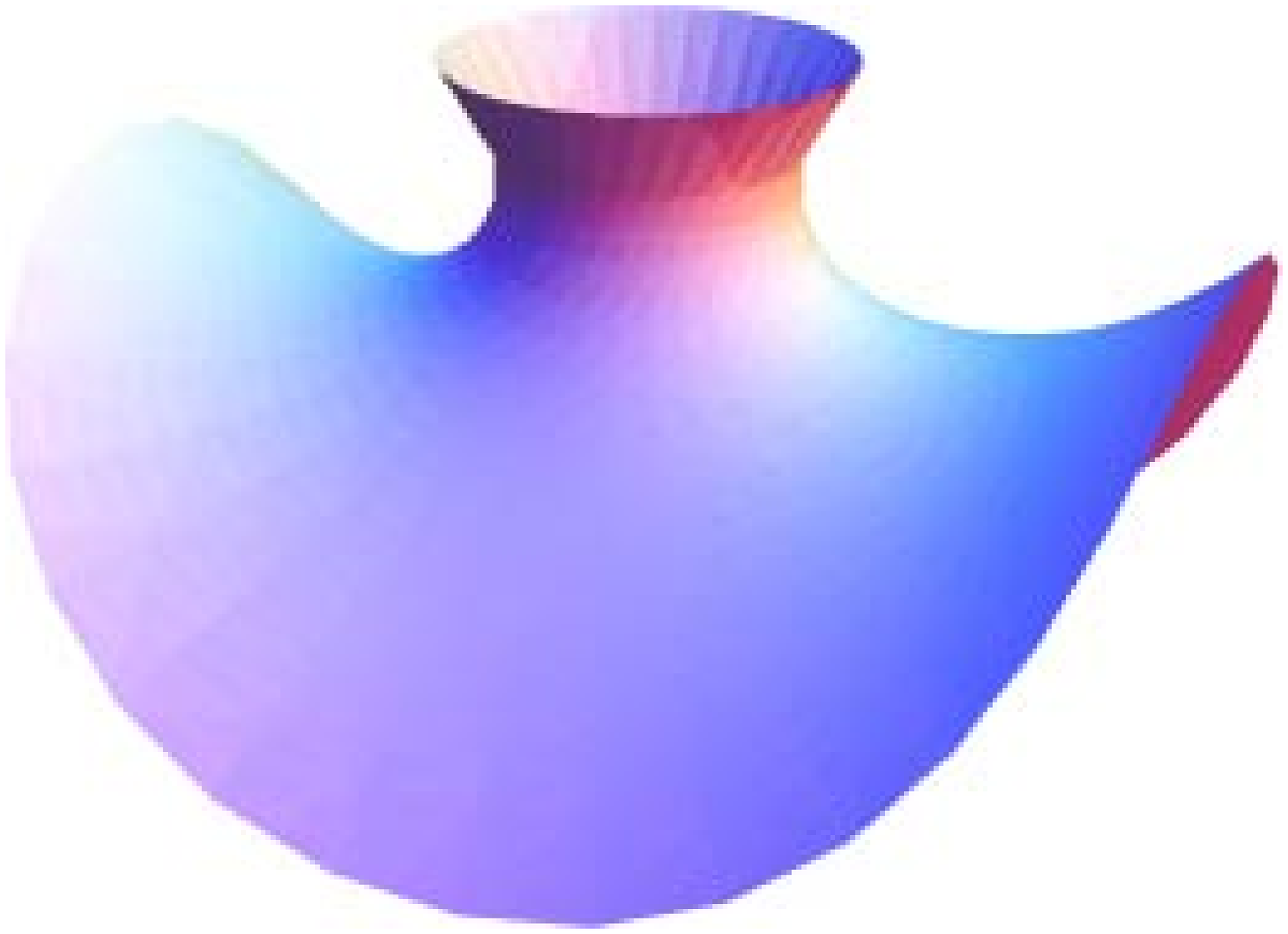}
\end{tabular}
\caption{Examples for $\gamma =0$.  
         The surface in the middle is a double cover of a catenoid.}
\label{fg:genus0}
\end{center}
\end{figure} %%%%%%%%%%%%%%%%%%%%%%%%%%%%%%%%%%%%%%%%%%%%%%%%%%%%%%%%%%%

However, for the case $\gamma >0$ no answer to 
Problem~\ref{prob:1} is known. 
Our first main result is to give a partial answer to this problem: 

\begin{main-theorem}\label{main1}
If $\gamma$ is equal to $1$ or an even number, 
there exists a complete conformal minimal 
surface of finite total curvature with two ends 
which satisfies equality in \eqref{eq:ineq}. 
\end{main-theorem}

Note that if we do not assume the equality in \eqref{eq:ineq}, 
then there exists a complete conformal minimal 
surface of finite total curvature with two ends for an arbitrary 
genus $\gamma\ge 0$.  (See, for instance, \cite{FS}.)  

We shall prove Main Theorem~\ref{main1} by explicit constructions. 
(See \S\ref{sec:examples}.) We now discuss 
the asymptotic behavior for our minimal surfaces in terms of $d_i$. 
For a minimal surface as in Problem~\ref{prob:1}, 
we have $(d_1,\,d_2)=(1,\,3),\,(2,\,2)$. 
The case $(d_1,\,d_2)=(1,\,3)$ corresponds to a minimal surface 
with an embedded end and an Enneper's type end. 
Recall that an embedded end is asymptotic to a plane or a catenoid. (See \cite{Sc}.) 
The minimal surface given in \S\ref{sec:genus1} has an embedded end 
which is asymptotic to a plane and $(d_1,\,d_2)=(1,\,3)$. 
(see Corollary~\ref{co:genus1}.) 
The minimal surface introduced in \S\ref{sec:oddgenus} is another example 
with an embedded end which is asymptotic to 
a half catenoid and $(d_1,\,d_2)=(1,\,3)$. 
The minimal surfaces with $(d_1,\,d_2)=(2,\,2)$ are obtained 
in \S\ref{sec:evengenus}. 
(See Corollary~\ref{co:evengenus}.) 

The minimal surfaces given in Corollary~\ref{co:genus1} and 
Corollary~\ref{co:evengenus} have symmetry groups 
with $4(\gamma+1)$ elements. 
Next we consider the uniqueness theorem for the symmetries. 
Uniqueness is also one of the important problems for minimal surfaces, 
and there are many uniqueness theorems. 
(See \cite{MaW, HM}.) 
Our other main theorem is as follows. 

\begin{main-theorem}\label{main2}
Let $f:M\to \mathbb{R}^3$ be a complete conformal minimal surface 
of finite total curvature with two ends and genus $\gamma$.  
Suppose that $f$ satisfies equality in \eqref{eq:ineq} 
and has $4(\gamma+1)$ symmetries.  
We assume either $\gamma =1$ and $(d_1,\,d_2)=(1,\,3)$, or 
$\gamma$ is an even number and $(d_1,\,d_2)=(2,\,2)$.  
Then $f$ is one of the minimal surfaces 
given in Main Theorem~$\ref{main1}$. 
\end{main-theorem}

At the end of this section, 
we discuss our work from the point of view of the Bj\"orling problem 
for minimal surfaces. 
The classical Bj\"orling problem is to determine 
a piece of a minimal surface containing a given analytic strip.  
This was named after E.~G.~Bj\"orling in 1844. 
%and O.~Bonnet had already discovered that the solution to Bj\"orling 
%problem also makes possible the determination of minimal surfaces 
%containing a given curve as a geodesic, as an asymptotic line, 
%or as a line of curvature in 1856. 
H.~A.~Schwarz gave an explicit solution to it. 
(See, for instance, \cite{N}.)  
Recently, Mira \cite{M} used the solution to the Bj\"orling problem 
to classify a certain class of minimal surfaces of genus $1$. 
Also, Meeks and Weber \cite{MW} produced an infinite sequence of complete 
minimal annuli by using the solution to the Bj\"orling problem and then gave 
a complete answer as to which curves appear as the singular set of a 
Colding-Minicozzi limit minimal lamination. 
Hence it is useful to study minimal surfaces from the point of view of 
the Bj\"orling problem.  
However, the existence of minimal surfaces of higher genus 
derived from the solution to the Bj\"orling problem seems to be unknown.  
In Section~\ref{sec:evengenus}, we show that our minimal surfaces, 
which have even numbers for the genus, are solutions to the Bj\"orling problem, 
and the generating curves are closed plane curves.  

The paper is organized as follows:  
Section~\ref{sec:examples} contains constructions of concrete examples 
to prove Main Theorem~\ref{main1} and is divided into two subsections. 
The genus $1$ case is provided in Section~\ref{sec:genus1} and even number 
cases are provided in Section~\ref{sec:evengenus}.  
Section~\ref{sec:evengenus} also 
contains the result from the point of view of the Bj\"orling problem. Moreover, we prove our uniqueness result in Section~\ref{uniqueness}. 
In Section~\ref{sec:remaining-problems} we refer to remaining problems 
related to our work.

\section{Construction of surfaces for Main Theorem~\ref{main1}} %%%%%%%%%
\label{sec:examples} %%%%%%%%%%%%%%%%%%%%%%%%%%%%%%%%%%%%%%%%%%%%%%%%%%%

In this section we will construct the surfaces for proving 
Main Theorem~\ref{main1} in the introduction. 
We will use the Weierstrass representation in Theorem~\ref{th:w-rep}, 
for which we need a Riemann surface $M$, 
a meromorphic function $g$, and a holomorphic differential $\eta$. 

\subsection{The case $\gamma =1$} %%%%%%%%%%%%%%%%%%%%%%%%%%%%%%%%%%%%%%
\label{sec:genus1} %%%%%%%%%%%%%%%%%%%%%%%%%%%%%%%%%%%%%%%%%%%%%%%%%%%%%

Let $\overline{M}_{\gamma}$ be the Riemann surface 
\[
  \overline{M}_\gamma
 =\left\{(z,w)\in (\mathbb{C}\cup\{\infty\})^2\,\left|\,
         w^{\gamma +1}=z(z^2-1)^\gamma\right.\right\}. 
\]
The surface we will consider is 
\[
M=\overline{M}_\gamma\setminus\left\{(0,0), (\infty, \infty)\right\}, 
\]
a Riemann surface of genus $\gamma $ from which two points have been removed. 
We want to define a complete conformal minimal immersion of $M$ into 
$\mathbb{R}^3$ by the Weierstrass representation in Theorem~\ref{th:w-rep}. 
To do this, set
\[
g=cw, \qquad  \eta = i \frac{dz}{z^2w},  
\]
where $c\in\mathbb{R}_{>0}$ is a positive constant to be determined. 

Let $\Phi$ be the $\mathbb{C}^3$-valued differential as in \eqref{eq:Phi}. 
We shall prove that \eqref{eq:surf} is a conformal minimal immersion of $M$. 

We begin with, we show by straightforward calculation how the following conformal 
diffeomorphisms $\kappa_1$ and $\kappa_2$ act on $\Phi$.  

\begin{lemma}[Symmetries of the surface]\label{lm:sym1}
Consider the following conformal mappings of $M${\rm :}
\[
\kappa_1(z,w)=(\bar z, \bar w), \qquad
\kappa_2(z,w)=(-z, e^{\pi i/(\gamma +1)}w).
\]
Then, 
\[
\kappa_1^*\Phi
=\begin{pmatrix}-1 & 0 & 0 \\
                 0 & 1 & 0 \\
                 0 & 0 &-1 \end{pmatrix}\overline{\Phi}, \qquad
\kappa_2^*\Phi
=\begin{pmatrix}-\cos\frac{\pi}{\gamma +1} & \sin\frac{\pi}{\gamma +1} & 0 \\
                -\sin\frac{\pi}{\gamma +1} &-\cos\frac{\pi}{\gamma +1} & 0 \\
                 0                         &                         0 &-1 
 \end{pmatrix}\Phi .
\]
\end{lemma}

Since \eqref{eq:1stff} gives a complete Riemannian metric on $M$ 
(see Table~\ref{tb:genus1}), 
it suffices to show that $f$ is well-defined on $M$ 
for the right choice of $c$. 

\begin{table}[htbp] %%%%%%%%%%%%%%%%%%%%%%%%%%%%%%%%%%%%%%%%%%%%%%%%%%%%
\begin{center}
\begin{tabular}{|c||c|c|c|c|}\hline 
$(z,w)$ & $(0,0)$ & $(1,0)$ & $(-1,0)$ & $(\infty,\infty)$ \\ \hline\hline
$g$     & $0^1$ & $0^\gamma$ & $0^\gamma$ & $\infty^{2\gamma +1}$ \\ \hline 
$\eta$  & $\infty^{\gamma +3}$ &       &       & $0^{3\gamma +1}$ \\ \hline 
\end{tabular}
\caption{Orders of zeros and poles of $g$ and $\eta$.}
\label{tb:genus1}
\end{center}
\end{table} %%%%%%%%%%%%%%%%%%%%%%%%%%%%%%%%%%%%%%%%%%%%%%%%%%%%%%%%%%%%

\begin{theorem}\label{existence1-thm} 
For any positive number $\gamma$, 
there exists a unique positive constant $c\in\mathbb{R}_{>0}$ for which the 
immersion $f$ given in \eqref{eq:surf} is well-defined on $M$.   
\end{theorem}

\begin{proof}
To establish this theorem we must show \eqref{eq:period} 
in Theorem~\ref{th:w-rep}.  
We will prove \eqref{eq:period1} and \eqref{eq:period2}, respectively. 
\eqref{eq:period2} follows from the exactness of $g\eta =icdz/z^2=d(-ic/z)$, 
and thus we will only have to show \eqref{eq:period1}. 
We first check the residues of $\eta$ and $g^2\eta$ at the ends 
$(0,0)$, $(\infty,\infty)$. 
At the end $(0,0)$, $w$ is a local coordinate for the Riemann surfce 
$\overline{M}_\gamma$, and then
$z=z(w)=w^{\gamma +1}\{(-1)^\gamma +\mathcal{O}(w^{2\gamma +2})\}$.  
We have
\[
\eta    = \left(\frac{\alpha_1}{w^{\gamma +3}}
                +\mathcal{O}(w^{\gamma -1})\right)dw
\qquad\text{and}\qquad
g^2\eta = \left(\frac{\alpha_2}{w^{\gamma +1}}
                +\mathcal{O}(w^{\gamma +1})\right)dw, 
\]
where $\alpha_j\in\mathbb{C}$ ($j=1,2$) are constants.  
These imply that both $\eta$ and $g^2\eta$ have no residues at $(0,0)$.  
Then the residue theorem yields that they have no residues at 
$(\infty,\infty)$ as well.  

We next consider path-integrals along topological $1$-cycles on 
$\overline{M}_\gamma$.  
We will give a convenient $1$-cycle. 

Define a $1$-cycle on $\overline{M}_\gamma$ as 
\begin{align*}
\ell
&=\left\{\left.(z,w)=\left(-t,\,
    \sqrt[\gamma+1]{-t(1-t^2)^\gamma}\,e^{\gamma\pi i/(\gamma +1)}\right)
   \,\right|\,-1\le t \le 0\right\} \\
&\qquad\qquad
  \cup 
  \left\{\left.(z,w)=\left(t,\,
    \sqrt[\gamma +1]{t(1-t^2)^\gamma}\,e^{-\gamma\pi i/(\gamma +1)}\right)
   \,\right|\,0\le t \le 1\right\}. 
\end{align*}
Recall that $(0,0)$ corresponds to the end of $f$.  
Avoiding the end $(0,0)$, we can deform $\ell$ to a $1$-cycle 
$\ell'$ on $M$ which is projected to a loop winding once around 
$[0,1]$ in the $z$-plane.  (See Figure~\ref{fg:loop1}.)  

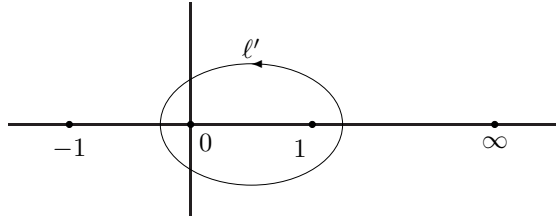
\begin{figure}[htbp] %%%%%%%%%%%%%%%%%%%%%%%%%%%%%%%%%%%%%%%%%%%%%%%%%%%
\begin{center}
\unitlength=1.15pt
\begin{picture}(300,70)
\put(60,30){\line(1,0){180}}%real axis
\thicklines
\put(120,0){\line(0,1){70}}%imaginary axis
\put(60,30){\line(1,0){20}}%real axis
\put(120,30){\line(1,0){40}}%real axis
\put(220,30){\line(1,0){20}}%real axis
\thinlines
\put(80,30){\circle*{2}}%point at -1
\put(120,30){\circle*{2}}%point at 0
\put(160,30){\circle*{2}}%point at 1
\put(220,30){\circle*{2}}%point at infinity
\put(80,22){\makebox(0,0)[cc]{$-1$}}
\put(125,24){\makebox(0,0)[cc]{$0$}}
\put(156,22){\makebox(0,0)[cc]{$1$}}
\put(220,24){\makebox(0,0)[cc]{$\infty$}}
\put(140,55){\makebox(0,0)[cc]{$\ell'$}}
\put(140,50){\vector(-1,0){0}}
\put(140,30){\ellipse{60}{40}}%ell
\end{picture}
\caption{Projection to the $z$-plane of the loop $\ell'\in\pi_1(M)$.}
\label{fg:loop1}
\end{center}
\end{figure} %%%%%%%%%%%%%%%%%%%%%%%%%%%%%%%%%%%%%%%%%%%%%%%%%%%%%%%%%%%

By the actions of the $\kappa _j$'s, we can obtain all of the $1$-cycles on $M$ 
from $\ell'$. 
If \eqref{eq:period} holds for this $\ell'$, then 
\[
 \Re\int_{\kappa_j\circ\ell'}\Phi 
=\Re\int_{\ell'}\kappa_j^*\Phi 
=K\Re\int_{\ell'}\Phi
={\bf 0}
\]
for some orthogonal matrix $K$, by Lemma~\ref{lm:sym1}. 
Hence all that remains to be done is to show that \eqref{eq:period1} holds 
for $\ell'$.

We now calculate path-integrals of $\eta $ and $g^2\eta $ along $\ell'$, 
and we want to reduce them to path-integrals along $\ell$ for simplicity. 
Note that both $\eta$ and $g^2\eta$ have poles at $(0,0)$.  
To avoid divergent integrals, 
here we add exact $1$-forms 
which have principal parts of $\eta$ and $g^2\eta$, respectively. 
It is straightforward to check
\[
 \dfrac{dz}{z^2w}-\dfrac{\gamma +1}{\gamma +2}d\left(\dfrac{z^2-1}{zw}\right)
 =\dfrac{\gamma}{\gamma +2}\dfrac{dz}{w}, \qquad
 \dfrac{w}{z^2}dz+\dfrac{\gamma +1}{\gamma}d\left(\dfrac{w}{z}\right)
 =\dfrac{2w}{z^2-1}dz. 
\]
So we have 
\begin{align*}
\oint_{\ell'}\eta 
&= \dfrac{i\gamma}{\gamma +2}\oint_{\ell'}\dfrac{dz}{w}
 = \dfrac{i\gamma}{\gamma +2}\oint_\ell\dfrac{dz}{w}
 = \dfrac{-2\gamma}{\gamma +2}\sin\dfrac{\gamma\pi}{\gamma +1}
   \int_0^1\dfrac{dt}{\sqrt[\gamma +1]{t(1-t^2)^\gamma}}, \\
\oint_{\ell'}g^2\eta 
&= 2ic^2\oint_{\ell'}\dfrac{w}{z^2-1}dz 
 = 2ic^2\oint_\ell\dfrac{w}{z^2-1}dz
 =-4c^2\sin\dfrac{\gamma\pi}{\gamma +1}
   \int_0^1\sqrt[\gamma +1]{\dfrac{t}{1-t^2}}dt. 
\end{align*}
By setting
\[
A_\gamma = \dfrac{\gamma}{\gamma +2}
           \int_0^1\dfrac{dt}{\sqrt[\gamma +1]{t(1-t^2)^\gamma}}
           \in\mathbb{R}_{>0},
\qquad
B_\gamma = 2\int_0^1\sqrt[\gamma +1]{\dfrac{t}{1-t^2}}dt
           \in\mathbb{R}_{>0},
\]
\eqref{eq:period1} is reduced to $A_\gamma = c^2 B_\gamma$.
Let us set 
\[
c=\sqrt{\frac{A_\gamma}{B_\gamma}}\in\mathbb{R}_{>0}. 
\]
This choice of $c$ satisfies \eqref{eq:period1} and is the unique positive 
real number that does so.  This completes the proof.  
(See Figure~\ref{fg:pe1}.) 
\end{proof}

\begin{figure}[htbp] %%%%%%%%%%%%%%%%%%%%%%%%%%%%%%%%%%%%%%%%%%%%%%%%%%%
\begin{center}
\begin{tabular}{cccc}
\includegraphics[width=.20\linewidth]{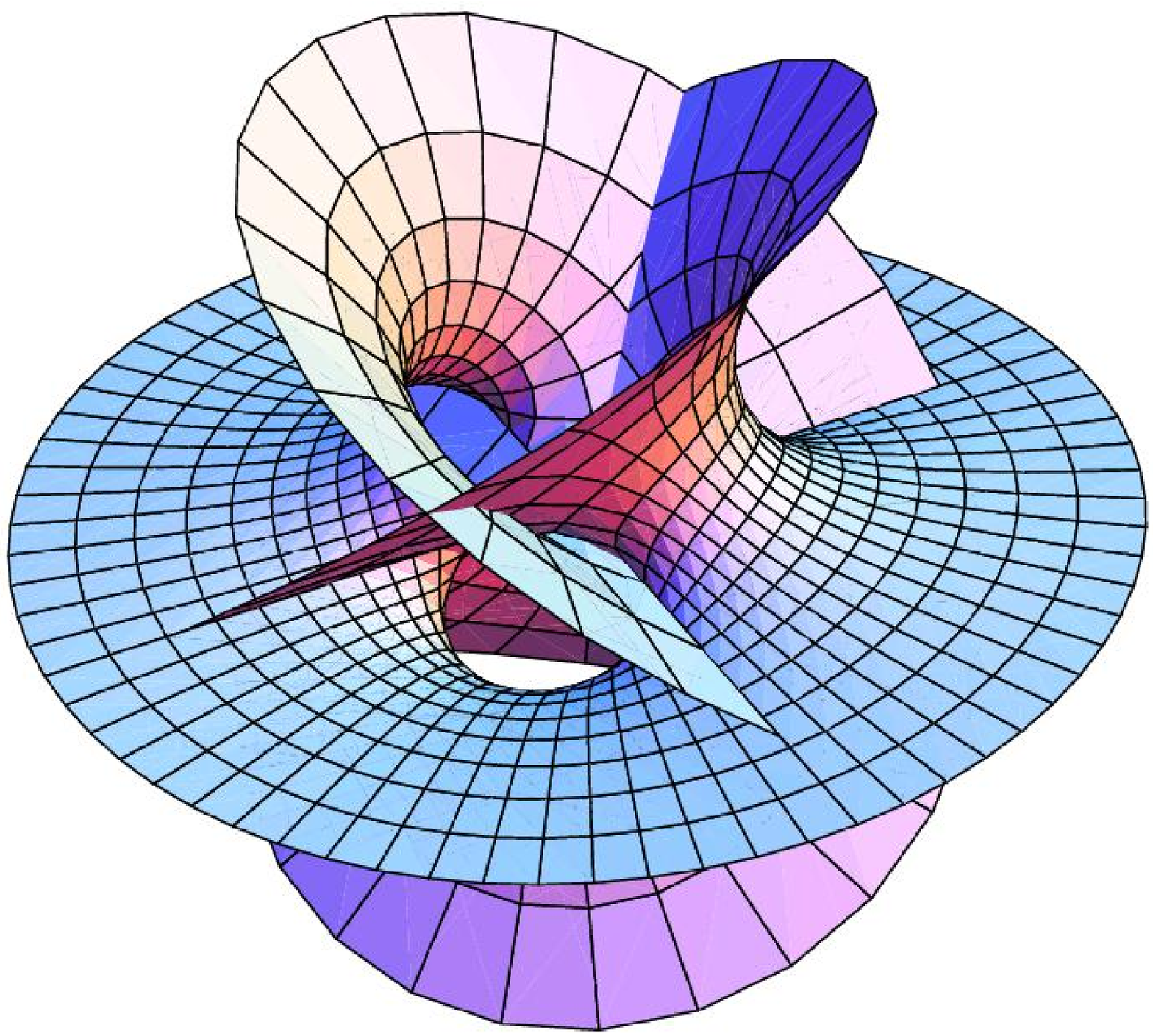} & 
\includegraphics[width=.20\linewidth]{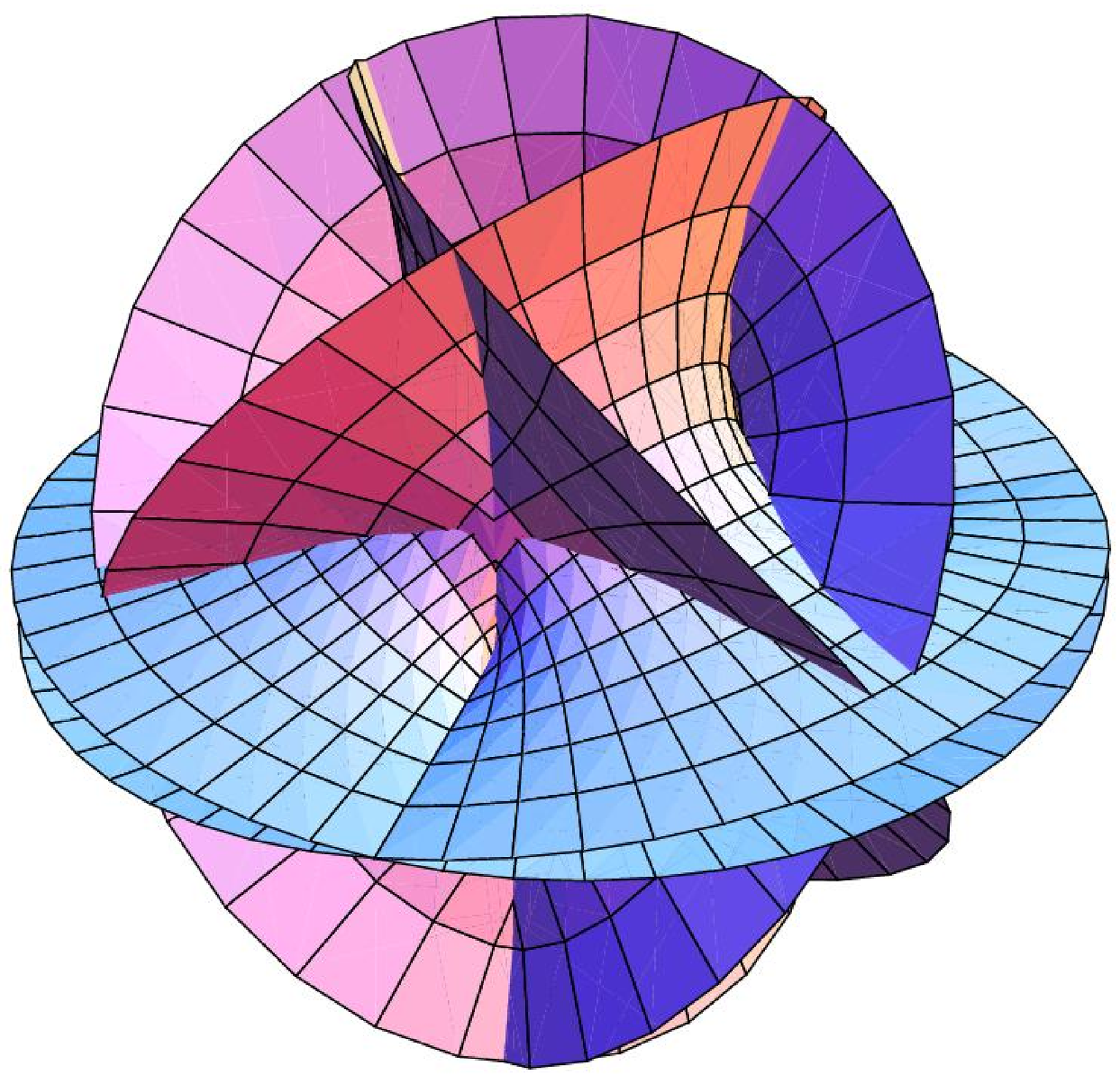} & 
\includegraphics[width=.20\linewidth]{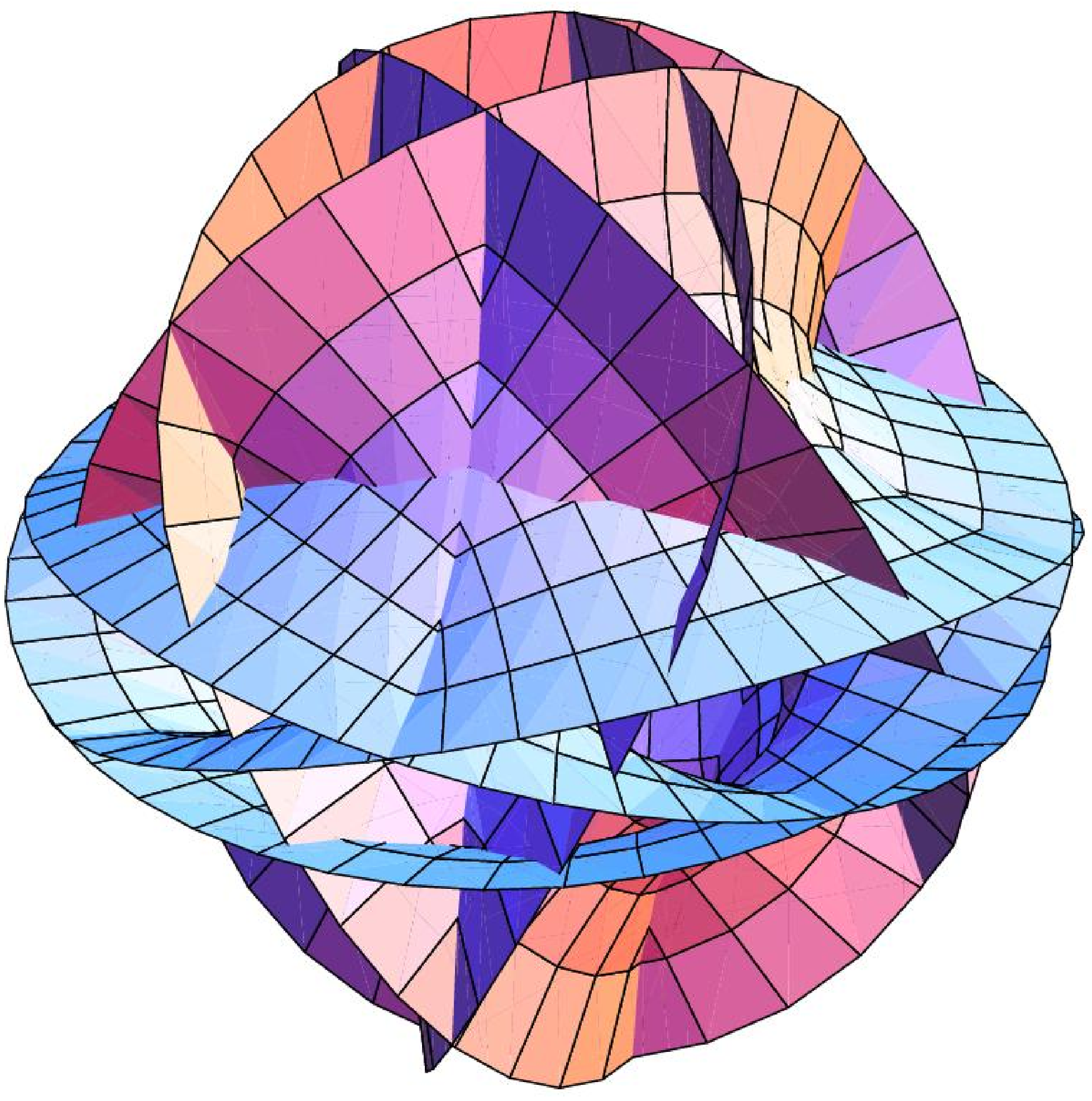} & 
\includegraphics[width=.20\linewidth]{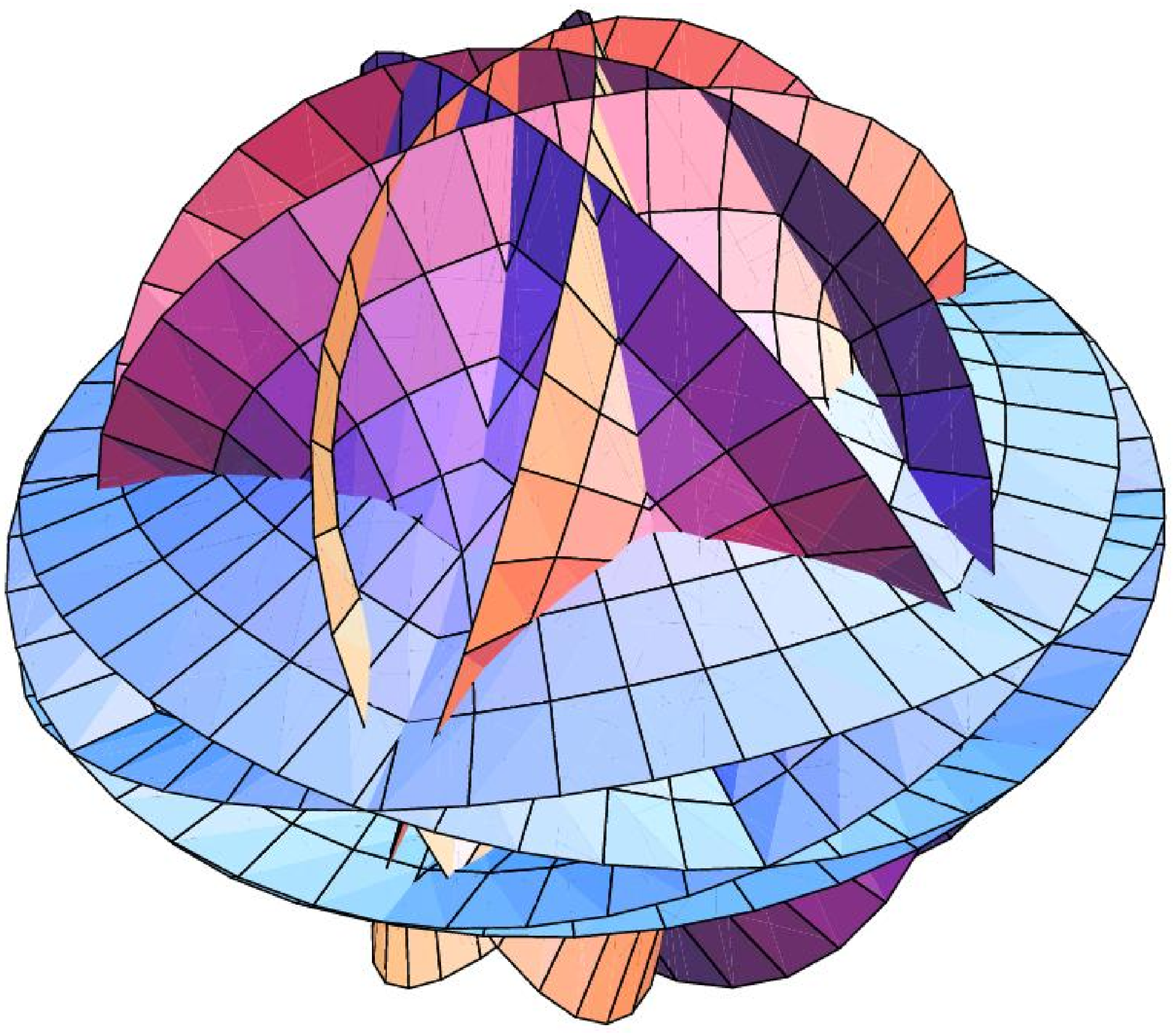} \\
$\gamma =1$ & $\gamma =2$ & $\gamma =3$ & $\gamma =4$
\end{tabular}
\caption{Minimal surfaces of genus $\gamma$ with two ends 
         which satisfy $\deg (g)=2\gamma +1$.}
\label{fg:pe1}
\end{center}
\end{figure} %%%%%%%%%%%%%%%%%%%%%%%%%%%%%%%%%%%%%%%%%%%%%%%%%%%%%%%%%%%

Since $\deg(g)=2\gamma +1$, 
$\deg(g)=\gamma +2$ if and only if $\gamma =1$. 
As a consequence, the next corollary follows:

\begin{corollary}\label{co:genus1}
There exists a complete conformal minimal surface of genus $1$ 
with two ends which has least total absolute curvature. 
\end{corollary}

\subsection{The case $\gamma$ is even} %%%%%%%%%%%%%%%%%%%%%%%%%%%%%%%%%
\label{sec:evengenus} %%%%%%%%%%%%%%%%%%%%%%%%%%%%%%%%%%%%%%%%%%%%%%%%%%

The following construction is similar to the construction in 
Section~\ref{sec:genus1}.  
Crucial arguments are given after \eqref{eq:period-even}.  

For an integer $k\geq 2$, 
let $\overline{M}_\gamma$ be the Riemann surface 
\[
\overline{M}_\gamma
 =\left\{(z,w)\in (\mathbb{C}\cup\{\infty\})^2\,\left|\,
         w^{k+1}=z^2\left(\dfrac{z-1}{z-a}\right)^k\right.\right\},
\]
where $a\in (1,\infty)$ is a constant to be determined.  
By the Riemann-Hurwitz formula, we see that 
\[
\gamma =\left\{\begin{array}{ll}
               k   & (\text{if $k$ is even}), \\
               k-1 & (\text{if $k$ is odd}).  \end{array}\right. 
\]
Note that the genus $\gamma$ is always even. 
(See Figure~\ref{fg:riem-surf}.)  

\begin{figure}[htbp] %%%%%%%%%%%%%%%%%%%%%%%%%%%%%%%%%%%%%%%%%%%%%%%%%%%
\begin{center}
\includegraphics[width=1.0\linewidth]{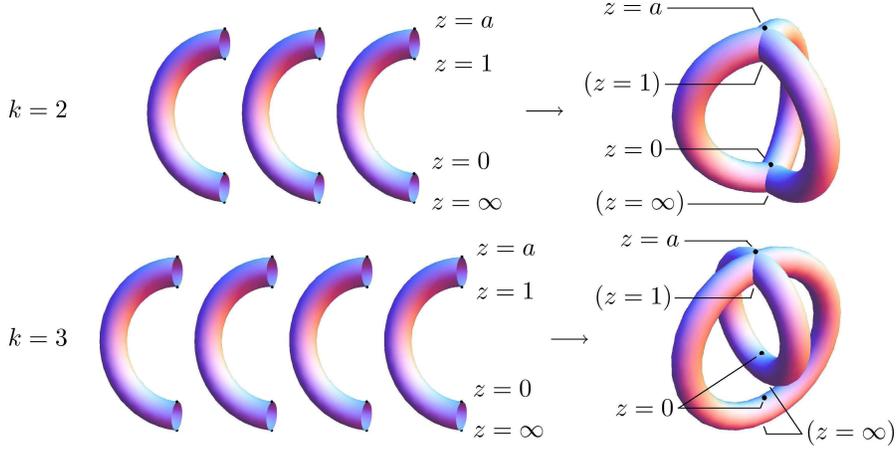}
\caption{Riemann surfaces $\overline{M}_\gamma$,  
         $k=2$ (top) and $k=3$ (bottom). 
         Both surfaces have genus 2. 
         When $k$ is odd, $\overline{M}_\gamma$ has self-intersections 
         near $z=0$ and $z=\infty$.  
         In the sketch in the bottom row, we see two different $z=0$ 
         (and two different $z=\infty$, which are hidden 
         from this viewpoint) but they are in fact the same points. 
         The reason we place these points differently is to reveal 
         their genus clearly.}
\label{fg:riem-surf}
\end{center}
\end{figure} %%%%%%%%%%%%%%%%%%%%%%%%%%%%%%%%%%%%%%%%%%%%%%%%%%%%%%%%%%%

We set 
\[
M=\overline{M}_\gamma\setminus\left\{(0,0), (\infty, \infty)\right\}, \quad
g=cw \;\; \left(c=a^{(k-2)/(2k+2)}\in\mathbb{R}_{>0}\right), \quad 
\eta = \frac{dz}{zw}.  
\]
Then \eqref{eq:1stff} gives a complete Riemannian metric $M$.  
(See Table~\ref{tb:genus-e}.) 

\begin{table}[htbp] %%%%%%%%%%%%%%%%%%%%%%%%%%%%%%%%%%%%%%%%%%%%%%%%%%%%
\begin{center}
\begin{tabular}{|c||c|c|c|c|}\hline 
$(z,w)$ & $(0,0)$ & $(1,0)$ & $(a,\infty)$ & $(\infty,\infty)$ \\ \hline\hline
$g$     & $0^2$      & $0^k$ & $\infty^k$ & $\infty^2$ \\ \hline 
$\eta$  & $\infty^4$ &       & $0^{2k}$   &       \\ \hline 
\end{tabular}
\\ \ \\ \ \\
\begin{tabular}{|c||c|c|c|c|}\hline 
$(z,w)$ & $(0,0)$ & $(1,0)$ & $(a,\infty)$ & $(\infty,\infty)$ \\ \hline\hline
$g$    & $0^2$      & $0^k$ & $\infty^k$ & $\infty^2$ \\ \hline 
$\eta$ & $\infty^3$ &       & $0^{2k}$   & $0^1$      \\ \hline 
\end{tabular}
\caption{Orders of zeros and poles of $g$ and $\eta$ when $k$ is odd (top) 
         and $k$ is even (bottom).}
\label{tb:genus-e}
\end{center}
\end{table} %%%%%%%%%%%%%%%%%%%%%%%%%%%%%%%%%%%%%%%%%%%%%%%%%%%%%%%%%%%%

Also, $\deg (g)=k+2$ for all $k\ge 2$.  
Thus equality in \eqref{eq:ineq} holds if and only if $k$ is even.  
Hereafter we assume $k$ is even. 

Let $\Phi$ be the $\mathbb{C}^3$-valued differential as in \eqref{eq:Phi}. 
We shall prove that \eqref{eq:surf} is a conformal minimal immersion of $M$. 

First, we observe the following symmetries $\kappa_1$, $\kappa_2$, 
and $\kappa_3$ of the surface.  

\begin{lemma}[Symmetries of the surface]\label{lm:sym2}
Consider the following conformal mappings of $M${\rm :}
\begin{eqnarray*}
&
\kappa_1(z,w)=\left(\bar z,\bar w\right), \qquad 
\kappa_2(z,w)=\left(z,e^{2\pi i/(k+1)} w\right), 
& \\ & 
\kappa_3(z,w)=\left(\dfrac{a}{z},
 \dfrac{1}{c^2w}\right). 
&
\end{eqnarray*}
Then, 
\begin{eqnarray*}
&
\kappa_1^*\Phi
=\begin{pmatrix} 1 & 0 & 0 \\
                 0 &-1 & 0 \\
                 0 & 0 & 1 \end{pmatrix}\overline{\Phi}, \qquad
\kappa_2^*\Phi
=\begin{pmatrix} \cos\frac{2\pi}{k+1} &-\sin\frac{2\pi}{k+1} & 0 \\
                 \sin\frac{2\pi}{k+1} &\phantom{-}\cos\frac{2\pi}{k+1} & 0 \\
                 0                    & 0                    & 1 
 \end{pmatrix}\Phi , 
& \\ &
\kappa_3^*\Phi= \begin{pmatrix} 1 & 0 & 0 \\
                                0 &-1 & 0 \\
                                0 & 0 &-1 \end{pmatrix}\Phi .
&
\end{eqnarray*}
\end{lemma}

As we have already seen the completeness of $f$, 
it suffices to show that $f$ is well-defined on $M$ 
for the right choice of $a\in (1,\infty)$.  

\begin{theorem} 
For any positive even number $k$, 
there exists a unique constant $a\in (1,\infty)$ for which the 
immersion $f$ given in \eqref{eq:surf} is well-defined on $M$.   
\end{theorem}

\begin{proof}
We will show \eqref{eq:period} in Theorem~\ref{th:w-rep}.  
It is easy to verify that there are no residues at the ends 
$(0,0)$, $(\infty,\infty)$. 
So all that remains is to choose $c$ so that \eqref{eq:period} is satisfied.  
\eqref{eq:period2} follows from the exactness of 
$g\eta=(c/z)dz=c\cdot d(\log z)$ and $c\in\mathbb{R}$, 
and hence we will only have to show \eqref{eq:period1}. 
To do this, we will give convenient $1$-cycles. 

Define a $1$-cycle on $\overline{M}_\gamma$ as 
\begin{align*}
\ell_1 &= \left\{\left. (z,w)=\left(-t,\,\sqrt[k+1]{t^2
                              \left(\frac{1+t}{a+t}\right)^k}\right)
          \,\right|\,-1\le t\le 0\right\} \\
       &  \qquad\qquad \cup 
          \left\{\left. (z,w)=\left(t,\,\sqrt[k+1]{t^2
                              \left(\frac{1-t}{a-t}\right)^k}
                              \,e^{2\pi i/(k+1)}\right)
          \,\right|\,0\le t\le 1\right\}.  
\end{align*}
Recall that $(0,0)$ corresponds to the end of $f$.  
Avoiding the end $(0,0)$, we can deform $\ell_1$ to a $1$-cycle 
$\ell_1'$ on $M$ which is projected to a loop winding once around 
$[0,1]$ in the $z$-plane.  
We also define another $1$-cycle on $M$ as 
\begin{align*}
\ell_2 &= \left\{\left. (z,w)=\left(-t,\,\sqrt[k+1]{t^2
                              \left(\frac{-t-1}{a+t}\right)^k}
                              \,e^{k\pi i/(k+1)}\right)
          \,\right|\,-a\le t\le -1\right\} \\
       &  \qquad\qquad \cup 
          \left\{\left. (z,w)=\left(t,\,\sqrt[k+1]{t^2
                              \left(\frac{t-1}{a-t}\right)^k}
                              \,e^{-k\pi i/(k+1)}\right)
          \,\right|\,1\le t\le a\right\}. 
\end{align*}
(See Figure~\ref{fg:loop2}.)  

\begin{figure}[htbp] %%%%%%%%%%%%%%%%%%%%%%%%%%%%%%%%%%%%%%%%%%%%%%%%%%%
\begin{center}
\unitlength=1.15pt
\begin{picture}(300,70)
\put(30,35){\line(1,0){240}}%real axis
\thicklines
%\put(150,0){\line(0,1){100}}%imaginary axis
\put(70,35){\line(1,0){40}}%real axis
\put(190,35){\line(1,0){40}}%real axis
\thinlines
\put(70,35){\circle*{2}}%point at 0
\put(110,35){\circle*{2}}%point at 1
\put(190,35){\circle*{2}}%point at a
\put(230,35){\circle*{2}}%point at infinity
\put(70,27){\makebox(0,0)[cc]{$0$}}
\put(114,29){\makebox(0,0)[cc]{$1$}}
%\put(150,27){\makebox(0,0)[cc]{$\sqrt{a}$}}
\put(190,29){\makebox(0,0)[cc]{$a$}}
\put(230,29){\makebox(0,0)[cc]{$\infty$}}
\put(100,65){\makebox(0,0)[cc]{$\ell_1'$}}
\put(90,55){\vector(-1,0){0}}
\put(90,15){\vector(1,0){0}}
%\put(90,35){\ellipse{60}{40}}%ell1
\bezier200(120,35)(120,55)(90,55)%ell1
\bezier200(90,55)(60,55)(60,35)%ell1
\bezier30(60,35)(60,15)(90,15)%ell1
\bezier30(90,15)(120,15)(120,35)%ell1
\put(160,60){\makebox(0,0)[cc]{$\ell_2$}}
\put(150,55){\vector(-1,0){0}}
\put(150,15){\vector(1,0){0}}
\put(150,35){\ellipse{100}{40}}%ell2
\end{picture} 
\caption{Projections to the $z$-plane of the loops $\ell_1'$ and 
         $\ell_2\in\pi_1(M)$.}
\label{fg:loop2}
\end{center}
\end{figure}
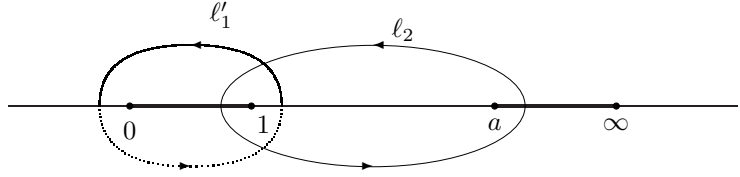 %%%%%%%%%%%%%%%%%%%%%%%%%%%%%%%%%%%%%%%%%%%%%%%%%%%%%%%%%%%

Again by the actions of the $\kappa _j$'s, we can obtain 
all of the $1$-cycles on $M$ from $\ell_1'$ and $\ell_2$. 
We now show that \eqref{eq:period1} holds for $\ell_1'$ and $\ell_2$.

First we calculate the path-integrals of $\eta$ and $g^2\eta$ along $\ell_2$.  
Then we have
\begin{align}
\oint_{\ell_2}\eta
&= 2i\sin\dfrac{k\pi}{k+1}
   \int_1^a\sqrt[k+1]{\dfrac{(a-t)^k}{t^{k+3}(t-1)^k}}dt, 
   \label{eq:ell2-eta} \\
\oint_{\ell_2}g^2\eta
&=-2ic^2\sin\dfrac{k\pi}{k+1}
   \int_1^a\sqrt[k+1]{\dfrac{(t-1)^k}{t^{k-1}(a-t)^k}}dt \nonumber \\
&=-2i\sin\dfrac{k\pi}{k+1}
   \int_1^a\sqrt[k+1]{\dfrac{(a-\tau)^k}{\tau^{k+3}(\tau -1)^k}}d\tau , 
   \label{eq:ell2-ggeta} 
\end{align}
where $\tau =a/t$. 
As a result, \eqref{eq:period1} holds for $\ell_2$.  

Next we calculate the path-integrals of $\eta $ and $g^2\eta $ along $\ell_1'$, 
and we want to reduce them to the path-integrals along $\ell_1$. 
Note that $\eta$ has a pole at $(0,0)$.  
To avoid a divergent integral, 
here we add an exact $1$-form 
which has the principal part of $\eta$. 
It is straightforward to check
\[
\eta-\frac{k+1}{2}d\left(\frac{z-1}{w}\right)
     =-\frac{k}{2}\frac{z-1}{w(z-a)}dz+\frac{1}{2}\frac{dz}{w}.
\]
Thus we have 
\begin{align*}
\oint_{\ell_1'}\eta
&= \oint_{\ell_1'}\left(
  -\frac{k}{2}\frac{z-1}{w(z-a)}dz+\frac{1}{2}\frac{dz}{w}\right) \\
&= \oint_{\ell_1}\left(
  -\frac{k}{2}\frac{z-1}{w(z-a)}dz+\frac{1}{2}\frac{dz}{w}\right) 
 = ie^{-\pi i/(k+1)}\sin\dfrac{\pi}{k+1}(kA_1-A_2),
\end{align*}
where
\[
A_1=\int_0^1\frac{(1-t)^{1/(k+1)}}{t^{2/(k+1)}(a-t)^{1/(k+1)}}dt,
\qquad
A_2=\int_0^1\frac{(a-t)^{k/(k+1)}}{t^{2/(k+1)}(1-t)^{k/(k+1)}}dt.
\]
Also, we have
\[
  \oint_{\ell_1'}g^2\eta
= \oint_{\ell_1}g^2\eta
= 2ie^{\pi i/(k+1)}\sin\dfrac{\pi}{k+1}a^{(k-2)/(k+1)}A_3,
\]
where
\[
A_3=\int_0^1\frac{(1-t)^{k/(k+1)}}{t^{(k-1)/(k+1)}(a-t)^{k/(k+1)}}dt.
\]
Hence for the loop $\ell_1'\in\pi_1(M)$, 
\eqref{eq:period1} holds if and only if 
\begin{equation}\label{eq:period-even}
kA_1+2a^{(k-2)/(k+1)}A_3-A_2=0. 
\end{equation}

Now we evaluate the values $A_1$, $A_2$, and $A_3$.  
Since $1/a\le 1/(a-t)\le 1/(a-1)$, we see that 
\begin{align*}
\frac{1}{a^{1/(k+1)}}B\left(\frac{k-1}{k+1},\frac{k+2}{k+1}\right)
\le 
A_1 &\le
\frac{1}{(a-1)^{1/(k+1)}}B\left(\frac{k-1}{k+1},\frac{k+2}{k+1}\right), \\
\frac{1}{a^{k/(k+1)}}B\left(\frac{2}{k+1},\frac{2k+1}{k+1}\right)
\le 
A_3 &\le
\frac{1}{(a-1)^{k/(k+1)}}B\left(\frac{2}{k+1},\frac{2k+1}{k+1}\right),
\end{align*}
where $B(x,y)$ is the beta function defined by 
\[
B(x,y)=\int_0^1t^{x-1}(1-t)^{y-1}dt
\qquad (\Re x>0,\; \Re y>0).  
\] 
Also, since $a-1\le a-t\le a$, we have
\[
(a-1)^{k/(k+1)}B\left(\frac{k-1}{k+1},\frac{1}{k+1}\right)
\le 
A_2 \le
a^{k/(k+1)}B\left(\frac{k-1}{k+1},\frac{1}{k+1}\right).
\]

It follows that for the case $a\to\infty$, we have 
$A_1\to 0$, $a^{(k-2)/(k+1)}A_3\to 0$, and $A_2\to\infty$.  
As a result, the left hand side of \eqref{eq:period-even} is negative. 
On the other hand, for the case $a\to 1$, we have 
\begin{align*}
kA_1 &+ 2a^{(k-2)/(k+1)}A_3-A_2 \\
&\ge \frac{k}{a^{1/(k+1)}}B\left(\frac{k-1}{k+1},\frac{k+2}{k+1}\right)
    +\frac{2}{a^{2/(k+1)}}B\left(\frac{2}{k+1},\frac{2k+1}{k+1}\right) \\
& \hspace{200pt}
    -a^{k/(k+1)}B\left(\frac{k-1}{k+1},\frac{1}{k+1}\right) \\
& =  \frac{1}{a^{1/(k+1)}}B\left(\frac{k-1}{k+1},\frac{1}{k+1}\right)
    +\frac{2}{a^{2/(k+1)}}B\left(\frac{2}{k+1},\frac{2k+1}{k+1}\right) \\
& \hspace{200pt}
    -a^{k/(k+1)}B\left(\frac{k-1}{k+1},\frac{1}{k+1}\right) \\
&\underset{a\to 1}{\longrightarrow}
     2B\left(\frac{2}{k+1},\frac{2k+1}{k+1}\right) > 0,
\end{align*}
and here we use the following formula for the beta function: 
\[
B\left(x,y+1\right)=\frac{y}{x+y}B\left(x,y\right).
\]
So, the left hand side of \eqref{eq:period-even} is positive. 

Therefore, the intermediate value theorem yields that there exists 
$a\in (1,\infty)$ which satisfies \eqref{eq:period-even}.  
Moreover, since all of $A_1$, $a^{(k-2)/(k+1)}A_3$, and $-A_2$ are 
monotone decreasing functions with respect to $a$, 
 the left hand side of \eqref{eq:period-even} is monotone decreasing 
function as well.  
This proves the uniqueness.  
(See Figure~\ref{fg:cc24}). 
\end{proof}

\begin{figure}[htbp] %%%%%%%%%%%%%%%%%%%%%%%%%%%%%%%%%%%%%%%%%%%%%%%%%%%
\begin{center}
\begin{tabular}{cccc}
 \raisebox{50pt}{$k=2$} & 
 \includegraphics[width=.33\linewidth]{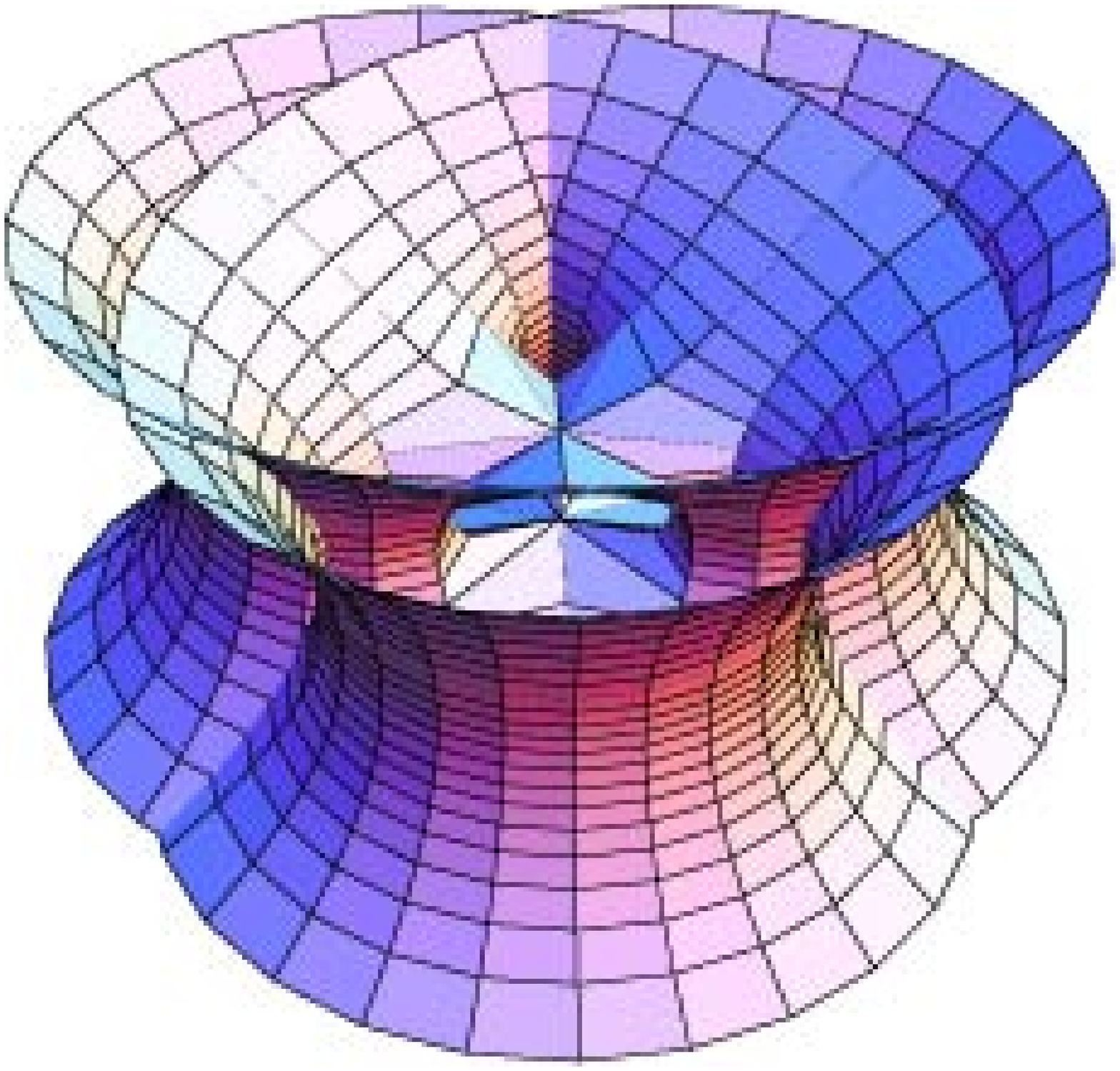} &
 \includegraphics[width=.26\linewidth]{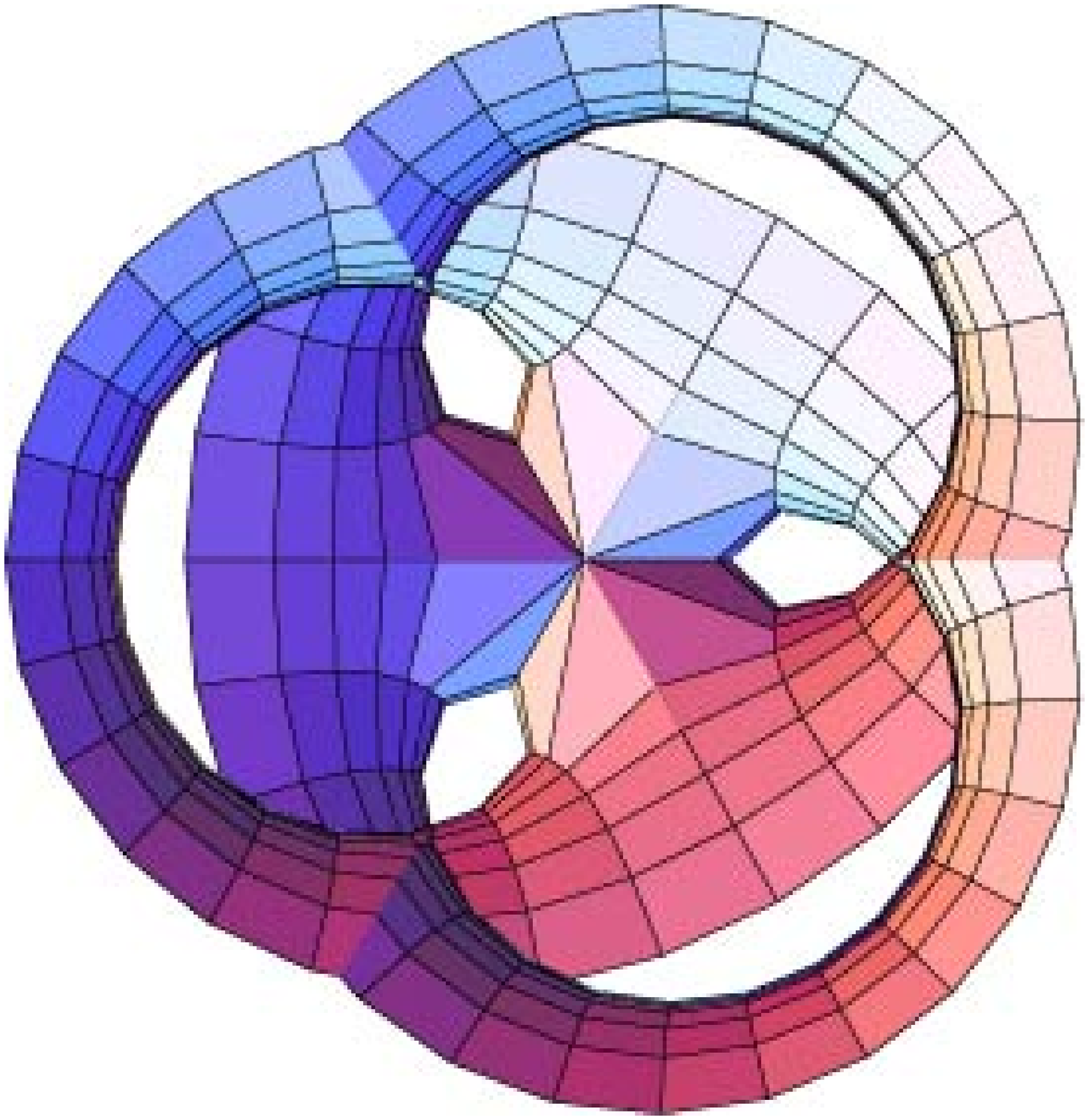} &
 \raisebox{15pt}{\includegraphics[width=.18\linewidth]{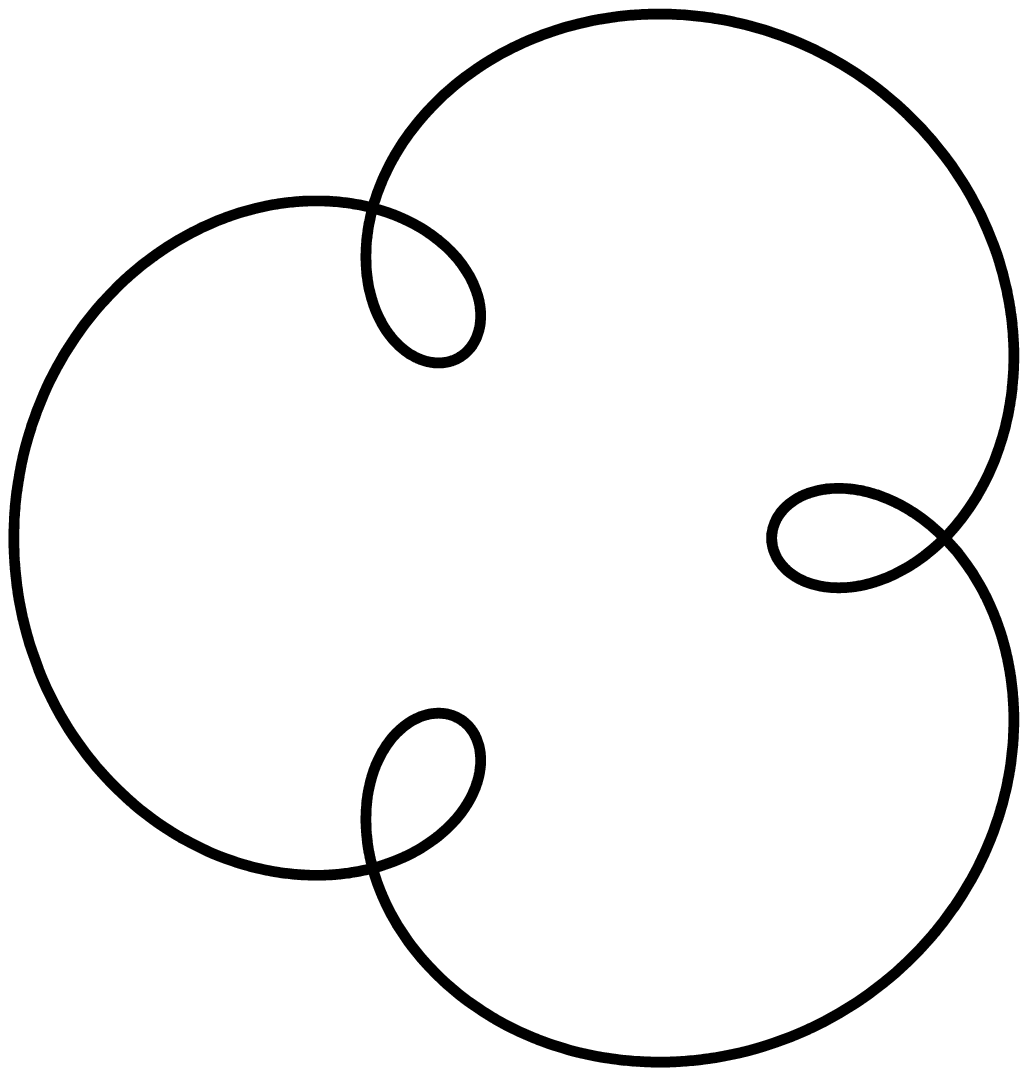}} \\
 \raisebox{50pt}{$k=4$} & 
 \includegraphics[width=.33\linewidth]{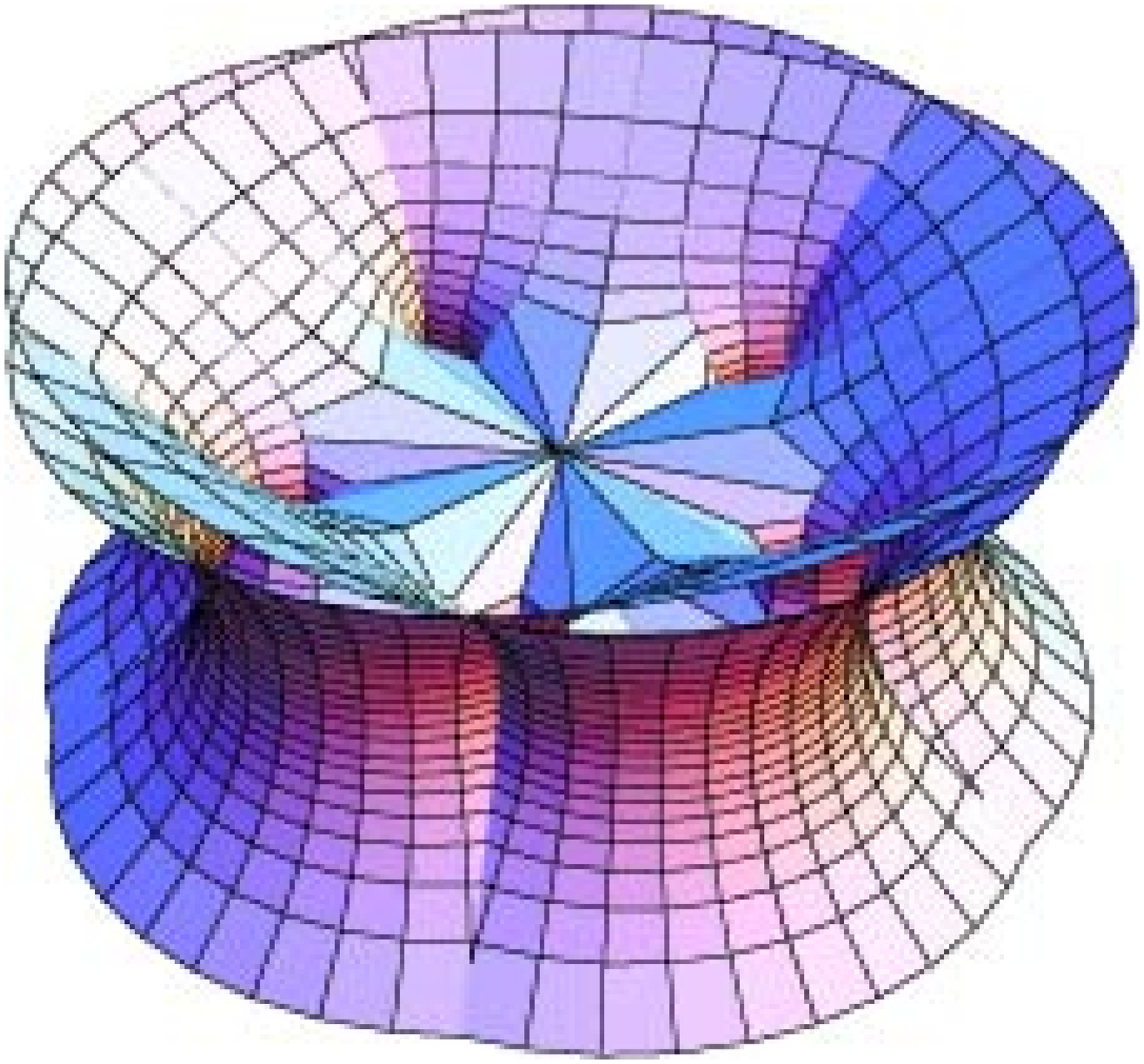} &
 \includegraphics[width=.26\linewidth]{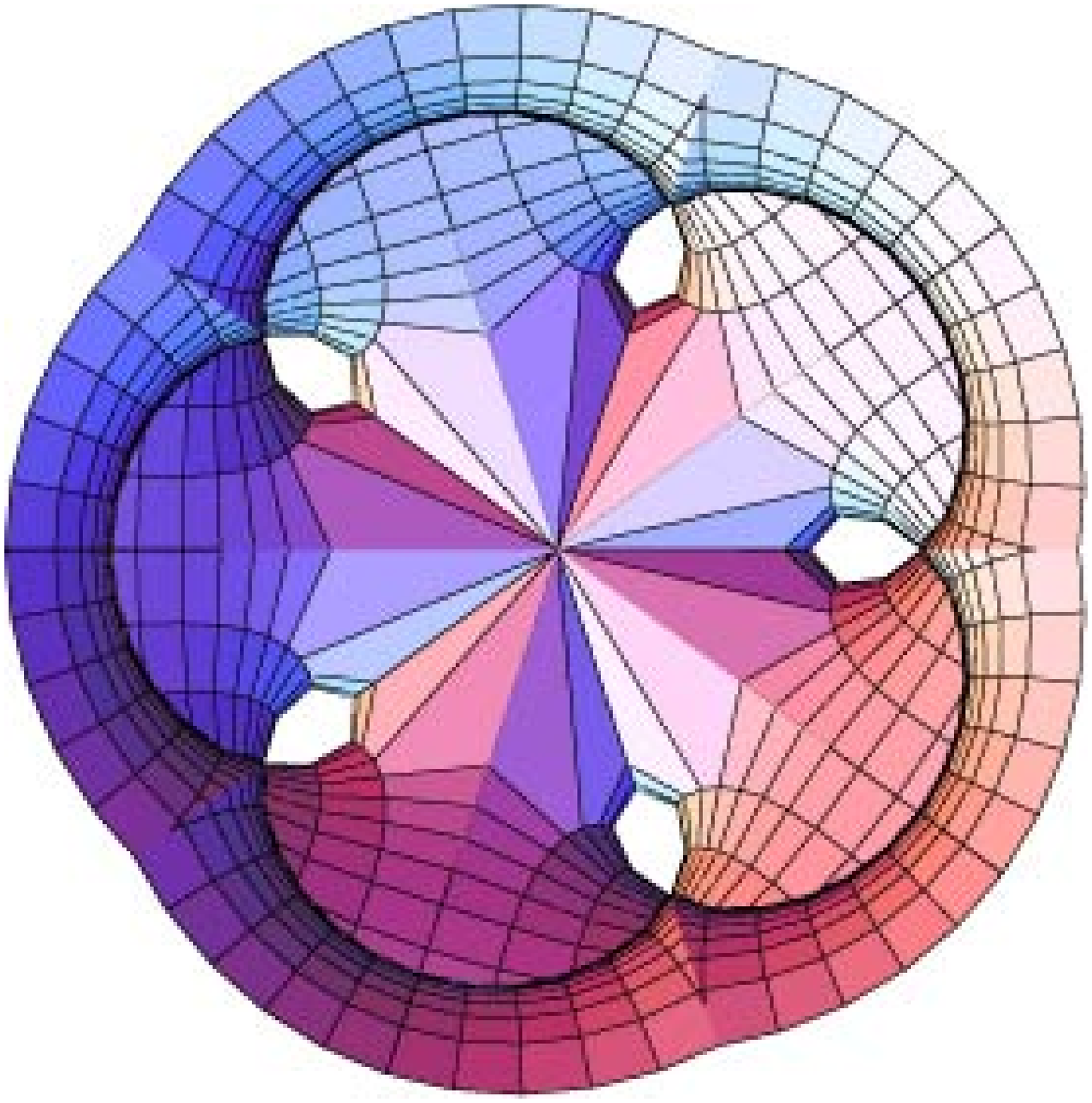} &
 \raisebox{15pt}{\includegraphics[width=.18\linewidth]{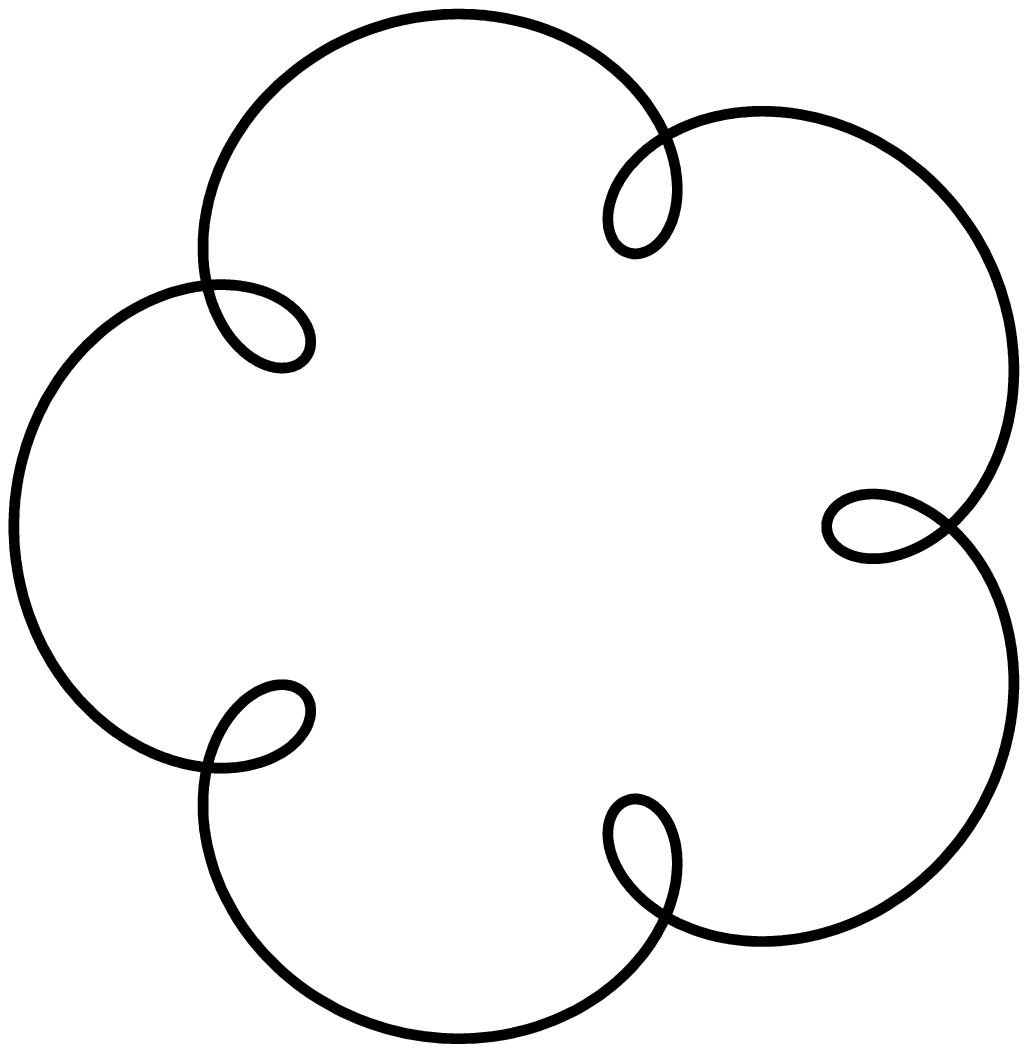}} 
\end{tabular}
\caption{Minimal surfaces of genus $k$ with two ends which 
         satisfy $\deg (g)=k+2$. 
         The middle columns show a half cut away of the surfaces by 
         the $xy$-plane.  
         The right columns show their intersection with the $xy$-plane.}
\label{fg:cc24}
\end{center}
\end{figure} %%%%%%%%%%%%%%%%%%%%%%%%%%%%%%%%%%%%%%%%%%%%%%%%%%%%%%%%%%%

Since $\deg(g)=\gamma + 2$, the next corollary follows:

\begin{corollary}\label{co:evengenus}
For all even number $\gamma$, 
there exists a complete conformal minimal surface of genus $\gamma$ 
with two ends which has least total absolute curvature. 
\end{corollary}

Combining Corollaries~\ref{co:genus1} and \ref{co:evengenus} proves 
Main Theorem~\ref{main1} in the introduction. 

%\begin{remark}
%If $k$ is an odd number such that $k\geq 3$, then we can prove the existence 
%of $a\in (1,\infty)$ satisfying \eqref{eq:period} in the same way.  
%But in this case, equality in \eqref{eq:ineq} does not hold 
%because the genus is $k-1$ and $\deg (g)=k+2$.  
%(See Figure~\ref{fg:cc35}.)  
%\end{remark}
%
%\begin{figure}[htbp] %%%%%%%%%%%%%%%%%%%%%%%%%%%%%%%%%%%%%%%%%%%%%%%%%%%
%\begin{center}
%\begin{tabular}{cccc}
% \raisebox{50pt}{$k=3$} & 
% \includegraphics[width=.33\linewidth]{cc3mesh.eps} &
% \includegraphics[width=.26\linewidth]{cc3t.eps} &
% \raisebox{15pt}{\includegraphics[width=.18\linewidth]{epitro3r.eps}} \\
% \raisebox{50pt}{$k=5$} & 
% \includegraphics[width=.33\linewidth]{cc5mesh.eps} &
% \includegraphics[width=.26\linewidth]{cc5t.eps} &
% \raisebox{15pt}{\includegraphics[width=.18\linewidth]{epitro5r.eps}} 
%\end{tabular}
%\caption{Minimal surfaces of genus $k-1$ with two ends which 
%         satisfy $\deg (g)=k+2$. 
%         The middle columns show a half cut away of the surfaces by 
%         $xy$-plane.  
%         The right columns show their intersection with $xy$-plane.}
%\label{fg:cc35}
%\end{center}
%\end{figure} %%%%%%%%%%%%%%%%%%%%%%%%%%%%%%%%%%%%%%%%%%%%%%%%%%%%%%%%%%%
%
Next we discuss the above minimal surfaces from the point of view of the 
Bj\"orling problem. 
As we mentioned in the introduction, there is a construction method for 
minimal surfaces from a given curve. We show that every minimal surface given 
in this subsection gives a solution to the Bj\"orling problem 
in the higher genus case and the generating curve 
is a closed planar geodesic. 

Let $l$ be a fixed point set of $\kappa_3\circ\kappa_1$. 
Using (1.1), we see that $\kappa_3\circ\kappa_1$ is an isometry, 
and thus $l$ is a geodesic. An explicit description of $l$ is given by 
\[\dfrac{a}{\overline{z}}=z,\qquad \dfrac{1}{c^2\,\overline{w}}=w,\]
that is, 
\[|z|=\sqrt{a},\qquad |w|=\dfrac{1}{c}.\]
Hence we conclude that $l$ is a closed geodesic. 
Moreover, by Lemma~\ref{lm:sym2}, $l$ lies in the $xy$-plane, 
and therefore the assertion follows.

\section{Uniqueness}\label{uniqueness}

In this section, we will prove Main Theorem~\ref{main2} through four subsections. 

\subsection{Symmetry}
First, we refer to some basic results about symmetries of a minimal surface. 
(See p. 349 in \cite{LM}.) 

Let $f:M\to \mathbb{R}^3$ be a conformal minimal immersion, 
with $(g,\,\eta)$ its Weierstrass data. Suppose that $A:M\to M$ 
is a diffeomorphism. $A$ is said to be a {\it symmetry} if there exists $O\in 
\mathcal{O}(3,\,\mathbb{R})$ and $v\in \mathbb{R}^3$ such that 
\[(f\circ A)(p)=O f(p)+v.\]
Denote by ${\rm Sym}(M)$ the group of symmetries of $M$, and by ${\rm Iso}(M)$ 
the isometry group of $M$. Then, by definition, ${\rm Sym}(M)$ is a subgroup of 
${\rm Iso}(M)$. 
Let ${\rm L}(M)$ be the group of holomorphic 
and antiholomorphic diffeomorphisms $\alpha$ of $M$ satisfying 
\[G\circ \alpha (p)=O\circ G (p),\]
where $G:M\to S^2$ is the Gauss map and $O\in \mathcal{O}(3,\,\mathbb{R})$ 
is a linear isometry of $\mathbb{R}^3$. 
We now assume that $f$ is complete, and of finite total curvature. 
Lopez and Martin pointed out that if one of the following three differentials $(1-g^2)\eta,\,i(1+g^2)\eta,\,2g\eta$ is not exact, 
then 
\[{\rm L}(M)={\rm Iso}(M)={\rm Sym}(M). \]

Suppose that $f$ has two ends. By Theorem~\ref{th:huber-osserman}, there exists a compact Riemann 
surface $\overline{M}_{\gamma}$ of genus $\gamma$ and two points 
$p_1,\,p_2 \in \overline{M}_{\gamma}$ such that $M$ 
is conformally equivalent to $\overline{M}_{\gamma}-\{p_1,\,p_2\}$. 
A symmetry of $f(M)$ extends to $\overline{M}_{\gamma}$ leaving the set $\{p_1,\,p_2\}$ 
invariant. 
By the Hurwitz' Theorem, the group ${\rm Sym}(M)$ is finite, and so up to a suitable 
choice of the origin, ${\rm Sym}(M)$ is a finite group $\Delta$ of orthogonal linear 
transformations of $\mathbb{R}^3$. 

We assume that ${\rm Sym}(M)$ has $4(\gamma+1)$ elements ($\gamma\geq 1$) and 
${\rm L}(M)={\rm Iso}(M)={\rm Sym}(M)$. 
If there is no symmetry in $\Delta$ such that either $p_1$ or $p_2$ is fixed, then $\Delta$ 
has at most $4$ elements by a fundamental argument in linear algebra. 
Hence, we may assume without loss of generality that there exists a symmetry such that 
$p_1$ is fixed by the symmetry. 
Up to rotations, we may assume $g(p_1)=0$, and then $\Delta$ leaves the $x_3$-axis invariant. 

We now focus on the following two cases: 
the case $\gamma=1$ with $(d_1,\,d_2)=(1,\,3)$, the even genus case with $(d_1,\,d_2)=(2,\,2)$ 
(for the definition of $d_i$, see Equation \eqref{eq:di}). 
For the former case, every symmetry in $\Delta$ leaves $p_i$ invariant. 
So we see $g(p_2)=0$ or $\infty$. 
For the latter case, we have $|\Delta|\geq 12$, and then there exist at least two symmetries which leave $p_i$ invariant. 
Hence $g(p_2)=0$ or $\infty$. 

Let $\Delta_0$ be the subgroup of holomorphic transformations in $\Delta$, and 
denote by $\mathcal{R}\subset \Delta_0$ the cyclic subgroup of rotations around 
the $x_3$-axis. Clearly, we obtain that 
\begin{equation}\label{ineq:1}
[\Delta:\Delta_0]\leq 2,\quad  [\Delta_0:\mathcal{R}]\leq 2.
\end{equation}
So the subgroups $\Delta_0\subset \Delta$ and $\mathcal{R}\subset \Delta_0$ 
are both normal. 

Let $R$ be the rotation around the $x_3$-axis with the smallest positive angle 
in $\Delta_0$, that is, $\mathcal{R}=\langle R\rangle$. 
We first consider the quotient map 
$\pi_{\mathcal{R}}:\overline{M}_{\gamma}\to \overline{M}_{\gamma}/\mathcal{R}$. 
From 
\eqref{ineq:1}, we see that 
\begin{equation}\label{ineq:2}
\deg(\pi_{\mathcal{R}})=|\mathcal{R}|\geq \gamma+1.
\end{equation}
By the Riemann-Hurwitz formula, we have 
\begin{align}\label{eq:00}
\nonumber |\mathcal{R}|(2-2\gamma(\overline{M}_{\gamma}/\mathcal{R})) &= 2-2\gamma+\sum_{p\in \overline{M}_{\gamma}} (\mu(p)-1) \\
&=  2-2\gamma+2 (|\mathcal{R}|-1)+\sum_{p\in M} (\mu(p)-1), 
\end{align}
where $\gamma(\overline{M}_{\gamma}/\mathcal{R})$ is the genus of 
$\overline{M}_{\gamma}/\mathcal{R}$ and $\mu(p)-1$ is the ramification index at $p$. 
Let $q'_1,\,\dots,\, q'_t$ be ramified values of $\pi_{\mathcal{R}}$ except for 
the $\pi_{\mathcal{R}}(p_i)$'s, and $m_i-1$ 
the ramification index at $p\in \pi_{\mathcal{R}}^{-1}(q'_i)$. 
Note that $2\leq m_i\leq |\mathcal{R}|$ and the ramification index at $p_i$ is 
$|\mathcal{R}|-1$. Combining \eqref{ineq:2} and \eqref{eq:00} yields 
\begin{align}
\label{eq:5} |\mathcal{R}|(2-2\gamma(\overline{M}_{\gamma}/\mathcal{R}))
&=2-2\gamma+2(|\mathcal{R}|-1)+\sum_{i=1}^t(m_i-1)\dfrac{|\mathcal{R}|}{m_i} \\
\nonumber &\geq 2+\sum_{i=1}^t(m_i-1)\dfrac{|\mathcal{R}|}{m_i} >0.
\end{align}
It follows that 
$\gamma(\overline{M}_{\gamma}/\mathcal{R})=0$. This, combined with \eqref{ineq:2} 
and \eqref{eq:5}, implies 
\begin{align}\label{eq:6}
2\gamma=\sum_{i=1}^t(m_i-1)\dfrac{|\mathcal{R}|}{m_i}=
|\mathcal{R}|\left\{\sum_{i=1}^t\left(\dfrac{1}{2}-\dfrac{1}{m_i}\right)
+\dfrac{t}{2}\right\}\geq (\gamma+1)\,\dfrac{t}{2}.
\end{align}
So $t\leq \dfrac{4\gamma}{\gamma+1}<4$, and thus $t=1,\,2,\,3$. We remark that 
$|\mathcal{R}|=\gamma+1,\,2(\gamma+1),\,4(\gamma+1)$. \\
\\
$\underline{{\rm The}\>\> {\rm case}\>\> t=1}$

From the first equality of \eqref{eq:6}, we obtain 
$2\gamma=|\mathcal{R}|(1-\frac{1}{m_1})$.
For the case $|\mathcal{R}|=\gamma+1$, 
$\frac{1}{m_1}=-\frac{\gamma-1}{\gamma+1}\leq 0$, which is absurd. 
Next, for the case $|\mathcal{R}|=2(\gamma+1)$, $m_1=\gamma+1$ holds. Finally, 
$|\mathcal{R}|=4(\gamma+1)$ gives 
$\frac{1}{2}>\frac{\gamma}{2(\gamma+1)}=1-\frac{1}{m_1}\geq 
1-\frac{1}{2}=\frac{1}{2}$ which leads to a contradiction. \\
\\
$\underline{{\rm The}\>\> {\rm case}\>\> t=2}$

We obtain $2\gamma=|\mathcal{R}|(2-\frac{1}{m_1}-\frac{1}{m_2})$ 
from the first equality of \eqref{eq:6}. 
Without loss of generality, we may assume $m_1\leq m_2$. Then, 
\[ 2-\dfrac{2}{m_1}\leq 
\dfrac{2\gamma}{|\mathcal{R}|}=2-\dfrac{1}{m_1}-\dfrac{1}{m_2}
\leq 2-\dfrac{2}{m_2}.\]
For the case $|\mathcal{R}|=\gamma+1$, $\gamma+1\leq m_2\leq |\mathcal{R}|=\gamma+1$ holds, 
and thus $m_2=\gamma+1$ and $m_1=\gamma+1$. 
Next we consider the case $|\mathcal{R}|\geq 2(\gamma+1)$. 
In this case, we have $2-\frac{2}{m_1}\leq \frac{2\gamma}{|\mathcal{R}|}\leq 
\frac{\gamma}{\gamma+1}<1$, and so $m_1<2$, which is absurd.  \\
\\
$\underline{{\rm The}\>\> {\rm case}\>\> t=3}$

It follows from the first equality of \eqref{eq:6} that 
$2\gamma=|\mathcal{R}|(3-\frac{1}{m_1}-\frac{1}{m_2}-\frac{1}{m_3})$. 
We may assume $m_1\leq m_2\leq m_3$. Then 
\[3-\dfrac{3}{m_1}\leq \dfrac{2\gamma}{|\mathcal{R}|}
=3-\dfrac{1}{m_1}-\dfrac{1}{m_2}-\dfrac{1}{m_3}\leq3-\dfrac{3}{m_3}.\]
For the case $|\mathcal{R}|=\gamma+1$, $3-\frac{3}{m_1}\leq \frac{2\gamma}
{|\mathcal{R}|}=\frac{2\gamma}{\gamma+1}<2$. So $m_1<3$, and hence $m_1=2$. 
As a result, 
\[\dfrac{5}{2}-\dfrac{2}{m_2}\leq \dfrac{2\gamma}{\gamma+1}
=\dfrac{5}{2}-\dfrac{1}{m_2}-\dfrac{1}{m_3}\leq\dfrac{5}{2}-\dfrac{2}{m_3}.\]
$\frac{5}{2}-\frac{2}{m_2}\leq \frac{2\gamma}{\gamma+1}$ gives 
$m_2\leq \frac{4(\gamma+1)}{\gamma+5}<4$. Thus $m_2=2,\,3$. 
For the case $m_2=2$, $\frac{2\gamma}{\gamma+1}=2-\frac{1}{m_3}$ holds, 
and so $m_3=\frac{\gamma+1}{2}$. 
For the case $m_2=3$, $\frac{2\gamma}{\gamma+1}=\frac{13}{6}-\frac{1}{m_3}$. 
It follows that $m_3=\frac{6(\gamma+1)}{\gamma+13}<6$, and hence 
$m_3=3,\,4,\,5$. So $\gamma=11,\,23,\,59$. 
For the case $|\mathcal{R}|\geq 2(\gamma+1)$,  
$3-\frac{3}{m_1}\leq \frac{2\gamma}{|\mathcal{R}|}\leq 
\frac{\gamma}{\gamma+1}<1$. So $m_1<\frac{3}{2}$, and this contradicts $m_1\geq 2$. 

\medskip

As a consequence, we obtain the following tables:
\begin{table}[htbp] %%%%%%%%%%%%%%%%%%%%%%%%%%%%%%%%%%%%%%%%%%%%%%%%%%%%
\begin{center}
 \begin{tabular}{|c|c|} \hline
  $|\mathcal{R}|$ & $m_i$ \\ \hline\hline
  $|\mathcal{R}|=\gamma+1$  & do not occur \\ \hline
  $|\mathcal{R}|=2(\gamma+1)$ & $m_1=\gamma+1$ \\ \hline
  $|\mathcal{R}|=4(\gamma+1)$ & do not occur  \\ \hline
 \end{tabular} 
\caption{The case $t=1$.}
\label{tb:t=1}
\end{center}
\end{table} %%%%%%%%%%%%%%%%%%%%%%%%%%%%%%%%%%%%%%%%%%%%%%%%%%%%%%%%%%%%

\begin{table}[htbp] %%%%%%%%%%%%%%%%%%%%%%%%%%%%%%%%%%%%%%%%%%%%%%%%%%%%
\begin{center}
 \begin{tabular}{|c|c|} \hline
  $|\mathcal{R}|$ & $m_i$ \\ \hline\hline
  $|\mathcal{R}|=\gamma+1$ & $m_1=m_2=\gamma+1$ \\ \hline
  $|\mathcal{R}|\geq 2(\gamma+1)$ & do not occur \\ \hline
 \end{tabular} 
\caption{The case $t=2$.}
\label{tb:t=2}
\end{center}
\end{table} %%%%%%%%%%%%%%%%%%%%%%%%%%%%%%%%%%%%%%%%%%%%%%%%%%%%%%%%%%%%

\begin{table}[htbp] %%%%%%%%%%%%%%%%%%%%%%%%%%%%%%%%%%%%%%%%%%%%%%%%%%%%
\begin{center}
  \begin{tabular}{|c|c|c|c|c|} \hline
  $|\mathcal{R}|$ & \multicolumn{4}{|c|}{$m_i$} \\ \hline\hline
  & $m_1$ & $m_2$ & $m_3$ & $\gamma$  \\ \cline{2-5}
  &  $2$  & $2$ &  $(\gamma+1)/2$  &  odd ($>1$)\\
  $|\mathcal{R}|=\gamma+1$ &  $2$  & $3$  & $3$ & $11$  \\
  &  $2$  & $3$  & $4$    & $23$ \\
  &  $2$  & $3$  & $5$   & $59$  \\ \hline
  $|\mathcal{R}|\geq 2(\gamma+1)$ & \multicolumn{4}{|c|}{do not occur} \\ \hline
 \end{tabular} 
\caption{The case $t=3$.}
\label{tb:t=3}
\end{center}
\end{table} %%%%%%%%%%%%%%%%%%%%%%%%%%%%%%%%%%%%%%%%%%%%%%%%%%%%%%%%%%%%

Note that, for $t=2$, $\pi_{\mathcal{R}}$ is a cyclic branched cover of $S^2$, of order $\gamma+1$, 
whose branch points are the fixed points of $\mathcal{R}$, that is, 
$p_1$, $p_2$, $\pi_{\mathcal{R}} ^{-1} (q'_1)$, $\pi_{\mathcal{R}} ^{-1} (q'_2)$. 

For the case $t=1$, $\pi_{\mathcal{R}}^{-1}(q'_1)=\{q_1,\,q_2\}$ for some two points 
$q_1,\,q_2\in \overline{M}_{\gamma}$. 
If $R$ leaves every $q_i$ invariant, then $m_1$ must be $2(\gamma+1)$. Thus $R(q_1)=q_2$ 
and $R(q_2)=q_1$. So $R^2(q_i)=q_i$ and $f(q_1)=f(q_2)\in \{x_3{\rm -axis}\}$. 
Now we consider the quotient map $\pi_{\langle R^2\rangle}:
\overline{M}_{\gamma}\to \overline{M}_{\gamma}/\langle R^2 \rangle$. 
From the Riemann-Hurwitz formula, 
\begin{align*}
|\langle R^2\rangle |\,(2-2\gamma(\overline{M}/\langle R^2\rangle ))&=2-2\gamma+
4(|\langle R^2 \rangle |-1) =2\gamma+2>0.
\end{align*}
Hence we obtain $\gamma(\overline{M}/\langle R^2\rangle )=0$. It follows that 
$\pi_{\langle R^2 \rangle}$ is a cyclic branched cover of $S^2$, of order $\gamma+1$, 
whose branch points are $p_1,\,p_2,\,q_1,\,q_2$. This case corresponds to the case $t=2$, and thus we can determine 
the case $t=1$ after we consider the case $t=2$. 

\medskip

Next, we consider the quotient map $\pi_{\Delta_0}:\overline{M}_{\gamma}
\to \overline{M}_{\gamma}/\Delta_0$ and repeat similar arguments as above. 
From the Riemann-Hurwitz formula, 
we obtain 
\begin{equation}\label{eq:0}
2\gamma -2=|\Delta_0|(2\gamma(\overline{M}_{\gamma}/\Delta_0)-2)+
\sum_{p\in \overline{M}_{\gamma}} (\mu(p)-1).
\end{equation}
We now treat the two cases that there is a 
symmetry $\sigma\in \Delta_0$ satisfying $\sigma(p_1)=p_2$ or not. 
For our case, we may exclude the case $t=3$, and consider the case $t=2$. 
It follows from \eqref{ineq:1} that 
$|\Delta_0|=2(\gamma+1)$ and $\sigma\in \Delta_0\setminus \mathcal{R}$. \\
$\,$\\
\textbf{The case $p_1$ can be transformed to $p_2$}

If there exists such $\sigma$, then the ramification index at $p_i$ 
must be $|\Delta_0|/2-1$. So \eqref{eq:0} can be reduced to 
\[2\gamma -2=|\Delta_0|(2\gamma(\overline{M}/\Delta_0)-2)+
2(|\Delta_0|/2-1)+\sum_{p\in M} (\mu(p)-1). \]
Hence, 
\begin{align}\label{eq:11}
\nonumber 2\gamma &= |\Delta_0|(2 \gamma(\overline{M}/\Delta_0)-1)
+\sum_{p\in M} (\mu(p)-1) \\
 &= 2(\gamma+1) (2 \gamma(\overline{M}/\Delta_0)-1)
+\sum_{p\in M} (\mu(p)-1) \\
\nonumber & \geq 2(\gamma+1) (2 \gamma(\overline{M}/\Delta_0)-1).
\end{align}
So the case $\gamma(\overline{M}/\Delta_0)>0$ leads to a contradiction, and thus 
$\gamma(\overline{M}/\Delta_0)=0$ holds. As a consequence, \eqref{eq:11} 
can be reduced to  
\begin{equation}\label{eq:12}
4\gamma +2= \sum_{p\in M} (\mu(p)-1). 
\end{equation}
Let $r'_1,\,\dots,\, r'_s$ be ramified values of $\pi_{\Delta_0}$ except for the 
$\pi_{\Delta_0}(p_i)$'s, and $m_i-1$ the ramification index at $p\in \pi_{\Delta_0}^{-1}(r'_i)$. 
Note that $2\leq m_i\leq |\Delta_0|$. \eqref{eq:12} can be rewritten as 
\begin{equation}\label{eq:13}
4\gamma +2= \sum_{i=1}^s (m_i-1)\dfrac{|\Delta_0|}{m_i}=2(\gamma+1)\sum_{i=1}^s \left(1-\dfrac{1}{m_i}\right). 
\end{equation}
If $s=1$, then, \eqref{eq:13} yields $2\gamma<0$, which is absurd. Hence $s\geq 2$, and \eqref{eq:13} takes the form 
\begin{align*}
2\gamma &=2(\gamma+1)\left\{\sum_{i=1}^s\left(1-\dfrac{1}{m_i}\right)-1\right\}
=2(\gamma+1)\left\{\dfrac{s-2}{2}+
\sum_{i=1}^s\left(\dfrac{1}{2}-\dfrac{1}{m_i}\right)   \right\}.
\end{align*}
So
\[1>\dfrac{2\gamma}{2(\gamma+1)}= \dfrac{s-2}{2}+
\sum_{i=1}^s\left(\dfrac{1}{2}-\dfrac{1}{m_i}\right) \geq 
\dfrac{s-2}{2}.\]
As a result,  $2\leq s<4$ follows, and thus $s=2,\,3$. \\
$\,$\\
$\underline{{\rm The}\>\> {\rm case}\>\> s=2}$

\eqref{eq:13} implies 
\[2\gamma =2(\gamma+1)\left(1-\dfrac{1}{m_1}-\dfrac{1}{m_2}\right),\]
that is, 
\[\dfrac{1}{m_1}+\dfrac{1}{m_2}=\dfrac{1}{\gamma+1}\]
holds. The inequalities $2\leq m_i \leq |\Delta_0| = 2(\gamma+1)$ yield 
$m_1=m_2=2(\gamma+1)$. \\
$\,$\\
$\underline{{\rm The}\>\> {\rm case}\>\> s=3}$

From \eqref{eq:13}, 
\begin{equation}\label{ineq:0}
2\gamma=2(\gamma+1)\left( 2-\dfrac{1}{m_1} 
-\dfrac{1}{m_2}-\dfrac{1}{m_3} \right).
\end{equation}
Without loss of generality, we may assume 
$m_1\leq m_2\leq m_3$. In this case, 
\[2-\dfrac{3}{m_1} \leq 2-\dfrac{1}{m_1}-\dfrac{1}{m_2}-\dfrac{1}{m_3} \leq
2-\dfrac{3}{m_3}.\]
By \eqref{ineq:0}, we obtain  
\[2(\gamma+1) \left( 2-\dfrac{3}{m_1}\right) \leq 
2(\gamma+1) \left( 2-\dfrac{1}{m_1} 
-\dfrac{1}{m_2}-\dfrac{1}{m_3} \right) = 2\gamma.\]
Thus we have 
\[m_1\leq \dfrac{3(\gamma+1)}{\gamma+2}<3.\]
Hence $m_1=2$. Moreover, let us consider the case $m_1=2$ and $m_2\leq m_3$. 
Then  
\[2(\gamma+1) \left( \dfrac{3}{2} -\dfrac{2}{m_2}\right) \leq 
2(\gamma+1) \left( 2-\dfrac{1}{2} -\dfrac{1}{m_2}-\dfrac{1}{m_3} \right) 
= 2\gamma,\]
and so 
\[m_2\leq \dfrac{4(\gamma+1)}{\gamma+3}<4.\]
It follows that $m_2=2,\,3$. 
For the case $m_2=2$, $2\gamma=2(\gamma+1)\left(1-\dfrac{1}{m_3}\right)$, 
and thus $m_3= \gamma+1$. 
For the case $m_2=3$,  
$2\gamma=2(\gamma+1)\left(\dfrac{7}{6}-\dfrac{1}{m_3}\right)$ and so 
$m_3= \dfrac{6(\gamma+1)}{\gamma+7}<6$. As a consequence, $(m_3,\,\gamma)=(3,\,5),\,(4,\,11),\,(5,\,29)$. \\
$\,$\\
\textbf{The case $p_1$ cannot be transformed to $p_2$}

If there does not exist $\sigma$ satisfying $\sigma(p_1)=p_2$, then the ramification 
index at $p_i$ must be $|\Delta_0|-1$. It follows that \eqref{eq:0} can be reduced to 
\[2\gamma -2=|\Delta_0|(2\gamma(\overline{M}/\Delta_0)-2)+
2(|\Delta_0|-1)+\sum_{p\in M} (\mu(p)-1),\]
and thus 
\begin{equation}\label{eq:1}
2\gamma =2|\Delta_0|\gamma(\overline{M}/\Delta_0)+\sum_{p\in M} (\mu(p)-1)
=4(\gamma+1)\gamma(\overline{M}/\Delta_0)+\sum_{p\in M} (\mu(p)-1).
\end{equation}
\eqref{eq:1} yields $\gamma(\overline{M}/\Delta_0)=0$, and so 
\begin{equation}\label{eq:4}
2\gamma =\sum_{p\in M} (\mu(p)-1).
\end{equation}
Suppose that $r'_1,\,\cdots,\, r'_s$ are ramified values of $\pi_{\Delta_0}$ 
except for the 
$\pi_{\Delta_0}(p_i)$'s, and $m_i-1$ the ramification index at $p\in \pi_{\Delta_0}^{-1}(r'_i)$ 
as above. \eqref{eq:4} can be rewritten as 
\[2\gamma =\sum_{i=1}^s (m_i-1)\dfrac{|\Delta_0|}{m_i}
=2(\gamma+1)\sum_{i=1}^s\left(1-\dfrac{1}{m_i}\right).\]
As a result, 
\[1>\dfrac{2\gamma}{2\gamma+2}= \sum_{i=1}^s\left(1-\dfrac{1}{m_i}\right)
=\sum_{i=1}^s\left\{\dfrac{1}{2}+ \left(\dfrac{1}{2}-\dfrac{1}{m_i}\right) \right\} 
\geq \dfrac{s}{2},\]
and hence $1\leq s <2$, that is, $s=1$. Thus, we have $m_1=\gamma+1$. 

\medskip

Therefore, we obtain the following tables: 
\begin{table}[htbp] %%%%%%%%%%%%%%%%%%%%%%%%%%%%%%%%%%%%%%%%%%%%%%%%%%%%
\begin{center}
 \begin{tabular}{|c|c|c|c|c|} \hline
 \multicolumn{2}{|c|}{$s$} & \multicolumn{3}{|c|}{$m_i$} \\ \hline\hline
 \multicolumn{2}{|c|}{$s=2$} &  \multicolumn{3}{|c|}{$m_1=m_2=2(\gamma+1)$} \\ \hline
   & $m_1$ & $m_2$ & $m_3$ & $\gamma$  \\ \cline{2-5}
  &  $2$  & $2$ &  $\gamma+1$  &  arbitrary  \\
 $s=3$ &  $2$  & $3$  & $3$ & $5$  \\
  &  $2$  & $3$  & $4$    & $11$ \\
  &  $2$  & $3$  & $5$   & $29$  \\ \hline
 \end{tabular} 
\caption{The case $p_1$ can be transformed to $p_2$.}
\label{tb:p1-to-p2}
\end{center}
\end{table} %%%%%%%%%%%%%%%%%%%%%%%%%%%%%%%%%%%%%%%%%%%%%%%%%%%%%%%%%%%%

\begin{table}[htbp] %%%%%%%%%%%%%%%%%%%%%%%%%%%%%%%%%%%%%%%%%%%%%%%%%%%%
\begin{center}
 \begin{tabular}{|c|c|} \hline
  $s$ & $m_i$ \\ \hline\hline
  $s=1$ & $m_1=\gamma+1$ \\ \hline
 \end{tabular} 
\caption{The case $p_1$ cannot be transformed to $p_2$.}
\label{tb:p1-to-p1}
\end{center}
\end{table} %%%%%%%%%%%%%%%%%%%%%%%%%%%%%%%%%%%%%%%%%%%%%%%%%%%%%%%%%%%%

\medskip

\subsection{Weierstrass data for the case $\gamma=1$ with $(d_1,\,d_2)=(1,\,3)$}

By Tables~\ref{tb:t=1}--\ref{tb:t=3}, we first consider the case $t=2$. 
Then $|\mathcal{R}|=2$ and we find $|\Delta_0|=4$ by \eqref{ineq:1}. 
Set $q_1$, $q_2$ as two branch points of $\pi_{\mathcal{R}}$ distinct from 
the $p_i$'s and 
$p'_i:=\pi_{\mathcal{R}}(p_i)$, $q'_i:=\pi_{\mathcal{R}}(q_i)$. 
Since $\pi_{\mathcal{R}}$ is a cyclic branched double cover of $S^2$, 
$\overline{M}_{1}$ can be given by 
\[v^{2}=(u-p'_1)^{m_1h_1}(u-p'_2)^{m_2h_2}(u-q'_1)^{m_3h_3}
(u-q'_2)^{m_4},\]
where $h_i\in \{1,\,-1\}$ ($i=1,\,2,\,3$), $(2,\,m_i)=1$, and $R(u,\,v)=\left(u,\,-v\right)$. 

Since $(d_1,\,d_2)=(1,\,3)$, 
there does not exist $\sigma \in \Delta_0$ satisfying 
$\sigma(p_1)=p_2$. 
By Table~\ref{tb:p1-to-p1} on $\Delta_0$, there exists a transformation 
$\tau\in \Delta_0 \setminus \mathcal{R}$ satisfying $\tau (q_1)=q_2$, and thus $m_3=m_4$. 
$\tau$ induces a degree $2$ transformation $\tau':\overline{M}_{1}/\mathcal{R} 
\to \overline{M}_{1}/\Delta_0$, that is, a transformation on $S^2$ such that 
$\tau' (q'_1)=\tau'(q'_2)$. 
Choosing suitable variables $(z,\,w)$, 
we can represent $\tau(z,\,w)=(-z,\,*)$, $\tau'(z)=z^2$, 
$p'_1=0$, $p'_2=\infty$, $q'_1=1$, $q'_2=-1$, and moreover, $\overline{M}_{1}$ can be rewritten as 
$w^2=z (z^2-1)$ 
(see Figure~\ref{fg:M-1}). 

\begin{figure}[htbp] %%%%%%%%%%%%%%%%%%%%%%%%%%%%%%%%%%%%%%%%%%%%%%%%%%%
\begin{center}
\begin{picture}(340,100)

\thicklines

\put(0,80){\line(1,0){125}} \put(0,20){\line(1,0){125}} 

\put(20,86){$p'_1$} \put(50,86){$p'_2$} \put(80,86){$q'_1$}
\put(104,86){$q'_2$}
\put(20,10){$p'_1$} \put(50,10){$p'_2$} 

\put(135,50){$\tau'$}

\put(180,80){\line(1,0){125}} \put(180,20){\line(1,0){125}} 

\put(200,86){$0$} \put(230,86){$\infty$} \put(260,86){$1$}
\put(284,86){$-1$}
\put(200,10){$0$} \put(230,10){$\infty$} \put(272,10){$1$}

\put(310,80){$z$} \put(310,20){$z^2$}

\put(150,78){$\cong$} \put(150,18){$\cong$}

\thinlines

\put(20,80){\line(0,-1){60}} \put(50,80){\line(0,-1){60}} 
\put(80,80){\line(1,-5){12}} \put(104,80){\line(-1,-5){12}}

\put(200,80){\line(0,-1){60}} \put(230,80){\line(0,-1){60}} 
\put(260,80){\line(1,-5){12}} \put(284,80){\line(-1,-5){12}}

\put(130,70){\vector(0,-1){40}} \put(313,75){\vector(0,-1){45}}
\end{picture}
\caption{The Riemann surface $\overline{M}_{1}$.}
\label{fg:M-1}
\end{center}
\end{figure}
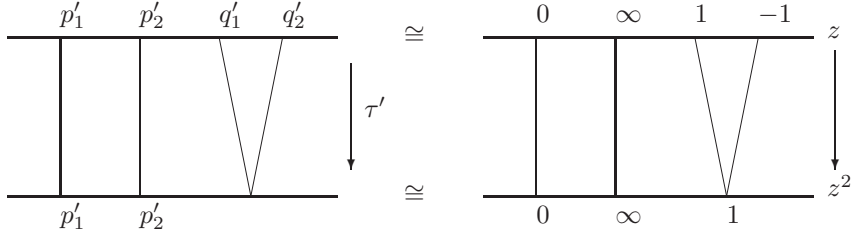 %%%%%%%%%%%%%%%%%%%%%%%%%%%%%%%%%%%%%%%%%%%%%%%%%%%%%%%%%%%%

Now we consider the Gauss map. 
$p_1,\,p_2,\,q_1,\,q_2$ are fixed points of $\mathcal{R}$, that is, fixed points by rotations around 
the $x_3$-axis. So we have 
$g(\{p_1,\,p_2,\,q_1,\,q_2\})\subset \{0,\,\infty\}$. Note that $q_1$ can be transformed to 
$q_2$ by the biholomorphism $\tau$. On the other hand, $p_1$ cannot 
be transformed to $p_2$. 
It follows that we essentially only need to consider the two cases 
as in Figure~\ref{fg:g-1}: 

\begin{figure}[htbp] %%%%%%%%%%%%%%%%%%%%%%%%%%%%%%%%%%%%%%%%%%%%%%%%%%%
\begin{center}
\begin{picture}(340,90)

\thicklines

\put(20,3){\line(1,0){120}} \put(180,3){\line(1,0){120}}

\put(55,65){\circle*{2}}\put(55,45){\circle*{2}}\put(55,25){\circle*{2}}
\put(105,45){\circle*{2}}
\put(215,60){\circle*{2}} \put(215,30){\circle*{2}}
\put(265,60){\circle*{2}} \put(265,30){\circle*{2}}

\put(42,65){$p_1$}\put(42,45){$q_1$}\put(42,25){$q_2$}
\put(92,45){$p_2$}
\put(202,60){$p_1$}\put(202,30){$p_2$}
\put(252,60){$q_1$}\put(252,30){$q_2$}

\put(52,-7){$0$}\put(102,-7){$\infty$}
\put(212,-7){$0$}\put(262,-7){$\infty$}

\thinlines

\put(155,0){$S^2$} \put(154,55){$\overline{M}_{1}$} \put(164,35){$g$}

\put(55,3){\line(0,1){80}} \put(105,3){\line(0,1){80}}
\put(215,3){\line(0,1){80}} \put(265,3){\line(0,1){80}}

\put(159,50){\vector(0,-1){37}}

\end{picture}
\caption{The possibilities of the Gauss map.}
\label{fg:g-1}
\end{center}
\end{figure}
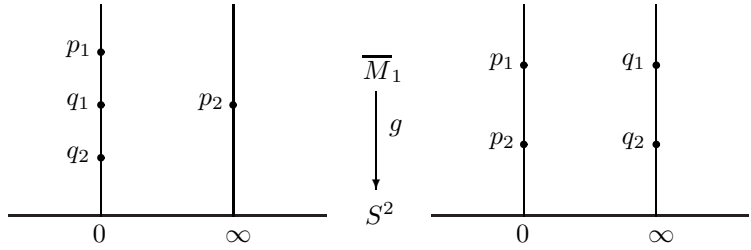 %%%%%%%%%%%%%%%%%%%%%%%%%%%%%%%%%%%%%%%%%%%%%%%%%%%%%%%%%%%%

In the right hand side case in Figure~\ref{fg:g-1}, 
the ramification index at $q_i$ is 
$\frac{\gamma+2}{2}-1\notin \mathbb{Z}$ 
since the ramification index at $q_1$ of $g$ must 
coincide with the ramification index at $q_2$ of $g$. 
Hence we only consider the left hand side case in Figure~\ref{fg:g-1}. 
The ramification index at $p_2$ may be $1-1,\,2-1,\,3-1$. 
If the ramification index is $2-1$, $g^{-1}(\infty)$ consists of $p_2$ and a 
simple pole $q\in \overline{M}_{1}$. Then $R(q)$ must be a pole of $g$, but 
$R(q)\notin \{p_2,\,q\}$.  
This contradicts $R(q)\in g^{-1}(\infty)$.  
So the divisor of $g$ is given by
\[
(g)=
\begin{cases}
p_1+q_1+q_2-3 p_2 \\
p_1+q_1+q_2-p_2-Q-R(Q) 
\end{cases}
\]
for a point $Q$. 
Since $\tau (p_2)=p_2$, $\tau$ leaves $\{$the poles of $g\}$ invariant. 
For the latter case, if we take $Q^*:=\tau (Q)$, then $Q^*$ must be a pole of $g$ 
which is distinct from the $R^i(Q)$'s. It leads to a contradiction, and so we only consider the former 
case. In this case, the divisor of the meromorphic function $z(z^2-1)$ coincides with 
that of $g^2$, Thus $g^2=c'\,z(z^2-1)$ holds for some constant $c'$. 
Hence, $\overline{M}_1$ and $g$ can be rewritten as 
\[ w^2=z(z^2-1),\quad  g=c\,w\]
for some constant $c$, and $R(z,\,w)=(z,\,-w)$, $\tau(z,\,w)=(-z,\,iw)$. Then, the divisor of $\eta$ is obtained by
\[(\eta)=
\begin{cases}
-2 p_1+2 p_2 \\
-4 p_1+4 p_2.
\end{cases}
\]
Thus, by a similar argument, 
\[
\eta^2=
\begin{cases}
c''\,\dfrac{(dz)^2}{z^3(z^2-1)} \\
c''\,\dfrac{(dz)^2}{z^5(z^2-1)}
\end{cases}
=
\begin{cases}
c''\left(\dfrac{dz}{z\,w}\right)^2 \\
c''\left(\dfrac{dz}{z^2\,w}\right)^2
\end{cases}
\]
hold for some constant $c''$. As a consequence, we obtain 
$\eta=c'''\dfrac{dz}{z\,w},\,c'''\dfrac{dz}{z^2\,w}$ for some constant $c'''$. 
The latter case is given in \S\ref{sec:genus1}, and we shall prove that the former case does not occur in 
\S~\ref{uniqueness-thm}. Note that the case $t=1$ does not occur in this case.

\subsection{Weierstrass data for the even genus case with $(d_1,\,d_2)=(2,\,2)$} 

We treat the case $t=2$ ($|\mathcal{R}|=\gamma+1$, $|\Delta_0|=2(\gamma+1)$). 
By Table~\ref{tb:t=2}, 
$\pi_{\mathcal{R}}:\overline{M}_{\gamma}\to \overline{M}_{\gamma}/\mathcal{R}$ 
is a cyclic branched cover of $S^2$. 
Thus $\overline{M}_{\gamma}$ can be represented by 
\[
v^{\gamma+1}=(u-p'_1)^{m_1h_1}(u-p'_2)^{m_2h_2}(u-q'_1)^{m_3h_3}(u-q'_2)^{m_4},
\]
where $(m_i,\,\gamma+1)=1$, $h_i=\pm 1$, and $R(u,\,v)=\left(u,\,e^{\frac{2 \pi}{\gamma+1}i}v\right)$. \\
\\
\textbf{The case $p_1$ cannot be transformed to $p_2$}

We assume that  there does not exist $\sigma \in \Delta_0$ such that 
$\sigma(p_1)=p_2$. From the table on $\Delta_0$, there exists a transformation 
$\tau\in \Delta_0\setminus \mathcal{R}$ satisfying $\tau (q_1)=q_2$, and thus $m_3=m_4$. 
$\tau$ induces a degree $2$ transformation $\tau':\overline{M}_{1}/\mathcal{R} 
\to \overline{M}_{1}/\Delta_0$, that is, a transformation  on $S^2$ such that $\tau' (q'_1)=\tau'(q'_2)$. 
Choosing suitable variables $(z,\,w)$, 
we can represent $\tau(z,\,w)=(-z,\,*)$, $\tau'(z)=z^2$, 
$p'_1=0$, $p'_2=\infty$, $q'_1=1$, $q'_2=-1$, and moreover, $\overline{M}_{\gamma}$ can be rewritten as 
$w^{\gamma+1}=z^{m_1 h_1}(z^2-1)^{m_3}$ and $R(z,\,w)=\left(z,\,e^{\frac{2 \pi}{\gamma+1}i}w\right)$  
(see Figure~\ref{fg:M-2}).

\begin{figure}[htbp] %%%%%%%%%%%%%%%%%%%%%%%%%%%%%%%%%%%%%%%%%%%%%%%%%%%
\begin{center}
\begin{picture}(340,100)

\thicklines

\put(0,80){\line(1,0){125}} \put(0,20){\line(1,0){125}} 

\put(20,86){$p'_1$} \put(50,86){$p'_2$} \put(80,86){$q'_1$}
\put(104,86){$q'_2$}
\put(20,10){$p'_1$} \put(50,10){$p'_2$} 

\put(135,50){$\tau'$}

\put(180,80){\line(1,0){125}} \put(180,20){\line(1,0){125}} 

\put(200,86){$0$} \put(230,86){$\infty$} \put(260,86){$1$}
\put(284,86){$-1$}
\put(200,10){$0$} \put(230,10){$\infty$} \put(272,10){$1$}

\put(310,80){$z$} \put(310,20){$z^2$}

\put(150,78){$\cong$} \put(150,18){$\cong$}

\thinlines

\put(20,80){\line(0,-1){60}} \put(50,80){\line(0,-1){60}} 
\put(80,80){\line(1,-5){12}} \put(104,80){\line(-1,-5){12}}

\put(200,80){\line(0,-1){60}} \put(230,80){\line(0,-1){60}} 
\put(260,80){\line(1,-5){12}} \put(284,80){\line(-1,-5){12}}

\put(130,70){\vector(0,-1){40}} \put(313,75){\vector(0,-1){45}}
\end{picture}
\caption{The Riemann surface $\overline{M}_{\gamma}$.}
\label{fg:M-2}
\end{center}
\end{figure}
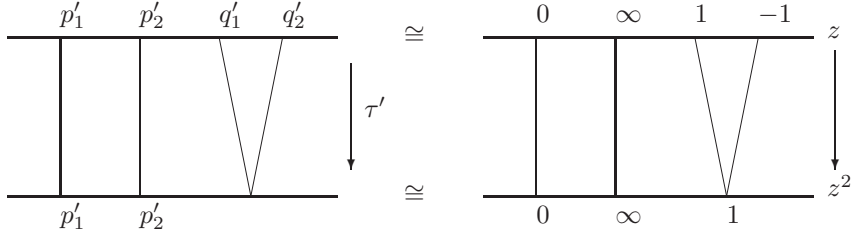 %%%%%%%%%%%%%%%%%%%%%%%%%%%%%%%%%%%%%%%%%%%%%%%%%%%%%%%%%%%%

Now we consider the Gauss map. 
$p_1,\,p_2,\,q_1,\,q_2$ are fixed points of $\mathcal{R}$, that is, fixed points by rotations around the 
$x_3$-axis. So we find $g(\{p_1,\,p_2,\,q_1,\,q_2\})\subset \{0,\,\infty\}$. Note that $q_1$ can be transformed to 
$q_2$ by the biholomorphism $\tau$. On the other hand, $p_1$ cannot 
be transformed to $p_2$. 
It follows that we essentially only need to consider the two cases 
as in Figure~\ref{fg:g-2}: 

\begin{figure}[htbp] %%%%%%%%%%%%%%%%%%%%%%%%%%%%%%%%%%%%%%%%%%%%%%%%%%%
\begin{center}
\begin{picture}(340,90)

\thicklines

\put(20,3){\line(1,0){120}} \put(180,3){\line(1,0){120}}

\put(55,65){\circle*{2}}\put(55,45){\circle*{2}}\put(55,25){\circle*{2}}
\put(105,45){\circle*{2}}
\put(215,60){\circle*{2}} \put(215,30){\circle*{2}}
\put(265,60){\circle*{2}} \put(265,30){\circle*{2}}

\put(42,65){$p_1$}\put(42,45){$q_1$}\put(42,25){$q_2$}
\put(92,45){$p_2$}
\put(202,60){$p_1$}\put(202,30){$p_2$}
\put(252,60){$q_1$}\put(252,30){$q_2$}

\put(52,-7){$0$}\put(102,-7){$\infty$}
\put(212,-7){$0$}\put(262,-7){$\infty$}

\thinlines

\put(155,0){$S^2$} \put(154,55){$\overline{M}_{\gamma}$} \put(164,35){$g$}

\put(55,3){\line(0,1){80}} \put(105,3){\line(0,1){80}}
\put(215,3){\line(0,1){80}} \put(265,3){\line(0,1){80}}

\put(159,50){\vector(0,-1){37}}

\end{picture}
\caption{The possibilities of the Gauss map.}
\label{fg:g-2}
\end{center}
\end{figure}
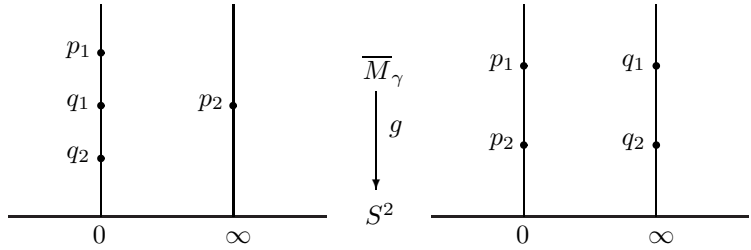 %%%%%%%%%%%%%%%%%%%%%%%%%%%%%%%%%%%%%%%%%%%%%%%%%%%%%%%%%%%%

\medskip 

\noindent
\underline{The left hand side case in Figure~\ref{fg:g-2}}

The divisor of $g$ is given by 
\[
(g)=
\begin{cases}
(\gamma+2-2N)p_1+N q_1+N q_2-(\gamma+2)p_2 \\
(\gamma+2-2N)p_1+N q_1+N q_2-p_2-Q-R(Q)-\cdots - R^{\gamma}(Q)
\end{cases}
\]
for a point $Q$. Note that $N>0$ and $\gamma+2-2N>0$. 
Since $\tau (p_2)=p_2$, $\tau$ leaves $\{$the poles of $g\}$ invariant. 
For the latter case, if we take $Q^*:=\tau (Q)$, then $Q^*$ must be a pole of $g$ 
which is distinct from the $R^i(Q)$'s.  
This leads to a contradiction, and so we only consider the former case. 
Then the divisor of $\eta$ is given by
\[
(\eta)=-3 p_1+(2\gamma+1) p_2. 
\]
Hence the divisor of $g\eta$ is obtained by 
\begin{align*}
(g\eta)&=(\gamma-2N-1)p_1+Nq_1+Nq_2+(\gamma-1) p_2. 
\end{align*}
If $\gamma-2N-1\geq 0$, then $g\eta$ is holomorphic. Thus $f$ is bounded and 
this leads to a contradiction. As a result, $\gamma-2N-1<0$ follows. 
The inequality $\gamma+2-2N>0$ yields $N<\frac{\gamma}{2}+1$. 
Also, since $\gamma$ is even, $N\leq \frac{\gamma}{2}$ holds. 
So we have $\gamma=2N$. 

It follows that the divisor of $g^{\gamma+1}$ coincides with that of 
$z^2(z^2-1)^{\frac{\gamma}{2}}$. Therefore, $\overline{M}_{\gamma}$ and $g$ can be 
rewritten as 
\[w^{\gamma+1}=z^2(z^2-1)^{\frac{\gamma}{2}},\>\>
g=c\,w\]
for some constant $c$, and $R(z,\,w)=(z,\,e^{\frac{2\pi}{\gamma+1} i}w)$, $\tau(z,\,w)=(-z,\,w)$. 
Furthermore, the divisor of $\eta^{\gamma+1}$ coincides with 
that of $\dfrac{(dz)^{\gamma+1}}{z^{\gamma-1}\,g^{2(\gamma+1)}}$. Hence 
\[\eta^{\gamma+1}=c''\,\dfrac{(dz)^{\gamma+1}}{z^{\gamma-1}\,g^{2(\gamma+1)}}
\left(=c''\,z^2\dfrac{(dz)^{\gamma+1}}{z^{\gamma+1}\,g^{2(\gamma+1)}}\right)\]
for some constant $c''$. By setting $z=u^{\gamma+1}$ and $w=u^2 v$, $\overline{M}_{\gamma}$ 
can be rewritten as 
\[
v^{\gamma+1} = (u^{2(\gamma+1)}-1)^{\frac{\gamma}{2}}, 
\]
and moreover 
\[
g= c u^2 v, \quad \eta=c''' \dfrac{u}{g^2}du 
\]
for some constant $c'''$. However, in this case, its genus is greater than $\gamma$, and such a case is excluded. 

\medskip 

\noindent
\underline{The right hand side case in Figure~\ref{fg:g-2}}

The divisor of $g$ is obtained by 
\[(g)=(\gamma+2-N)p_1+N p_2-\dfrac{\gamma+2}{2} q_1-\dfrac{\gamma+2}{2} q_2,\]
where $N>0$ and $\gamma+2-N>0$. 
Also, the divisor of $\eta$ is given by 
\[(\eta)=-3 p_1-3 p_2+(\gamma+2)q_1+(\gamma+2)q_2. \]
Thus the divisor of $g\eta$ is obtained by 
\begin{align*}
(g\eta)=(\gamma-N-1)p_1+(N-3) p_2 +\frac{\gamma+2}{2}q_1+\frac{\gamma+2}{2}q_2. 
\end{align*}
If $\gamma-N-1 \geq 0$ and $N-3\geq 0$, then $g\eta$ is holomorphic. 
Hence $\gamma-N-1<0$ or $N-3<0$. 
From the inequality $\gamma+2-N>0$, the case $\gamma-N-1<0$ corresponds to $\gamma=N,\,N-1$. 
The case $N-3<0$ implies $N=1,\,2$. Essentially, we consider the cases 
$N=1,\,2$. 
The divisor of $g^{\gamma+1}$ coincides with 
$\dfrac{z^{\gamma+2-N}}{(z^2-1)^{\frac{\gamma+2}{2}}}$. 
As a consequence, 
$\overline{M}_{\gamma}$ and $g$ can be rewritten as 
\[w^{\gamma+1}=\dfrac{(z^2-1)^{\frac{\gamma+2}{2}}}{z^{\gamma+2-N}},\>\> 
g=\dfrac{c}{w},\] 
for some constant $c$. 
If $N=1$, then $\gamma+2-N$ and $\gamma+1$ are not coprime. 
So $N=2$, and $R(z,\,w)=(z,\,e^{\frac{2\pi}{\gamma+1} i}w)$, 
$\tau(z,\,w)=(-z,\,w)$. 
Furthermore, the divisor $\eta^{\gamma+1}$ coincides with that of 
$\dfrac{(z^2-1)^{2}}{z^{\gamma+3}}(dz)^{\gamma+1}$. 
As a consequence, 
\[\eta^{\gamma+1}=c'\, \dfrac{(z^2-1)^{2}}{z^{\gamma+3}}(dz)^{\gamma+1}
\left(=\dfrac{c'}{z^6}\,\left(\dfrac{z^3\,w^4\,dz}{(z^2-1)^2}\right)^{\gamma+1}\right)\]
for some constant $c'$. By similar arguments as above, 
we may exclude this case except the case $\gamma=2$. 
For the case $\gamma=2$, $\overline{M}_{\gamma}$ and $g$ can be rewritten as 
\[
w^{3}=\dfrac{(z^2-1)^2}{z^2},\>\> g=\dfrac{c}{w},
\] 
and moreover, we find $\eta=c''\frac{w}{z}dz$ for some constants $c$, $c''$. 
However, this surface has a transformation $\sigma\in \Delta_0$ defined by $\sigma (z,\,w)=(\frac{1}{z},\,w)$, and we have 
$\sigma(p_1)=p_2$. This contradicts our assumption. 

\medskip

\noindent
\textbf{The case $p_1$ can be transformed to $p_2$}

Suppose that there exists $\sigma \in \Delta_0$ such that 
$\sigma(p_1)=p_2$. 
By Table~\ref{tb:p1-to-p2}, we consider two cases, that is, 
the case $s=2$ and the case $s=3$. 
Note that $\sigma(p_2)=p_1$ and $\sigma\in \Delta_0\setminus \mathcal{R}$. 

\medskip 

\noindent
\underline{The case $s=2$} 

By Table~\ref{tb:p1-to-p2}, every $q_i$ must be branch 
points of $\pi_{\Delta_0}$ with the ramified index $2(\gamma+1)-1$. 
Hence, $\sigma (q_i)=q_i$ for $i=1,\,2$, and moreover, $\sigma$ induces a degree $2$ transformation 
$\sigma':\overline{M}_{\gamma}/\mathcal{R}\to \overline{M}_{\gamma}/\Delta_0$, 
that is a transformation on $S^2$ 
and the $q'_i$'s are two fixed points of $\sigma'$  
(see Figure~\ref{fg:M-3}).  

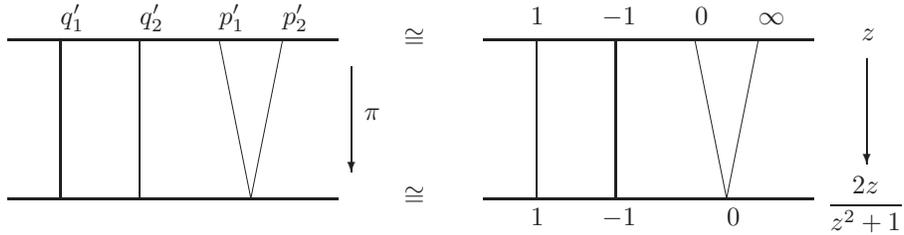
\begin{figure}[htbp] %%%%%%%%%%%%%%%%%%%%%%%%%%%%%%%%%%%%%%%%%%%%%%%%%%%
\begin{center}

\begin{picture}(340,100)

\thicklines

\put(0,80){\line(1,0){125}} \put(0,20){\line(1,0){125}} 

\put(20,86){$q'_1$} \put(50,86){$q'_2$} \put(80,86){$p'_1$}
\put(104,86){$p'_2$}

\put(135,50){$\pi$}

\put(180,80){\line(1,0){125}} \put(180,20){\line(1,0){125}} 

\put(198,86){$1$} \put(225,86){$-1$} \put(260,86){$0$}
\put(284,86){$\infty$}
\put(198,10){$1$} \put(225,10){$-1$} \put(272,10){$0$}

\put(323,80){$z$} \put(310,15){$\dfrac{2 z}{z^2+1}$}

\put(150,78){$\cong$} \put(150,18){$\cong$}

\thinlines

\put(20,80){\line(0,-1){60}} \put(50,80){\line(0,-1){60}} 
\put(80,80){\line(1,-5){12}} \put(104,80){\line(-1,-5){12}}

\put(200,80){\line(0,-1){60}} \put(230,80){\line(0,-1){60}} 
\put(260,80){\line(1,-5){12}} \put(284,80){\line(-1,-5){12}}

\put(130,70){\vector(0,-1){40}} \put(325,73){\vector(0,-1){40}}
\end{picture}
\caption{The Riemann surface $\overline{M}_{1}$ (for the case $s=2$).}
\label{fg:M-3}
\end{center}
\end{figure} %%%%%%%%%%%%%%%%%%%%%%%%%%%%%%%%%%%%%%%%%%%%%%%%%%%%%%%%%%%%

By suitable variables $(z,\,w)$, we have $\sigma (z,\,w)=\left(\dfrac{1}{z},\,*\right)$, 
$\sigma'(z)=\dfrac{2 z}{z^2+1}$, $p'_1=0$，$p'_2=\infty$，$q'_1=1$，$q'_2=-1$. 
Also, $\overline{M}_{\gamma}$ is given by 
$w^2=z^{m_1 h_1}(z-1)^{m_2 h_2}(z+1)^{m_3 h_3}$. 

We consider the Gauss map. 
Essentially, we treat the two cases 
as in Figure~\ref{fg:g-3}: 

\begin{figure}[htbp] %%%%%%%%%%%%%%%%%%%%%%%%%%%%%%%%%%%%%%%%%%%%%%%%%%%
\begin{center}
\begin{picture}(340,90)

\thicklines

\put(20,3){\line(1,0){120}} \put(180,3){\line(1,0){120}}

\put(55,65){\circle*{2}}\put(55,45){\circle*{2}}\put(55,25){\circle*{2}}
\put(105,45){\circle*{2}}
\put(215,60){\circle*{2}} \put(215,30){\circle*{2}}
\put(265,60){\circle*{2}} \put(265,30){\circle*{2}}

\put(42,65){$q_1$}\put(42,45){$p_1$}\put(42,25){$p_2$}
\put(92,45){$q_2$}
\put(202,60){$p_1$}\put(202,30){$p_2$}
\put(252,60){$q_1$}\put(252,30){$q_2$}

\put(52,-7){$0$}\put(102,-7){$\infty$}
\put(212,-7){$0$}\put(262,-7){$\infty$}

\thinlines

\put(155,0){$S^2$} \put(154,55){$\overline{M}$} \put(164,35){$g$}

\put(55,3){\line(0,1){80}} \put(105,3){\line(0,1){80}}
\put(215,3){\line(0,1){80}} \put(265,3){\line(0,1){80}}

\put(159,50){\vector(0,-1){37}}

\end{picture}
\caption{The possibilities of the Gauss map.}
\label{fg:g-3}
\end{center}
\end{figure}
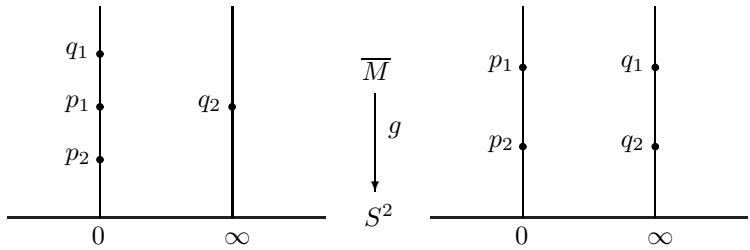 %%%%%%%%%%%%%%%%%%%%%%%%%%%%%%%%%%%%%%%%%%%%%%%%%%%%%%%%%%%%

For the case in the left hand side of Figure~\ref{fg:g-3}, 
the divisor of $g$ is given by 
\[
(g)=
\begin{cases}
(\gamma+2-2N)q_1+N p_1+N p_2-(\gamma+2)q_2 \\
(\gamma+2-2N)q_1+N p_1+N p_2-q_2-Q-R(Q)-\cdots-R^{\gamma}(Q),
\end{cases}
\]
where $N>0$ and $\gamma+2-2N>0$. Since $\sigma (q_2)=q_2$, $\sigma$ leaves $\{$the poles of $g\}$ invariant. 
For the latter case, if we take $Q^*:=\sigma (Q)$, then $Q^*$ must be a pole of $g$ which is distinct from the $Q_i$'s. 
This leads to a contradiction, and so we only consider the former case. 
In this case, we see 
\[(\eta)=-3 p_1-3 p_2+(2\gamma+4)q_2.\]
Then 
\[(g\eta)=(\gamma+2-2N)q_1+(N-3) p_1+(N-3) p_2+(\gamma+2)q_2.\]
It follows that $N-3<0$. Also $N>0$ yields $N=1,\,2$. 

Thus the divisor of $g^{\gamma+1}$ coincides with that of $\dfrac{z^{N}(z-1)^{\gamma+2-2 N}}{(z+1)^{\gamma+2}}$. 
Therefore, $\overline{M}_{\gamma}$ and $g$ can be rewritten as 
\[
w^{\gamma+1}=\dfrac{z^{N}(z-1)^{\gamma+2-2 N}}{(z+1)^{\gamma+2}} \>\>(N=1,\,2), 
\quad g=c\,w
\]
for some constant $c$, 
and $R(z,\,w)=(z,\,e^{\frac{2\pi}{\gamma+1} i}w)$, $\sigma(z,\,w)=(\frac{1}{z},\,w)$. Furthermore, by similar arguments, $\eta$ can be 
obtained by 
\[\eta^{\gamma+1}=c'\, \dfrac{(z+1)^{\gamma+4}}{z^{\gamma+3}(z-1)^{\gamma}}
(dz)^{\gamma+1}=
\begin{cases}
c'\,\dfrac{z+1}{z-1}\,\left(\dfrac{(z-1)dz}{z(z+1)w^2}\right)^{\gamma+1} \>\>(N=1) \\
c'\,\left(\dfrac{z+1}{z-1}\right)^2\,\left(\dfrac{dz}{z w}\right)^{\gamma+1} 
\>\>(N=2)
\end{cases}
\]
for some constant $c'$. So we may exclude this case, like the previous case. 

Next we consider the case in the right hand side of Figure~\ref{fg:g-3}. 
Then the divisor of $g$ is obtained by 
\[
(g)=\dfrac{\gamma+2}{2} p_1+\dfrac{\gamma+2}{2} p_2-N q_1-(\gamma+2-N)q_2, 
\]
where $N>0$ and $\gamma+2-N>0$. In this case, the divisor of $\eta$ is given by 
\[(\eta)=-3 p_1-3 p_2+ 2 N q_1+(2\gamma+4-2N)q_2. \]
Thus 
\[(g\eta)=\dfrac{\gamma-4}{2} p_1+\dfrac{\gamma-4}{2} p_2+  N q_1+(\gamma+2-N)q_2. \]
Hence $\gamma-4<0$ and so $\gamma=2$. Moreover, 
from the inequality $\gamma+2-N>0$, we 
obtain $N=1,\,2,\,3$. 
If $N=1$, then $\gamma+2-N$ and $\gamma+1$ are not coprime. 
Also, if $N=3$, then $N$ and $\gamma+1$ are not coprime.  
So we have $N=2$. 
As a result, the divisor of $g^{\gamma+1}(=g^3)$ coincides with that of 
$\dfrac{z^{\frac{\gamma+2}{2}}}{(z-1)^N (z+1)^{\gamma+2-N}}\left(=\dfrac{z^2}{(z-1)^2 (z+1)^2}\right)$. 
Therefore, $\overline{M}_2$ and $g$ can be rewritten as 
\[
w^{3}=\dfrac{(z-1)^2 (z+1)^2}{z^2},\quad g=\dfrac{c}{w}
\]
for some constant $c$, and $R(z,\,w)=(z,\,e^{\frac{2\pi}{3} i}w)$, $\sigma(z,\,w)=(\frac{1}{z},\,w)$. 
Also, we have $\eta=c'\frac{w}{z} dz$ for some constant $c'$. 
We shall prove that this case does not occur in \S\ref{uniqueness-thm}. 

\medskip 

\noindent
\underline{The case $s=3$}

By Table~\ref{tb:p1-to-p2}, the $p_i$'s and $q_i$'s must be branch 
points of $\pi_{\Delta_0}$ with the ramified index $(\gamma+1)-1$. 
As a result, there exist two 
sets $\{r_1^{(1)},\,\dots,\,r_{\gamma+1}^{(1)}\}$, 
$\{r_1^{(2)},\,\dots,\, r_{\gamma+1}^{(2)}\}$ of branch points with the ramified 
index $2-1$ of $\pi_{\Delta_0}$ 
satisfying $\pi_{\Delta_0}(r_i^{(1)})=\pi_{\Delta_0}(r_j^{(1)})$ and 
$\pi_{\Delta_0}(r_i^{(2)})=\pi_{\Delta_0}(r_j^{(2)})$ for $1\leq i,\,j\leq \gamma+1$. Note that 
$r_i^{(1)}$ and $r_i^{(2)}$ are distinct from the $p_i$'s and $q_i$'s. 
Hence, $\sigma (q_1)=q_2$ and $\sigma(q_2)=q_1$, 
and moreover, $\sigma$ induces a degree $2$ transformation 
$\sigma':\overline{M}_{\gamma}/\mathcal{R}\to \overline{M}_{\gamma}/\Delta_0$, that is, a transformation on $S^2$ 
and $r'_i:=\pi_{\mathcal{R}}(r_j^{(i)})$'s are two branch points of $\sigma'$ 
(see Figure~\ref{fg:M-4}). 

\begin{figure}[htbp] %%%%%%%%%%%%%%%%%%%%%%%%%%%%%%%%%%%%%%%%%%%%%%%%%%%
\begin{center}
\begin{picture}(340,100)

\thicklines

\put(0,80){\line(1,0){120}} \put(0,20){\line(1,0){120}} 

\put(10,86){$p'_1$} \put(34,86){$p'_2$} \put(50,86){$q'_1$}
\put(74,86){$q'_2$} \put(90,86){$r'_1$} \put(110,86){$r'_2$}

\put(130,50){$\sigma'$}

\put(170,80){\line(1,0){120}} \put(170,20){\line(1,0){120}} 

\put(180,86){$0$} \put(200,86){$\infty$} \put(220,86){$1$}
\put(240,86){$a$}
\put(190,8){$0$} \put(225,8){$\frac{2 a}{a+1}$} 
\put(254,86){$\sqrt{a}$} \put(274,86){$-\sqrt{a}$}

\put(310,80){$z$} \put(297,15){$\dfrac{2 a z}{z^2+a}$}

\put(145,78){$\cong$} \put(145,18){$\cong$}

\thinlines

\put(10,80){\line(1,-5){12}} \put(34,80){\line(-1,-5){12}}
\put(90,80){\line(0,-1){60}} \put(110,80){\line(0,-1){60}} 
\put(50,80){\line(1,-5){12}} \put(74,80){\line(-1,-5){12}}

\put(180,80){\line(1,-5){12}} \put(204,80){\line(-1,-5){12}}
\put(260,80){\line(0,-1){60}} \put(280,80){\line(0,-1){60}} 
\put(220,80){\line(1,-5){12}} \put(244,80){\line(-1,-5){12}}

\put(125,70){\vector(0,-1){40}} \put(313,75){\vector(0,-1){45}}
\end{picture}
\caption{The Riemann surface $\overline{M}_{1}$ (the case $s=3$).}
\label{fg:M-4}
\end{center}
\end{figure}
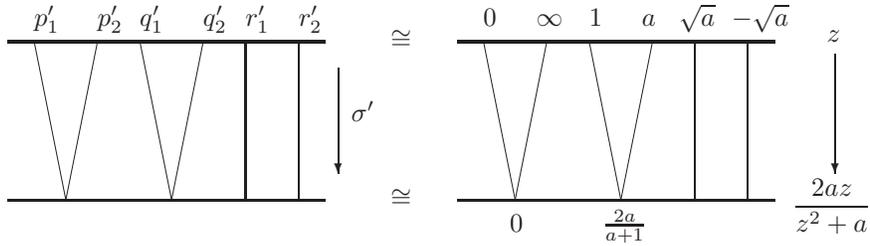 %%%%%%%%%%%%%%%%%%%%%%%%%%%%%%%%%%%%%%%%%%%%%%%%%%%%%%%%%%%%

Choosing suitable variables $(z,\,w)$, for $a\in\mathbb{C} \setminus \{0\}$, we have 
$\sigma (z,\,w)=\left(\dfrac{a}{z},\,*\right)$，$\sigma'(z)=\dfrac{2 a z}{z^2+a}$, 
$p'_1=0$, $p'_2=\infty$, $q'_1=1$, $q'_2=a$, 
$r'_1=\sqrt{a}$, $r'_2=-\sqrt{a}$. Also, $\overline{M}_{\gamma}$ is given by 
$w^{\gamma+1}=z^{m_1 h_1}(z-1)^{m_2 h_2}(z-a)^{m_2 h_3}$ 
(see Figure~\ref{fg:M-4}). 

We now consider the Gauss map. 
Essentially, we consider the two cases 
as in Figure~\ref{fg:g-4}: 

\begin{figure}[htbp] %%%%%%%%%%%%%%%%%%%%%%%%%%%%%%%%%%%%%%%%%%%%%%%%%%%
\begin{center}
\begin{picture}(340,90)

\thicklines

\put(20,3){\line(1,0){120}} \put(180,3){\line(1,0){120}}

\put(55,60){\circle*{2}} \put(55,30){\circle*{2}}
\put(105,60){\circle*{2}} \put(105,30){\circle*{2}}
\put(215,60){\circle*{2}} \put(215,30){\circle*{2}}
\put(265,60){\circle*{2}} \put(265,30){\circle*{2}}

\put(42,60){$p_1$}\put(42,30){$q_1$}
\put(92,60){$p_2$}\put(92,30){$q_2$}
\put(202,60){$p_1$}\put(202,30){$p_2$}
\put(252,60){$q_1$}\put(252,30){$q_2$}

\put(52,-7){$0$}\put(102,-7){$\infty$}
\put(212,-7){$0$}\put(262,-7){$\infty$}

\thinlines

\put(155,0){$S^2$} \put(154,55){$\overline{M}$} \put(164,35){$g$}

\put(55,3){\line(0,1){80}} \put(105,3){\line(0,1){80}}
\put(215,3){\line(0,1){80}} \put(265,3){\line(0,1){80}}

\put(159,50){\vector(0,-1){37}}

\end{picture}
\caption{The possibilities of the Gauss map.}
\label{fg:g-4}
\end{center}
\end{figure}
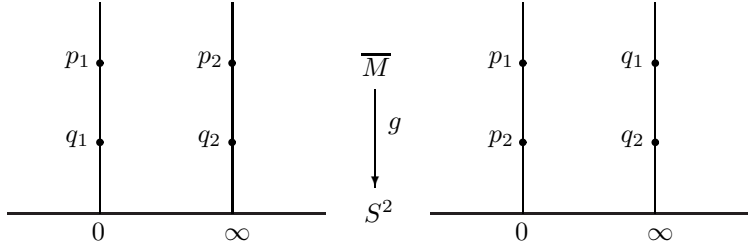 %%%%%%%%%%%%%%%%%%%%%%%%%%%%%%%%%%%%%%%%%%%%%%%%%%%%%%%%%%%%

Then the divisor of $g$ is obtained by 
\[(g)=
\begin{cases}
(\gamma+2-N) p_1 +N q_1- (\gamma+2-N) p_2 -N q_2 
&({\rm the }\>\>{\rm LHS }\>\>{\rm case})\\
\dfrac{\gamma+2}{2}p_1+\dfrac{\gamma+2}{2}p_2
-\dfrac{\gamma+2}{2}q_1-\dfrac{\gamma+2}{2}q_2 
&({\rm the }\>\>{\rm RHS }\>\>{\rm case}),
\end{cases}
\]
where $N>0$ and $\gamma+2-N>0$. The divisor of $\eta$ is given by 
\[
(\eta)=
\begin{cases}
-3 p_1 +(2\gamma-2 N+1)p_2+2 N q_2\\
-3 p_1-3 p_2 +(\gamma+2) q_1+(\gamma+2) q_2.
\end{cases}
\]
So the divisor of $g\eta$ is obtained by 
\[
(g\eta)=
\begin{cases}
(\gamma-N-1) p_1 +(\gamma- N-1)p_2+N q_1+ N q_2 \\
\dfrac{\gamma-4}{2} p_1+\dfrac{\gamma-4}{2} p_2+ \dfrac{\gamma+2}{2} q_1
+\dfrac{\gamma+2}{2} q_2.
\end{cases}
\]
For the former case, if $\gamma-N-1\geq 0$, then $g\eta$ is holomorphic. Thus 
$\gamma-N-1<0$ holds. Also, the inequality $\gamma+2-N>0$ yields 
$N=\gamma,\,\gamma+1$. For the latter case, $\gamma-4\geq 0$ cannot hold, 
and hence $\gamma=2$. As a consequence, 
\begin{align*}
g^{\gamma+1}=
\begin{cases}
c'\,z^{\gamma+2-N}\left(\dfrac{z-1}{z-a}\right)^{N} 
& ({\rm the}\>\>{\rm LHS}\>\>{\rm case})\\
c'\,\left(\dfrac{z}{(z-1)(z-a)}\right)^{\frac{\gamma+2}{2}}
& ({\rm the}\>\>{\rm RHS}\>\>{\rm case})
\end{cases}
\end{align*}
for some constant $c'$. 
If $N=\gamma+1$, then $\gamma+1$ and $N$ are not coprime. 
Thus $N=\gamma$ follows. 
Therefore, $\overline{M}_{\gamma}$ and $g$ can be rewritten as 
\[
\begin{cases}
w^{\gamma+1}=z^{2}\left(\dfrac{z-1}{z-a}\right)^{\gamma},
\>\> g=c\,w & ({\rm the}\>\>{\rm LHS}\>\>{\rm case})\\
w^3=\left(\dfrac{(z-1)(z-a)}{z}\right)^2, \>\> 
g=\dfrac{c}{w} & ({\rm the}\>\>{\rm RHS}\>\>{\rm case})
\end{cases}
\]
and 
\[
R(z,\,w)=(z,\,e^{\frac{2\pi}{\gamma+1} i}w), \quad 
\sigma (z,\,w)=
\begin{cases}
\left(\dfrac{a}{z},\,\dfrac{a^{\frac{\gamma+2}{\gamma+1}}}{w}\right) & ({\rm the}\>\>{\rm LHS}\>\>{\rm case}) \\
\left(\dfrac{a}{z},\,w\right) & ({\rm the}\>\>{\rm RHS}\>\>{\rm case}).
\end{cases}
\] 
Also, for some constant $c'$, we have 
\[
\eta=
\begin{cases}
c' \dfrac{dz}{zw} & ({\rm the}\>\>{\rm LHS}\>\>{\rm case})\\[6pt]
c' \dfrac{w}{z}dz & ({\rm the}\>\>{\rm RHS}\>\>{\rm case}).
\end{cases}
\]
$\Delta\setminus\Delta_0\neq \varnothing$ implies that 
there exists a degree $2$ antiholomorphic transformation. 
Hence $a\in\mathbb{R}$ and the antiholomorphic transformation 
can be represented by $(z,\,w)\mapsto (\overline{z},\,\overline{w})$. 
The former case corresponds to our result in \S\ref{sec:evengenus}, and 
the latter case is considered in \S\ref{uniqueness-thm}. 
Note that the case $t=1$ does not occur. 

\medskip

\subsection{Well-definedness}\label{uniqueness-thm}

In this subsection, we consider the well-definedness for the following two cases: 
\begin{center}
 \begin{tabular}{|c|c|c|c|} \hline
$\overline{M}_{\gamma} $   & $g$ & $\eta$ & symmetries  \\ \hline\hline
$w^2=z(z^2-1)$  &  $c w$  & $c'\dfrac{dz}{zw}$ &  $R (z,\,w)=(-z,\,i w)$ \\ \hline
$w^3=\dfrac{(z-1)^2 (z-a)^2}{z^2}$  &  $\dfrac{c}{w}$  & $c' \dfrac{w}{z}dz$  & $R (z,\,w)=(z,\,e^{\frac{2\pi}{3}i}w)$,    
$\sigma (z,\,w)=(\frac{a}{z},\,w)$ \\[5pt] \hline
 \end{tabular} 
\end{center}
Note that $c,\,c'\in \mathbb{C}\setminus\{0\}$, $a\in \mathbb{R}\setminus \{0,\,1\}$.  
$M$ is given by
\[
M=\begin{cases}
\overline{M}_{1} \setminus \{(0,\,0),\,(\infty,\,\infty)\} & ({\rm the}\>\>{\rm former}\>\>{\rm case}) \\ 
\overline{M}_{2} \setminus \{(0,\,\infty),\,(\infty,\,\infty)\} & ({\rm the}\>\>{\rm latter}\>\>{\rm case}).
\end{cases}
\]
The case $a=-1$ corresponds to the surface which we treat for the case $s=2$. 
Note that the Weierstrass data $(e^{i\theta}g,e^{-i\theta}\eta)$ produces 
the same minimal surface as $(g,\eta)$ rotated by an angle $\theta$ 
around the $x_3$-axis. 
So after a suitable rotation of the surface, we may assume $c\in\mathbb{R}_+$.
Also, multiplying a positive real number into $\eta$ is just a homothety, so 
we may assume that $|c'|=1$. 

Our claim is that all cases do not occur. 

\medskip

\noindent
\underline{The former case}

First we consider $\Phi=\,^t (\Phi_1,\,\Phi_2,\,\Phi_3)$ in Theorem~\ref{th:w-rep}: 
\[
\Phi_1=\left(\dfrac{1}{w}-c^2w\right)c'\dfrac{dz}{z},\> \Phi_2=i \left(\dfrac{1}{w}+c^2w\right)c'\dfrac{dz}{z},\> \Phi_3=2c c' \dfrac{dz}{z}.
\]
For the residue of $\Phi_3$ at $z=0$ to be real, we see that $c'=\pm 1$. 
We may choose $c'=1$. 
We shall use the notation in the proof of Theorem~\ref{existence1-thm} for $\gamma=1$. 

Straightforward calculation yields 
\[
\dfrac{dz}{zw} -d\left( \dfrac{2w}{z} \right) = -\dfrac{z}{w} dz.
\]
Thus we have 
\begin{align*}
\oint_{\ell'} \eta &= - \oint_{\ell'} \dfrac{z}{w} dz = - \oint_{\ell} \dfrac{z}{w} dz = -2i  \int_0^1 \sqrt{\dfrac{t}{1-t^2}} dt, \\
\oint_{\ell'} g^2 \eta &= c^2  \oint_{\ell} \dfrac{w}{z} dz =  -2i c^2 \int_0^1 \sqrt{\dfrac{1-t^2}{t}} dt. 
\end{align*}
\eqref{eq:period1} implies 
\[
- \int_0^1 \sqrt{\dfrac{t}{1-t^2}} dt=c^2 \int_0^1 \sqrt{\dfrac{1-t^2}{t}} dt. 
\]
So we have $c^2<0$ and this contradicts $c>0$.
\medskip

\noindent
\underline{The latter case}

First we consider $\Phi=\,^t (\Phi_1,\,\Phi_2,\,\Phi_3)$ in Theorem~\ref{th:w-rep}: 
\[
\Phi_1=\left(w-\dfrac{c^2}{w}\right)c'\dfrac{dz}{z},\> \Phi_2=i \left(w+\dfrac{c^2}{w}\right)c'\dfrac{dz}{z},\> \Phi_3=2c c' \dfrac{dz}{z}.
\]
For the residue of $\Phi_3$ at $z=0$ to be real, we see that $c'=\pm 1$. 
We may choose $c'=1$. 

We shall show that for any $c>0$ and $a\in \mathbb{R}\setminus \{0,\,1\}$, 
the period condition \eqref{eq:period} cannot be satisfied. 

A straightforward calculation yields  
\begin{align}\label{exact-form2-1}
\eta +\frac{3}{2}dw=\left(\frac{w}{z-1}+\frac{w}{z-a}\right)dz.
\end{align}
Note that the right hand side of this equation is a holomorphic differential on 
$\overline{M}_2\setminus\{(\infty,\infty)\}$. 

We now consider the following three cases: 
$a>1$, $0<a<1$, $a<0$. \\
$\,$\\
(i) The case $a>1$: 

We set 
\begin{align*}
\ell=&\left\{ (z,\,w)=\left(t,\,\sqrt[3]{\frac{(1-t)^2(a-t)^2}{t^2}}\right) \>\Bigg|\>0\leq t\leq 1 \right\} \\
&\quad \cup \left\{ (z,\,w)=\left(-t,\,e^{\frac{2}{3}\pi i}\sqrt[3]{\frac{(1+t)^2(a+t)^2}{t^2}} \right)\>\Bigg|\> -1\leq t\leq 0 \right\}. 
\end{align*}
From \eqref{exact-form2-1}, we have 
\begin{align}
\label{path-int1} 
\oint_\ell\eta
&=-(1-e^{2\pi i/3})\int_0^1 
  \left(\sqrt[3]{\frac{(a-t)^2}{t^2(1-t)}}
       +\sqrt[3]{\frac{(1-t)^2}{t^2(a-t)}}\right)dt, \\
\label{path-int2} 
\int_{\ell} g^2\eta 
&= c^2(1-e^{-2\pi i/3}) 
  \int_0^1 \dfrac{dt}{\sqrt[3]{t(1-t)^2(a-t)^2}}.  
\end{align}
Thus \eqref{eq:period1} is equivalent to 
\[
-\int_0^1 \left(\sqrt[3]{\frac{(a-t)^2}{t^2(1-t)}}
        +\sqrt[3]{\frac{(1-t)^2}{t^2(a-t)}}\right)dt
=c^2\int_0^1 \dfrac{dt}{\sqrt[3]{t(1-t)^2(a-t)^2}}.
\]
But this is impossible because the left hand side is a negative real number 
and the right hand side is a positive real number.
\\
$\,$\\
(ii) The case $0<a<1$: 

We set 
\begin{align*}
\ell=&\left\{ (z,\,w)=\left(at,\,\sqrt[3]{\frac{(at-1)^2(1-t)^2}{t^2}}\right) \>\Bigg|\>0\leq t\leq 1 \right\} \\
&\quad \cup \left\{ (z,\,w)=\left(-at,\,e^{\frac{2}{3}\pi i}\sqrt[3]{\frac{(at+1)^2(1+t)^2}{t^2}} \right)\>\Bigg|\> -1\leq t\leq 0 \right\}. 
\end{align*}
From \eqref{exact-form2-1}, we have 
\begin{align}
\label{path-int2-1} 
\oint_\ell\eta
&=-(1-e^{2\pi i/3})\int_0^1 
  \left(a\sqrt[3]{\frac{(1-t)^2}{t^2(1-at)}}
       +\sqrt[3]{\frac{(1-at)^2}{t^2(1-t)}}\right)dt, \\
\label{path-int2-2} 
\int_{\ell} g^2\eta 
&= c^2(1-e^{-2\pi i/3}) 
  \int_0^1 \dfrac{dt}{\sqrt[3]{t(1-t)^2(1-at)^2}}.  
\end{align}
Thus \eqref{eq:period1} is equivalent to 
\[
-\int_0^1\left(a\sqrt[3]{\frac{(1-t)^2}{t^2(1-at)}}
         +\sqrt[3]{\frac{(1-at)^2}{t^2(1-t)}}\right)dt
=c^2\int_0^1 \dfrac{dt}{\sqrt[3]{t(1-t)^2(1-at)^2}},
\]
but again this is impossible by the same reason as in the case (i). 
\\
$\,$\\
(iii) The case $a<0$: 

We set 
\begin{align*}
\ell=&\left\{ (z,\,w)=\left(at,\,\sqrt[3]{\frac{(1-at)^2(1-t)^2}{t^2}}\right) \>\Bigg|\>0\leq t\leq 1 \right\} \\
&\quad \cup \left\{ (z,\,w)=\left(-at,\,e^{\frac{2}{3}\pi i}\sqrt[3]{\frac{(1+at)^2(1+t)^2}{t^2}} \right)\>\Bigg|\> -1\leq t\leq 0 \right\}, \\
\ell'=&\left\{ (z,\,w)=\left(t,\,\sqrt[3]{\frac{(1-t)^2(t-a)^2}{t^2}}\right) \>\Bigg|\>0\leq t\leq 1 \right\} \\
&\quad \cup \left\{ (z,\,w)=\left(-t,\,e^{\frac{2}{3}\pi i}\sqrt[3]{\frac{(1+t)^2(t+a)^2}{t^2}} \right)\>\Bigg|\> -1\leq t\leq 0 \right\}. 
\end{align*}
From \eqref{exact-form2-1}, we have 
\begin{align}
\label{path-int3-1} 
\oint_\ell\eta
&=-(1-e^{2\pi i/3})\int_0^1 
  \left(a\sqrt[3]{\frac{(1-t)^2}{t^2(1-at)}}
       +\sqrt[3]{\frac{(1-at)^2}{t^2(1-t)}}\right)dt, \\
\label{path-int3-2} 
\int_{\ell} g^2\eta 
&= c^2(1-e^{-2\pi i/3}) 
  \int_0^1 \dfrac{dt}{\sqrt[3]{t(1-t)^2(1-at)^2}}.  
\end{align}
Thus \eqref{eq:period1} is equivalent to 
\begin{equation}\label{eq:per-a-neg1}
-\int_0^1\left(a\sqrt[3]{\frac{(1-t)^2}{t^2(1-at)}}
         +\sqrt[3]{\frac{(1-at)^2}{t^2(1-t)}}\right)dt
=c^2\int_0^1 \dfrac{dt}{\sqrt[3]{t(1-t)^2(1-at)^2}}. 
\end{equation}
The right hand side is clearly positive. 
So now we estimate the left hand side. 
\begin{align*}
(\mathrm{LHS})
&=-a\int_0^1\sqrt[3]{\frac{(1-t)^2}{t^2(1-at)}}dt
  -\int_0^1\sqrt[3]{\frac{(1-at)^2}{t^2(1-t)}}dt \\
&\le -a\int_0^1\sqrt[3]{\frac{(1-t)^2}{t^2}}dt
  -\int_0^1\sqrt[3]{\frac{1}{t^2(1-t)}}dt \\
&= -aB\left(\frac{1}{3},\frac{5}{3}\right)-B\left(\frac{1}{3},\frac{2}{3}\right) \\
&= -\left(\frac{2}{3}a+1\right)B\left(\frac{1}{3},\frac{2}{3}\right), 
\end{align*}
where $B(x,y)$ is the classical beta function mentioned in \S\ref{sec:evengenus}.
Hence, if $-\left(\frac{2}{3}a+1\right)\le 0$, 
\eqref{eq:per-a-neg1} never holds.  
That is, 
\begin{equation}\label{eq:claim1}
\mbox{if $-3/2\le a < 0$, \eqref{eq:per-a-neg1} never holds.} 
\end{equation}

On the ather hand, 
we have 
\begin{align}
\label{path-int4-1} 
\oint_{\ell'}\eta
&=-(1-e^{2\pi i/3})\int_0^1 
  \left(\sqrt[3]{\frac{(t-a)^2}{t^2(1-t)}}
       -\sqrt[3]{\frac{(1-t)^2}{t^2(t-a)}}\right)dt, \\
\label{path-int4-2} 
\int_{\ell'} g^2\eta 
&= c^2(1-e^{-2\pi i/3}) 
  \int_0^1 \dfrac{dt}{\sqrt[3]{t(1-t)^2(t-a)^2}},  
\end{align}
from \eqref{exact-form2-1}.
Thus \eqref{eq:period1} is equivalent to 
\begin{equation}\label{eq:per-a-neg2}
-\int_0^1\left(\sqrt[3]{\frac{(t-a)^2}{t^2(1-t)}}
        -\sqrt[3]{\frac{(1-t)^2}{t^2(t-a)}}\right)dt
=c^2\int_0^1 \dfrac{dt}{\sqrt[3]{t(1-t)^2(t-a)^2}}. 
\end{equation}
The right hand side is again positive. 
So now we estimate the left hand side again. 
\begin{align*}
(\mathrm{LHS})
&=-\int_0^1\sqrt[3]{\frac{(t-a)^2}{t^2(1-t)}}dt
  +\int_0^1\sqrt[3]{\frac{(1-t)^2}{t^2(t-a)}}dt \\
&\le -\int_0^1\sqrt[3]{\frac{(-a)^2}{t^2(1-t)}}dt
  +\int_0^1\sqrt[3]{\frac{(1-t)^2}{t^2(-a)}}dt \\
&= -(-a)^{2/3}B\left(\frac{1}{3},\frac{2}{3}\right)
   +(-a)^{-1/3}B\left(\frac{1}{3},\frac{5}{3}\right) \\
&= (-a)^{-1/3}\left(a+\frac{2}{3}\right)B\left(\frac{1}{3},\frac{2}{3}\right). 
\end{align*}
Hence, if $a+\frac{2}{3}\le 0$, 
\eqref{eq:per-a-neg2} never holds.  
That is, 
\begin{equation}\label{eq:claim2}
\mbox{if $a\le -2/3$, \eqref{eq:per-a-neg2} never holds.} 
\end{equation}

Combining \eqref{eq:claim1} and \eqref{eq:claim2}, 
the period condition cannot be solved.  

Main Theorem~\ref{main2} is an immediate consequence by the 
above arguments.

\section{Remaining problems} %%%%%%%%%%%%%%%%%%%%%%%%%%%%%%%%%%%%%%%%%%%
\label{sec:remaining-problems} %%%%%%%%%%%%%%%%%%%%%%%%%%%%%%%%%%%%%%%%%

In this section we introduce remaining problems related to this work. 

\subsection{The case that $\gamma$ odd and greater than 1} %%%%%%%%%%%%%%%%%%
\label{sec:oddgenus} %%%%%%%%%%%%%%%%%%%%%%%%%%%%%%%%%%%%%%%%%%%%%%%%%%%

For the case that the genus $\gamma$ is odd and greater than 1, 
a complete minimal surface of finite total curvature 
$f:M=\overline{M}_\gamma\setminus\{p_1,p_2\}\to\mathbb{R}^3$ 
which satisfies equality in \eqref{eq:ineq} is yet to be 
found.  
However, Matthias Weber \cite{We} has constructed 
the following examples numerically. 

\begin{example}[Weber]
Let $\gamma$ be a positive integer. Define 
\[
F_1(z;a_1,a_3,\ldots ,a_{2\gamma -1})=\prod_{i=1}^\gamma (z-a_{2i-1}), \;\;
F_2(z;a_2,a_4,\ldots ,a_{2\gamma})=\prod_{i=1}^\gamma (z-a_{2i}),
\]
where $1=a_1<a_2<\cdots <a_{2\gamma}$ are constants to be determined.  
Define a compact Riemann surface $\overline{M}_\gamma$ of genus $\gamma$ by 
\[
\overline{M}_\gamma
 =\left\{(z,w)\in (\mathbb{C}\cup\{\infty\})^2\,\left|\,
        w^2=z\dfrac{F_1(z;a_1,a_3,\ldots ,a_{2\gamma -1})}
                   {F_2(z;a_2,a_4,\ldots ,a_{2\gamma})}\right.\right\}. 
\]
We set 
\[
M=\overline{M}_\gamma\setminus\left\{(0,0), (\infty, \infty)\right\},\quad
g=c\frac{w}{z+1} \;\; (c>0), \quad \eta = \frac{(z+1)^2}{zw}dz.  
\]
Then there exist constants $c, a_2, a_3, \ldots , a_{2\gamma}$ 
such that \eqref{eq:period} holds.  (See Figure~\ref{fg:ce1-4}.) 
\end{example}

\begin{figure}[htbp] %%%%%%%%%%%%%%%%%%%%%%%%%%%%%%%%%%%%%%%%%%%%%%%%%%%
\begin{center}
\begin{tabular}{cccc}
 \includegraphics[width=.21\linewidth]{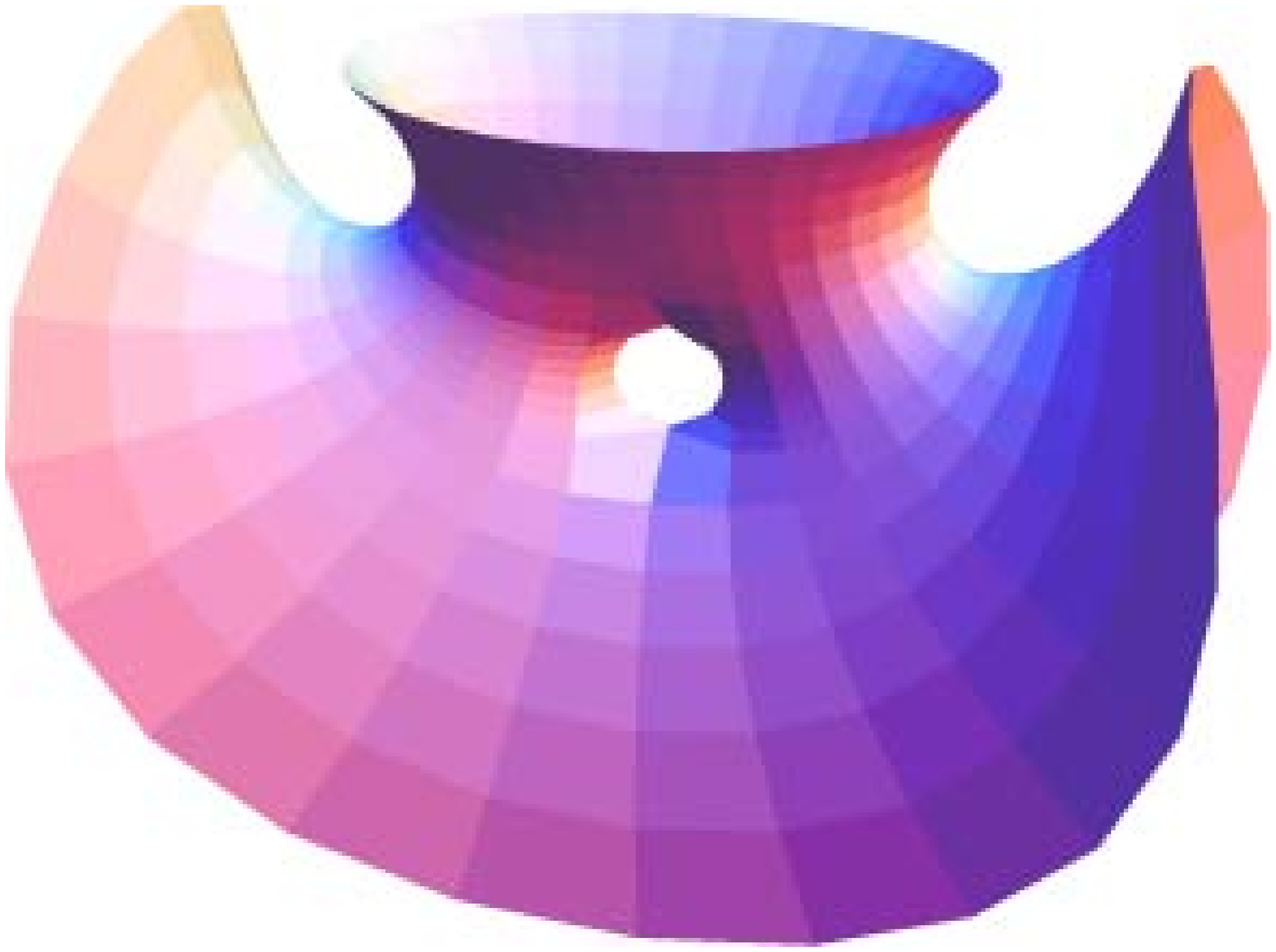} &
 \includegraphics[width=.21\linewidth]{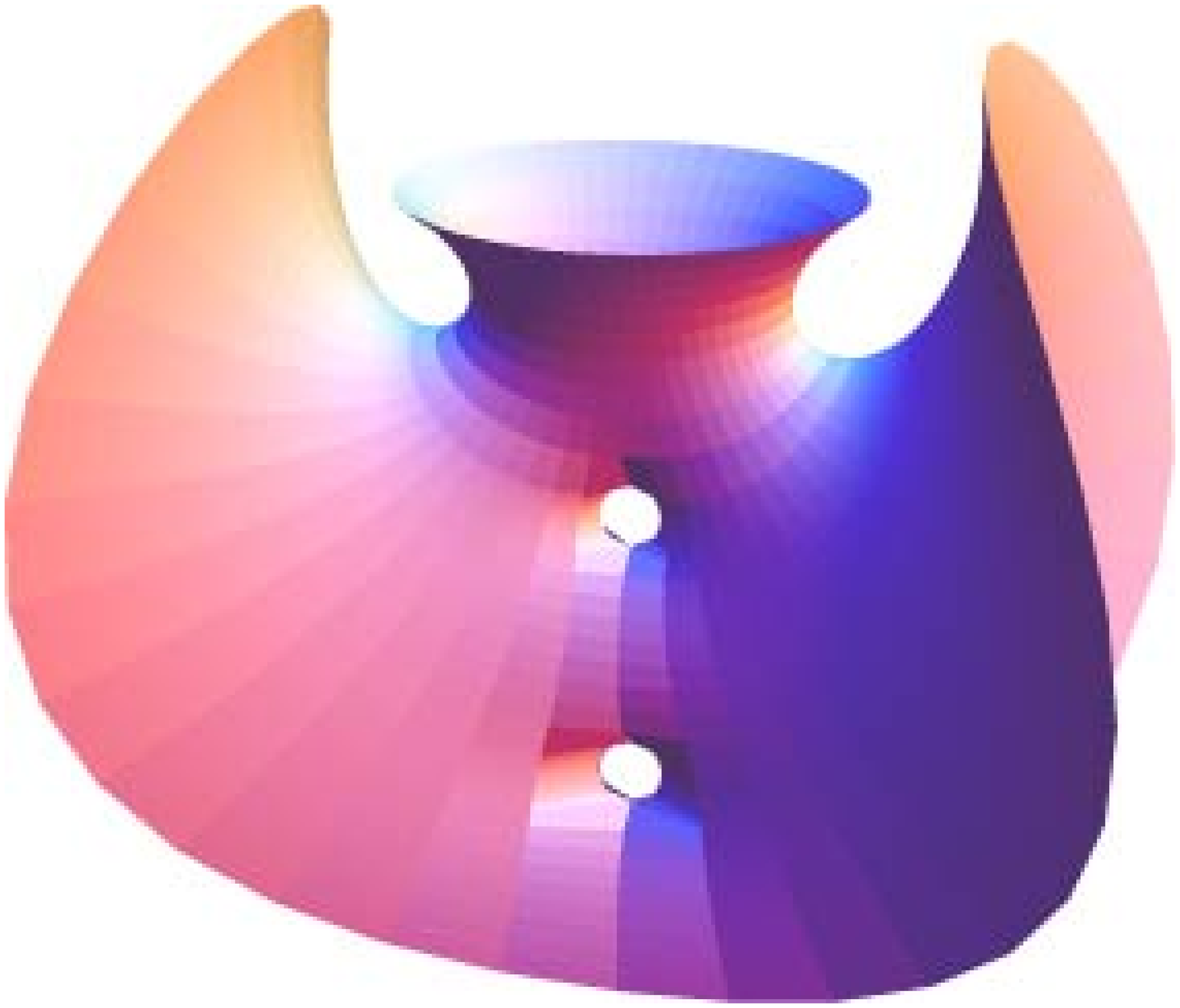} &
 \includegraphics[width=.21\linewidth]{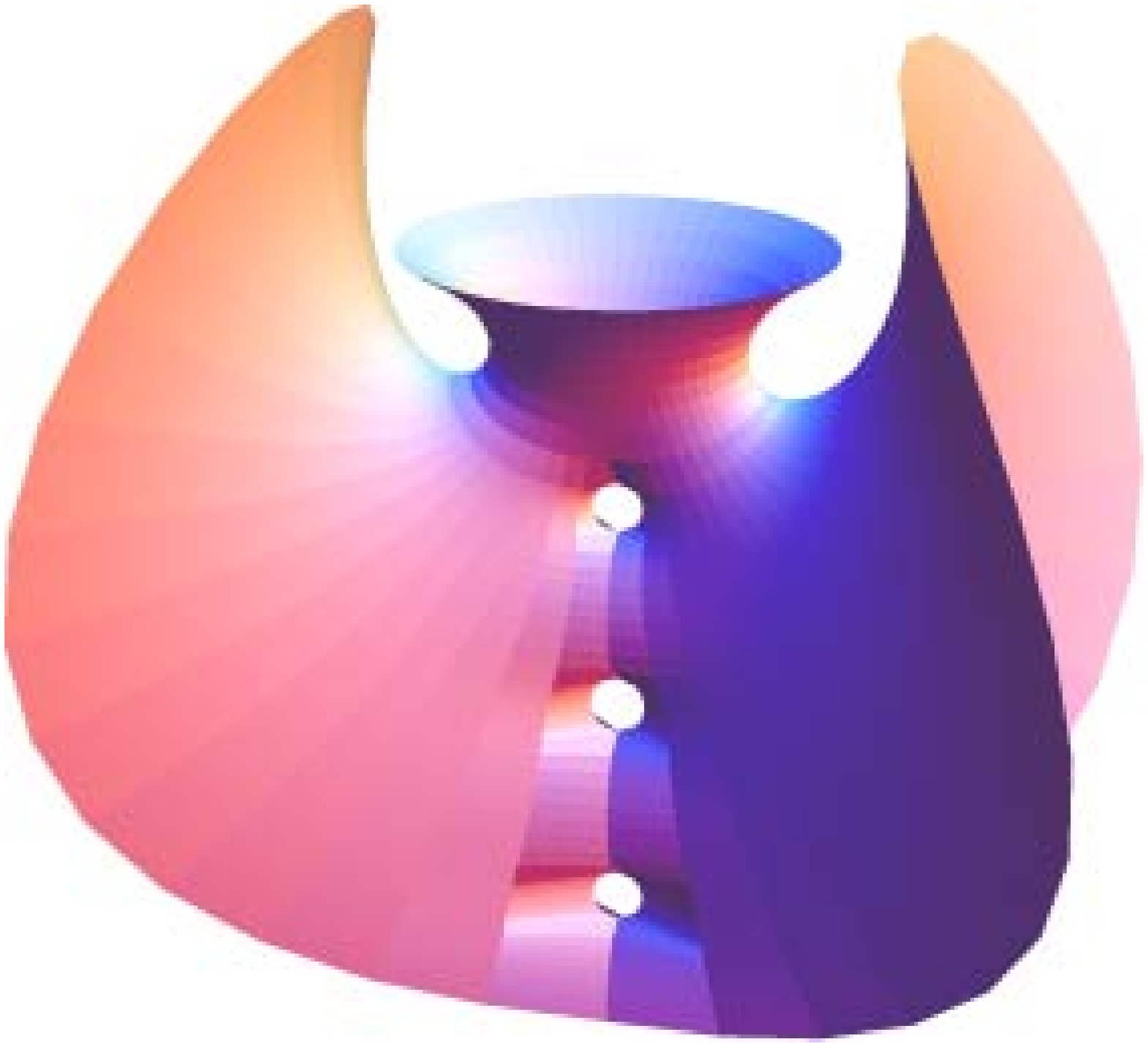} &
 \includegraphics[width=.21\linewidth]{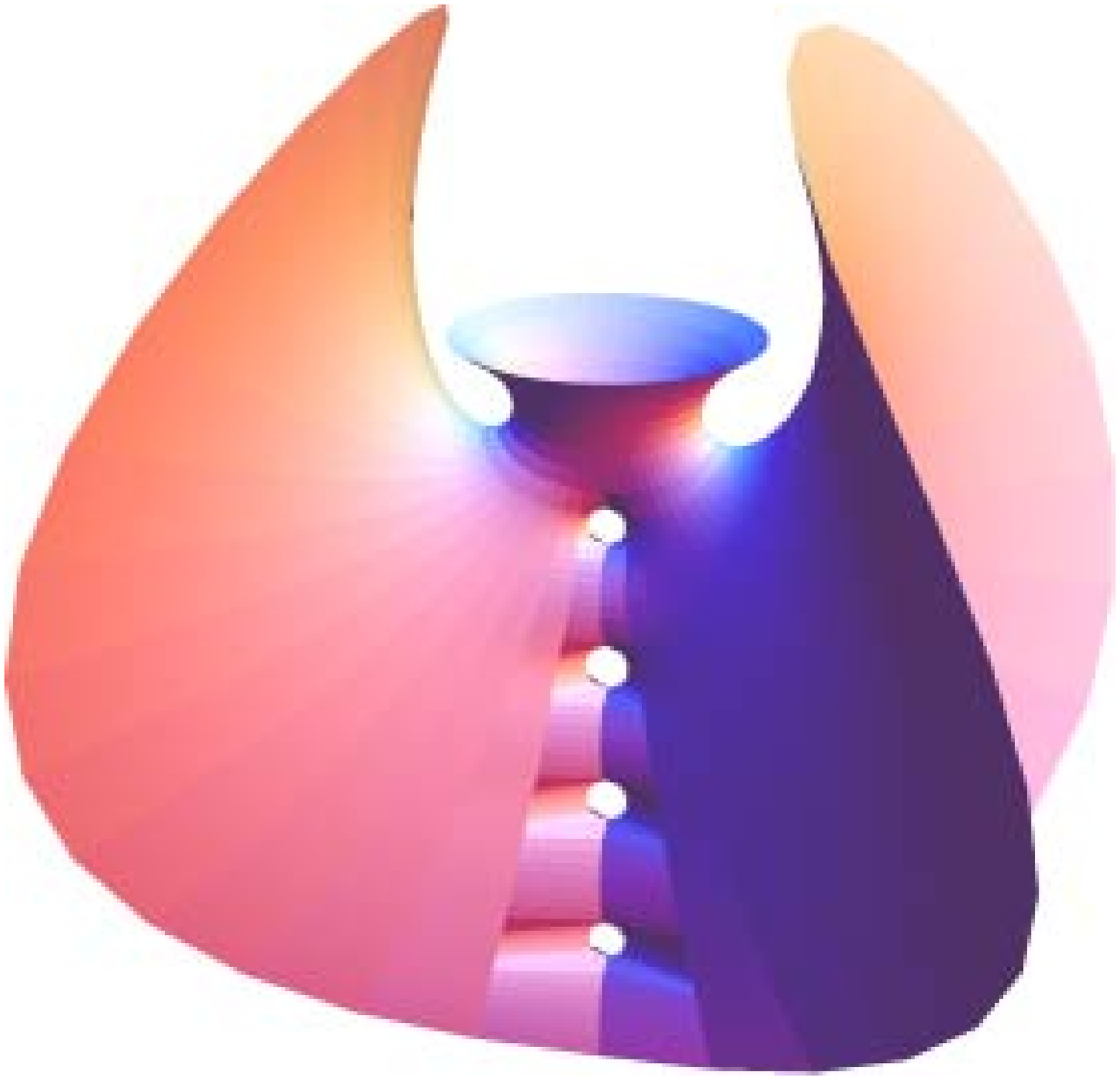} \\
 $\gamma =1$ & $\gamma =2$ & $\gamma =3$ & $\gamma =4$ 
\end{tabular}
\caption{Minimal surfaces of genus $\gamma$ with two ends 
         which satisfy $\deg (g)=\gamma +2$.}
\label{fg:ce1-4}
\end{center}
\end{figure} %%%%%%%%%%%%%%%%%%%%%%%%%%%%%%%%%%%%%%%%%%%%%%%%%%%%%%%%%%%

For $\gamma =1$, we can prove the existence of the surface rigorously. 
However, for other cases, since the surface does not have enough symmetry, 
the rigorous proof of the existence still remains an open problem.  

\subsection{Existence of non-orientable minimal surfaces} %%%%%%%%%%%%%%%%%%

Our work is devoted to minimal surfaces satisfying $\deg(g)= \gamma +2$.  
On the other hand, it is important to consider 
the existence of non-orientable minimal surfaces with $\deg(g)= \gamma +3$. 
Now, we review non-orientable minimal surfaces in $\mathbb{R}^3$. 

Let $f':M'\to\mathbb{R}^3$ be a minimal immersion of a non-orientable 
surface into $\mathbb{R}^3$. 
Then the oriented two sheeted covering space $M$ of $M'$ 
naturally inherits a Riemann surface structure and we have a canonical projection 
$\pi:M\to M'$.  
We can also define a map $I:M\to M$ such that $\pi\circ I=\pi$, 
which is an antiholomorphic involution on $M$ without fixed points.  
Here $M'$ can be identified with $M/\langle I\rangle$. 
In this way, if $f:M\to\mathbb{R}^3$ is a conformal minimal surface and there 
is an antiholomorphic involution $I:M\to M$ without fixed points so that 
$f\circ I=f$, 
then we can define a non-orientable minimal surface 
$f':M'=M/\langle I\rangle\to\mathbb{R}^3$.  
Conversely, every non-orientable minimal surface 
is obtained in this procedure. 

Suppose that $f':M'=M/\langle I\rangle\to\mathbb{R}^3$ is complete 
and of finite total curvature.  
Then, 
we can apply Theorems~\ref{th:huber-osserman} and \ref{th:oss-ineq} 
to the conformal minimal immersion $f:M\to\mathbb{R}^3$. 
Furthermore, we have a stronger restriction on the topology of $M'$ or $M$. 
In fact, Meeks \cite{Me} showed that the Euler characteristic 
$\chi (\overline{M}_\gamma)$ and $2\deg(g)$ are congruent modulo $4$, 
where $g$ is the Gauss map of $f$. 
By these facts, we can observe that for every complete non-orientable 
minimal surface of finite total curvature, 
$\deg(g)\geq \gamma +3$ holds. 

For $\gamma =0$ and $\gamma =1$, Meeks' M\"obius strip \cite{Me} and 
L\'opez' Klein bottle \cite{Lo} satisfy $\deg(g)=\gamma +3$, respectively. 
But, for $\gamma \geq 2$, no examples with $\deg(g)=\gamma +3$ are known. 
So, it is interesting to give a minimal surface satisfying 
$\deg(g)= \gamma +3$ with an antiholomorphic involution without fixed points. 
This problem appeared in \cite{LM} and \cite{Ma}.

\begin{acknowledgements}
The authors thank Professor Reiko Miyaoka for a fruitful discussion 
when the authors attended her informal lectures on minimal surfaces 
at Kyushu University.  
The authors also thank Professors Shin Kato, 
Francisco J. L\'{o}pez, Wayne Rossman, Masaaki Umehara, Matthias Weber, 
and Kotaro Yamada for their valuable comments.  
They would further like to thank the referee for comments that 
significantly improved the results here. 
\end{acknowledgements}

\bigskip
%%%%%%%%%%%% Authors addresses %%%%%%%%%%%%%
{\small
\noindent
Shoichi Fujimori \\
Department of Mathematics \\ 
Okayama University \\ 
Okayama 700-8530, Japan.  \\
{\itshape E-mail address}\/: fujimori@math.okayama-u.ac.jp
\par\vskip4ex
\noindent
Toshihiro Shoda \\
Faculty of Culture and Education \\
Saga University \\
1 Honjo-machi, Saga-city, Saga 840-8502, Japan. \\
{\itshape E-mail address}\/: tshoda@cc.saga-u.ac.jp
}
%
%topmargin=\the\topmargin \\
%headheight=\the\headheight \\
%headsep=\the\headsep \\
%textheight=\the\textheight \\
%oddsidemargin=\the\oddsidemargin \\
%evensidemargin=\the\evensidemargin \\
%textwidth=\the\textwidth \\
\end{document}